\documentclass[11pt]{article} 
\usepackage{amsfonts,amsmath,latexsym,amssymb,mathrsfs,amsthm,comment,setspace}
\usepackage{slashbox}
\usepackage{caption}
\usepackage{enumitem}
\usepackage{algorithmic, float, graphicx}
\usepackage{comment}

\let\OLDthebibliography\thebibliography
\renewcommand\thebibliography[1]{
  \OLDthebibliography{#1}
  \setlength{\parskip}{0pt}
  \setlength{\itemsep}{0pt plus 0.0ex}
}

\newcommand{\bea}{\begin{eqnarray}}
\newcommand{\ena}{\end{eqnarray}}
\newcommand{\beq}{\begin{equation}}
\newcommand{\enq}{\end{equation}}
\newcommand{\beas}{\begin{eqnarray*}}
\newcommand{\enas}{\end{eqnarray*}}

\newcommand{\PP}{{\mathbb{P}} }
\newcommand{\E}{{\mathbb{E}} }

\newcommand{\EE}{{\mathbb{E}} }

\def\numberlikeadb{\global\def\theequation{\thesection.\arabic{equation}}}
\numberlikeadb
\newtheorem{theorem}{Theorem}[section]
\newtheorem{lemma}[theorem]{Lemma}
\newtheorem{corollary}[theorem]{Corollary}

\newtheorem{proposition}[theorem]{Proposition}
\newtheorem{remark}[theorem]{Remark}

\usepackage{color} 
\newcommand{\gr}[1]{{\color{black} #1}}
\newcommand{\aaa}[1]{{\color{black} #1}}
\newcommand{\gdr}[1]{{\color{black} #1}}

\usepackage{lscape}
\usepackage{caption}
\usepackage{multirow}
\begin{document}

\title{Wasserstein distance error bounds for the multivariate normal approximation of the maximum likelihood estimator}
\author{Andreas Anastasiou\footnote{Department of Mathematics and Statistics, University of Cyprus, P.O. Box: 20537, 1678, Nicosia, Cyprus, 
anastasiou.andreas@ucy.ac.cy}\:\: and Robert E. Gaunt\footnote{Department of Mathematics, The University of Manchester, Oxford Road, Manchester M13 9PL, UK, robert.gaunt@manchester.ac.uk}}

\date{} 
\maketitle

\vspace{-10mm}


\begin{abstract}
We obtain explicit $p$-Wasserstein distance error bounds between the distribution of the multi-parameter MLE and the multivariate normal distribution.  Our general bounds are given for possibly high-dimensional, independent and identically distributed random vectors.  Our general bounds are of the optimal $\mathcal{O}(n^{-1/2})$ order. Explicit numerical constants are given when $p\in(1,2]$, and in the case $p>2$ the bounds are explicit up to a constant factor that only depends on $p$. We apply our general bounds to derive Wasserstein distance error bounds for the multivariate normal approximation of the MLE in several settings; these being single-parameter exponential families, the normal distribution under canonical parametrisation, and the multivariate normal distribution under non-canonical parametrisation. In addition, we provide upper bounds with respect to the bounded Wasserstein distance when the MLE is implicitly defined.
\end{abstract}

\noindent{{\bf{Keywords:}}} Maximum likelihood estimation;
multivariate normal approximation; normal approximation; Wasserstein distance; Stein's method 

\noindent{{{\bf{AMS 2010 Subject Classification:}}} Primary 60F05; 62E17; 62F10; 62F12 

\section{Introduction}

The asymptotic normality of the maximum likelihood estimator (MLE), under regularity conditions, is one of the most fundamental and well-known results in statistical theory.  However, \gr{progress has only been made very recently} on the problem of deriving error bounds for the distance between the distribution of the MLE, under general regularity conditions, and its limiting normal distribution.
This is in part due to the fact that the MLE is in general a nonlinear statistic for which classical techniques for distributional approximation, such as Stein's method \cite{stein}, are difficult to apply directly, although, amongst other works, \cite{cs07} and \cite{pm16} have obtained optimal order Berry-Esseen-type bounds for quite broad classes of nonlinear statistics.

In recent years, however, there have been a number of contributions to the problem of quantifying the closeness of the MLE to its asymptotic normal distribution.  
Under general regularity conditions, \cite{ar15} used Stein's method to obtain an explicit $\mathcal{O}(n^{-1/2})$ bound, where $n$ is the sample size, between the distribution of the single-parameter MLE and the normal distribution in the bounded Wasserstein metric (this and all other probability metrics mentioned in this paper will be defined in Section \ref{sec2.2met}).  In the special case that the MLE can be expressed as a suitably smooth function of a sum of independent and identically distributed (i.i.d.) observations, \cite{al15} obtained bounds that sharpen and simplify those of  \cite{ar15}.   The results of \cite{ar15} were extended by \cite{a18} to quantify the closeness between the multi-parameter MLE and its limiting multivariate normal distribution.  However, the added technical difficulties of multivariate normal approximation by Stein's method meant that these bounds were given in a smooth test function metric (we also define this metric in Section \ref{sec2.2met}) that is weaker than the bounded Wasserstein metric.  Under the requirement that the statistic of interest can be expressed as a sum of independent random elements, \cite{pm16} used the delta method to establish uniform and non-uniform Kolmogorov distance bounds on the rate of convergence to normality for various statistics, including the single-parameter MLE.  The bounds obtained were of the optimal $\mathcal{O}(n^{-1/2})$ order.  The recent paper \cite{p17} subsequently extended the results of \cite{pm16} to cover general regularity conditions and settings in which the MLE is not necessarily a function of the sum of independent random terms.  The nonuniform bounds of \cite{p17} are the only such bounds in the literature for the normal approximation of the MLE. 

In this paper, we obtain, under general regularity conditions, optimal order $\mathcal{O}(n^{-1/2})$ bounds on the distance between the distribution of the multi-parameter MLE and its limiting multivariate normal distribution, with respect to the \gr{$p$-}Wasserstein metric. \gr{A general 1-Wasserstein distance} bound appears in
Theorem \ref{Theorem_multi}, and a simpler bound for the single-parameter MLE is given in Theorem \ref{Theoremnoncan}. \gr{We provide $p$-Wasserstein distance analogues of these bounds in Theorem \ref{thmwasp}.} These results are a technical advancement over the works of \cite{ar15} and \cite{a18}, because the \gr{$1$-}Wasserstein metric is a strictly stronger metric than those used in these works\gr{, and the $p$-Wasserstein metric ($p\geq1$) is a stronger metric still (provided it is well-defined for the probability distributions under consideration).} \gr{Moreover, Wasserstein distances are} natural and widely used probability metrics that have many applications in statistics (see \cite{pz19}).  Our bounds also remove an additional constant $\epsilon$ that appears in the bounds of \cite{ar15} and \cite{a18}, 
 and further comparisons between our bounds are given in Remark \ref{remcomparison}. In obtaining our bounds, we use Stein's method and in particular make use of the very recent advances in the literature on optimal (or near-optimal) order Wasserstein distance bounds for the multivariate normal approximation of sums of independent random vectors; see the recent works \cite{bonis,cfp17,fang18,f18b,gms,raic,zhai} for important contributions to this body of research.  Our results to some extent complement this literature by giving optimal order Wasserstein distance bounds for multivariate normal approximation in the much more general setting of the MLE under general regularity conditions, which is in general a nonlinear statistic.  In fact, to the best of our knowledge, this paper contains the first examples of optimal order Wasserstein distance bounds for the multivariate normal approximation of nonlinear statistics. 

The work of \cite{p17} is significant in that the bounds are given in the Kolmogorov metric, which is a technically demanding metric to work in \gr{when applying Stein's method}, and is particularly important in statistics, as bounds in this metric can be used, for example, to construct conservative confidence intervals.  It should be noted, however, that, as already mentioned, \gr{Wasserstein distances have} many applications in statistics \cite{pz19}, and, as observed by \cite{bx06}, the Wasserstein distance between probability distributions has the theoretically desirable property of taking into account not only the amounts  by which their  probabilities differ, as is the case in the Kolmogorov distance, but also where these differences take place.  For the single-parameter case, our results complement those of \cite{p17} by giving bounds in another important probability metric, and have the advantage of being explicit, whilst those of \cite{p17} are (in the case of uniform bounds) of the form $Cn^{-1/2}$, where $C$ is an unspecified constant that does not involve $n$.  For the multi-parameter MLE, one can extract explicit sub-optimal order $\mathcal{O}(n^{-1/4})$ Kolmogorov distance bounds for the multivariate normal approximation from our \gr{1-}Wasserstein distance bounds (see inequality (\ref{dkdwnor})). It should be noted that a similar procedure can be used to extract Kolmogorov distance bounds from those of \cite{a18}, although, as a consequence of the weaker metric used in that work, these are of the worse order $\mathcal{O}\big(n^{-1/8}\big)$ (see Remark \ref{2020rem}).  For the time being, to the best of our knowledge, the $\mathcal{O}(n^{-1/4})$ Kolmogorov distance bounds for the multi-parameter MLE that can be deduced from our Wasserstein distance bounds have the best dependence on $n$ in the current literature.




The rest of the paper is organised as follows.  In Section \ref{sec2}, we present the setting of the paper.  This includes the notation, regularity conditions for our main results, definitions of the probability metrics used in the paper and a relationship between the \gr{1-}Wasserstein and Kolmogorov metrics, and we also recall some results from the literature on Stein's method for normal and multivariate normal approximation.  In Section \ref{sec3}, we state and prove our \gr{main results. Theorem \ref{Theorem_multi} provides} an optimal order Wasserstein distance bound on the closeness between the distribution of the multi-parameter MLE and its limiting multivariate normal distribution.  We also present a simplified bound in the univariate case \gr{(Theorem \ref{Theoremnoncan}). Theorem \ref{thmwasp} provides $p$-Wasserstein distance analogues of the bounds of Theorems \ref{Theorem_multi} and \ref{Theoremnoncan}.} 
 In Section \ref{sec4}, we apply the results of Section \ref{sec3} in the settings of single-parameter exponential families, the normal distribution under canonical parametrisation, and the multivariate normal distribution under non-canonical parametrisation. \aaa{In addition, we provide upper bounds for cases where the MLE cannot be expressed analytically \gdr{with respect to the bounded Wasserstein distance}.}  In Section \ref{sec4.56}, we carry out a  simulation study to assess the accuracy of our bounds.   Some technical proofs, examples, and calculations are postponed to Appendix \ref{appa}.



\section{Setting}\label{sec2}


\subsection{Regularity conditions}

The notation that is used throughout the paper is as follows. The parameter space is $\Theta \subset \mathbb{R}^d$ equipped with the Euclidean norm. 
 Let $\boldsymbol{\theta} = (\theta_1,\theta_2, \ldots, \theta_d)^{\intercal}$ denote a parameter from the parameter space, while $\boldsymbol{\theta}_0 = \left(\theta_{0,1}, \theta_{0,2}, \ldots, \theta_{0,d}\right)^{\intercal}$ denotes the true, but unknown, value of the parameter. 
For $\boldsymbol{X}=(\boldsymbol{X}_1, \boldsymbol{X}_2, \ldots, \boldsymbol{X}_n)$ being i.i.d$.$ random vectors in $\mathbb{R}^t$, $t\in\mathbb{Z}^+$, we denote by $f(\boldsymbol{x}_i|\boldsymbol{\theta})$ the probability density (or mass) function of $\boldsymbol{X}_i$. The likelihood function is $L(\boldsymbol{\theta}; \boldsymbol{x}) = \prod_{i=1}^{n}f(\boldsymbol{x}_i|\boldsymbol{\theta})$, where $\boldsymbol{x} = (\boldsymbol{x}_1, \boldsymbol{x}_2, \ldots, \boldsymbol{x}_n)$.
Its natural logarithm, called the log-likelihood function, is denoted by $\ell(\boldsymbol{\theta};\boldsymbol{x})=\log L(\boldsymbol{\theta}; \boldsymbol{x})$. We shall write $\nabla=\big(\frac{\partial}{\partial \theta_1},\ldots,\frac{\partial}{\partial \theta_d}\big)^\intercal$ to denote the gradient operator with respect to the unknown parameter vector $\boldsymbol{\theta}$. A maximum likelihood estimate (not seen as a random vector) is a value in the parameter space which maximises the likelihood function. For many models, the MLE as a random vector exists and is also unique, in which case it is denoted by $\boldsymbol{\hat{\theta}}_n(\boldsymbol{X})$, the MLE for $\boldsymbol{\theta}_0$ based on the sample $\boldsymbol{X}$. A set of assumptions that ensure existence and uniqueness of the MLE are given in \cite{Makelainen}. This is known as the `regular' case. However, existence and uniqueness of the MLE cannot be taken for granted; see  \cite{Billingsley} for an example of non-uniqueness. We shall write $\mathbb{E}$ to denote the expectation with respect to $\boldsymbol{\theta}_0$, and $\mathbb{E}_{\boldsymbol{\theta}}$ to denote the expectation with respect to $\boldsymbol{\theta}$.


Let us now present standard regularity conditions under which asymptotic normality of the MLE holds \cite{Davison}:
\begin{itemize}[leftmargin=0.55in]
\item[(R.C.1)] The densities defined by any two different values of $\boldsymbol{\theta}$ are distinct.
\item[(R.C.2)] For all $\boldsymbol{\theta} \in \Theta$, $\EE_{\boldsymbol{\theta}}\left[\nabla\left(\ell\left(\boldsymbol{\theta};\boldsymbol{X}\right)\right)\right] = \boldsymbol{0}$.
\item[(R.C.3)] The expected Fisher information matrix for a single random vector $I(\boldsymbol{\theta})$ is finite and positive definite. For $r,s\in\{1,2,\ldots,d\}$, its elements satisfy
\begin{equation}
\nonumber n[I(\boldsymbol{\theta})]_{rs} = \EE_{\boldsymbol{\theta}}\left[ \frac{\partial}{\partial \theta_r} \ell(\boldsymbol{\theta};\boldsymbol{X})\frac{\partial}{\partial \theta_s}\ell(\boldsymbol{\theta};\boldsymbol{X})\right] = \EE_{\boldsymbol{\theta}}\left[ -\frac{\partial^2}{\partial \theta_r \partial \theta_s} \ell(\boldsymbol{\theta};\boldsymbol{X}) \right].
\end{equation}
This condition implies that $nI(\boldsymbol{\theta})$ is the covariance matrix of $\nabla(\ell(\boldsymbol{\theta};\boldsymbol{\gr{X}}))$.
\item[(R.C.4)] For any $\boldsymbol{\theta}_0 \in \boldsymbol{\Theta}$ and for $\boldsymbol{\mathbb{X}}$ denoting the support of the data, there exists $\epsilon_0 > 0$ and functions $M_{rst}(\boldsymbol{x})$ (they can depend on $\boldsymbol{\theta}_0$), such that for $\boldsymbol{\theta} = (\theta_1, \theta_2, \ldots, \theta_d)$ and $r, s, t, \in\{1,2,\ldots,d\},$ the third order partial derivatives $\frac{\partial^3}{\partial \theta_r \partial \theta_s \partial \theta_t}\ell(\boldsymbol{\theta};\boldsymbol{x})$ exist almost surely in the neighbourhood $|\theta_j - \theta_{0,j}| < \epsilon_0$, $j=1,2,\ldots,d$, and satisfy 
\begin{equation}
\nonumber \left|\frac{\partial^3}{\partial \theta_r \partial \theta_s \partial \theta_t}\ell(\boldsymbol{\theta};\boldsymbol{x})\right| \leq M_{rst}(\boldsymbol{x}), \; \forall \boldsymbol{x} \in \boldsymbol{\mathbb{X}},\; \left|\theta_j - \theta_{0,j}\right| < \epsilon_0,\; j=1,2,\ldots,d,
\end{equation}
with $\EE[M_{rst}(\boldsymbol{X})] < \infty$. 
\end{itemize}
In addition to these regularity conditions, \cite{Davison} assumes that the true value $\boldsymbol{\theta}_0$ of $\boldsymbol{\theta}$ is interior to the parameter space $\Theta\subset \mathbb{R}^d$, which is compact. Throughout this paper, we shall instead assume that the parameter space $\Theta\subset \mathbb{R}^d$ is open.
 Conditions (R.C.1), (R.C.3) and (R.C.4) are stated explicitly on page 118 of \cite{Davison}.  We have expressed (R.C.4) slightly differently to how it is stated in \cite{Davison}, so that our presentation is consistent with that from the book \cite{cb02} and a similar regularity condition  (R.C.4') of \cite{a18}, which are both referred to in our paper. Condition (R.C.2) is not stated on page 118 of \cite{Davison}, but is crucial to the proof and is implied by equation (4.32) on page 124 of \cite{Davison} in which an interchange in the order of integration and differentiation is assumed. 
 
 The asymptotic normality of the MLE was first discussed by \cite{f25}. Here, with the above regularity conditions, we present the following statement of the asymptotic normality of the multi-parameter MLE for i.i.d$.$ random vectors; for the independent but not necessarily identically distributed case see \cite{Hoadley}.

\begin{theorem}[Davison \cite{Davison}]\label{Theorem_asymptotic_MLE_i.n.i.d}
Let $\boldsymbol{X}_1, \boldsymbol{X}_2, \ldots, \boldsymbol{X}_n$ be i.i.d$.$ random vectors with probability density (or mass) functions $f(\boldsymbol{x}_i|\boldsymbol{\theta})$, where $\boldsymbol{\theta} \in \Theta \subset \mathbb{R}^d$, and $\Theta$ is compact. Assume that the MLE $\hat{\boldsymbol{\theta}}_n(\boldsymbol{X})$ exists and is unique and that the regularity conditions (R.C.1)--(R.C.4) hold. Let $\boldsymbol{Z} \sim \mathrm{MVN}\left(\boldsymbol{0},I_{d}\right)$, where $\boldsymbol{0}$ is the $d \times 1$ zero vector and $I_{d}$ is the $d \times d$ identity matrix.  Then
\begin{equation*}
\sqrt{n}\left[I(\boldsymbol{\theta}_0)\right]^{1/2}\big(\hat{\boldsymbol{\theta}}_n(\boldsymbol{X}) - \boldsymbol{\theta}_0\big) \xrightarrow[{n \to \infty}]{{\rm d}} \boldsymbol{Z}.
\end{equation*}
\end{theorem}

A quantitative version of Theorem \ref{Theorem_asymptotic_MLE_i.n.i.d} was obtained by \cite{a18} 
(in the i.i.d$.$ setting) 
under slightly stronger regularity conditions, these being (R.C.1)--(R.C.3) and the following condition (R.C.4').  
Before presenting this condition, we introduce some notation.  Let the subscript $(m)$ denote an index for which the quantity $|\hat{\theta}_n(\boldsymbol{x})_{(m)} - \theta_{0,(m)}|$ is the largest among the $d$ components:
\begin{align}
\nonumber & (m) \in \left\lbrace 1,\ldots,d\right\rbrace\;{\rm is\;such\;that\;} |\hat{\theta}_n(\boldsymbol{x})_{(m)} - \theta_{0,(m)}| \geq |\hat{\theta}_n(\boldsymbol{x})_j - \theta_{0,j}|,\: \forall j \in \left\lbrace 1,\ldots, d \right\rbrace.
\end{align}
Let
\begin{equation}\label{cm9}Q_{(m)}=Q_{(m)}(\boldsymbol{X},\boldsymbol{\theta}_0) := \hat{\theta}_n(\boldsymbol{X})_{(m)} - \theta_{0,(m)}.
\end{equation}

\begin{itemize} [leftmargin=0.55in]
\item [(R.C.4')] The log-likelihood $\ell(\boldsymbol{\theta};\boldsymbol{x})$ is three times differentiable with respect to the unknown vector parameter $\boldsymbol{\theta}$ and the third order partial derivatives are continuous in $\boldsymbol{\theta}$. In addition, for any $\boldsymbol{\theta}_0 \in \Theta$ there exists $0<\epsilon=\epsilon(\boldsymbol{\theta}_0)$ and functions $M_{kjl}(\boldsymbol{x}),\;\forall k,j,l \in \left\lbrace 1,2,\ldots,d \right\rbrace$, such that $\big|\frac{\partial^3}{\partial \theta_{k}\partial \theta_{j}\partial \theta_{l}}\ell(\boldsymbol{\theta},\boldsymbol{x})\big| \leq M_{kjl}(\boldsymbol{x})$ for all $\boldsymbol{\theta} \in \Theta$ with $|\theta_j - \theta_{0,j}| < \epsilon$, $\forall j \in \left\lbrace 1,2,\ldots,d\right\rbrace$. Also, for $Q_{(m)}$ as in \eqref{cm9}, assume that $\EE[\left(M_{kjl}(\boldsymbol{X})\right)^2\,|\,|Q_{(m)}| < \epsilon ] < \infty$.
\end{itemize}

\gr{In Theorems \ref{Theorem_multi} and \ref{thmwasp}}, we shall work with the same regularity conditions as \cite{a18}, but with (R.C.4') replaced by the following condition \gr{(R.C.4''($p$))}. 
Before stating condition \gr{(R.C.4''($p$))}, we introduce some terminology.  We say that $M(\boldsymbol{\theta}; \boldsymbol{x})$ is monotonic in the multivariate context if for all fixed $\tilde{\theta}_1, \tilde{\theta}_2, \ldots, \tilde{\theta}_d$ and $\boldsymbol{x}$ we have that, for each $s \in \left\lbrace 1,2,\ldots, d\right\rbrace$,
\begin{equation}
\label{monotonicity}
\theta_s \rightarrow M(\tilde{\theta}_1, \tilde{\theta}_2, \ldots, \tilde{\theta}_{s-1}, \theta_s, \tilde{\theta}_{s+1}, \ldots, \tilde{\theta}_d;\boldsymbol{x})
\end{equation}
is a monotonic function.

\begin{itemize}[leftmargin=0.82in] 
\item [\gr{(R.C.4''($p$))}] All third order partial derivatives of the log-likelihood $\ell(\boldsymbol{\theta};\boldsymbol{x})$ with respect to the unknown vector parameter $\boldsymbol{\theta}$ exist. Also, for any $\boldsymbol{\theta} \in \boldsymbol{\Theta}$ and for $\boldsymbol{\mathbb{X}}$ denoting the support of the data, we assume that for any $j, l, q \in \left\lbrace 1,2,\ldots, d\right\rbrace$ there exists a function $M_{qlj}(\boldsymbol{\theta};\boldsymbol{x})$, which is monotonic in the sense defined in (\ref{monotonicity}), such that
\begin{equation}
\nonumber \left|\frac{\partial^3}{\partial\theta_q\partial\theta_l\partial\theta_j}\ell(\boldsymbol{\theta};\boldsymbol{x})\right| \leq M_{qlj}(\boldsymbol{\theta};\boldsymbol{x}), \quad \forall \boldsymbol{x} \in \boldsymbol{\mathbb{X}},
\end{equation}
and
\begin{equation}\label{intconrc}
\mathrm{max}_{\substack{\tilde{\theta}_m \in \left\lbrace\hat{\theta}_n(\boldsymbol{X})_m, \theta_{0,m}\right\rbrace\\m \in \left\lbrace 1,2,\ldots, d\right\rbrace}}\E\big[\big|(\hat{\theta}_n(\boldsymbol{X})_{l} - \theta_{0,l})(\hat{\theta}_n(\boldsymbol{X})_{q} - \theta_{0,q})M_{qlj}(\tilde{\boldsymbol{\theta}}; \boldsymbol{X})\big|^p\big] < \infty.
\end{equation}
In the univariate $d=1$ case we drop the subscripts and write $M(\boldsymbol{\theta};\boldsymbol{x})$.
\end{itemize}

\gr{We include reference to the variable $p$ in the name of our condition (R.C.4''($p$)) to emphasis the fact that the integrability condition (\ref{intconrc}) depends on $p$, the order of the Wasserstein distance under consideration.  In the case $p=1$, corresponding to the classical $1$-Wasserstein distance, we shall simply write (R.C.4'').} 

\begin{remark}\gr{For brevity, in this remark we discuss the condition (R.C.4''); similar comments apply to the more general condition (R.C.4''($p$)).} 
The motivation for introducing (R.C.4'') is that in the proof of Theorem \ref{Theorem_multi} it allows us to bound one of the remainder terms in the \gr{$1$-}Wasserstein metric, which would not be possible if instead working with (R.C.4) or (R.C.4').  Conditions (R.C.4), (R.C.4') and (R.C.4'') each require all third order partial derivatives of $\ell(\boldsymbol{\theta};\boldsymbol{x})$ to exist.  Each condition then also involves an integrability condition involving a function that dominates the absolute value of these partial derivatives in a certain way. For a given MLE, verifying the integrability conditions in (R.C.4') and (R.C.4'') each have extra difficulty compared to (R.C.4): (R.C.4') involves a conditional expectation, whilst for (R.C.4'') the expectations in  (\ref{intconrc}) involve the MLE. In Section \ref{sec4}, we give some examples in which the MLE takes a relatively simple form, for which the verification of (R.C.4'') follows from elementary calculations, and is simpler to work with than the integrability condition involving conditional expectations in (R.C.4').  For complicated MLEs it inevitably becomes more involved to verify (R.C.4''). In Appendix \ref{appig}, we give an illustration of how (R.C.4'') can be verified for more \gr{complicated} MLEs using the example of the inverse gamma distribution. A comparison between (R.C.4') and (R.C.4'') in the context of obtaining error bounds for the distance between the distribution of the MLE and the multivariate normal distribution is given in Remark \ref{remcomparison}.
\end{remark}



In the case of univariate i.i.d$.$ random variables we work with (R.C.4'') and the following simpler regularity conditions:
\begin{itemize}
\item[\gr{(R1)}] The densities defined by any two different values of $\theta$ are distinct.
\item[\gr{(R2)}] 
The density $f(x|\theta)$ is three times differentiable with respect to $\theta$, the third derivative is continuous in $\theta$, and $\int f(x|\theta)\,\mathrm{d}x$ can be differentiated three times under the integral sign.
\item[\gr{(R3)}]  $i(\theta_0) \neq 0$, where $i(\theta)$ is the expected Fisher information for one random variable.
\end{itemize}

These regularity conditions are the same as those used in \cite{cb02} and \cite{ar15} with the exception that (R.C.4'') is replaced by a univariate version of (R.C.4) and (R.C.4'), respectively.



\subsection{Probability metrics}\label{sec2.2met}

\gr{Let $\boldsymbol{X}$ and $\boldsymbol{Y}$ be $\mathbb{R}^d$-valued random vectors. Fix $p\geq1$ and suppose that $\mathbb{E}[|\boldsymbol{X}|^p]<\infty$ and $\mathbb{E}[|\boldsymbol{Y}|^p]<\infty$, where $|\cdot|$ denotes the usual Euclidean norm. Then the $p$-Wasserstein distance between the distributions of $\boldsymbol{X}$ and $\boldsymbol{Y}$ is defined by
\begin{equation}\label{pwasdefn}d_{\mathrm{W}_p}(\boldsymbol{X},\boldsymbol{Y})=\big(\inf\mathbb{E}[|\boldsymbol{X}'-\boldsymbol{Y}'|^p]\big)^{1/p},
\end{equation}
where the infimum is taken over all joint distributions of $\boldsymbol{X}'$ and $\boldsymbol{Y}'$ that have the same law as $\boldsymbol{X}$ and $\boldsymbol{Y}$, respectively.  In the case $p=1$, corresponding to the 1-Wasserstein distance, we shall drop the subscript 1 and write $d_{\mathrm{W}}$. The infimum in (\ref{pwasdefn}) is actually a minimum in that there exists a pair of jointly distributed random variables $(\boldsymbol{X}^*,\boldsymbol{Y}^*)$ with $\mathcal{L}(\boldsymbol{X}^*)=\mathcal{L}(\boldsymbol{X})$ and $\mathcal{L}(\boldsymbol{Y}^*)=\mathcal{L}(\boldsymbol{Y})$ such that
\[d_{\mathrm{W}_p}(\boldsymbol{X},\boldsymbol{Y})=\big(\mathbb{E}[|\boldsymbol{X}^*-\boldsymbol{Y}^*|^p]\big)^{1/p}\]
(see Chapter 6 of \cite{v09} and Lemma 1 of \cite{mr18}).   By H\"older's inequality, it follows that, if $1\leq p<q$, then \begin{equation}\label{pqpq}d_{\mathrm{W}_p}(\boldsymbol{X},\boldsymbol{Y})\leq d_{\mathrm{W}_q}(\boldsymbol{X},\boldsymbol{Y})\end{equation}
 for all $\boldsymbol{X}$ and $\boldsymbol{Y}$  such that $\mathbb{E}[|\boldsymbol{X}|^q]<\infty$ and $\mathbb{E}[|\boldsymbol{Y}|^q]<\infty$ (see again Chapter 6 of \cite{v09} and Lemma 1 of \cite{mr18}).}

\gr{The 1-Wasserstein metric and several other} probability metrics used in this paper can be conveniently expressed as integral probability metrics.  For $\mathbb{R}^d$-valued random vectors $\boldsymbol{X}$ and $\boldsymbol{Y}$, integral probability metrics are of the form 
\begin{equation}\label{dhdef}d_{\mathcal{H}}(\boldsymbol{X},\boldsymbol{Y}):=\sup_{h\in\mathcal{H}}|\mathbb{E}[h(\boldsymbol{X})]-\mathbb{E}[h(\boldsymbol{Y})]|
\end{equation}
for some class of functions $\mathcal{H}$. At this stage, we introduce some notation.  
For vectors $\boldsymbol{a}=(a_1,\ldots,a_d)\in\mathbb{R}^d$ and $\boldsymbol{b}=(b_1,\ldots,b_d)\in\mathbb{R}^d$, we write $\boldsymbol{a}\leq\boldsymbol{b}$ provided $a_i\leq b_i$ for $i=1,\ldots,d$. 
For a three times differentiable function $h:\mathbb{R}^d\rightarrow\mathbb{R}$ (denoted by $h\in C_b^3(\mathbb{R}^d)$), we abbreviate $|h|_1 :=\mathrm{max}_i\big\|\frac{\partial}{\partial x_i}h\big\|$, $|h|_2 := \mathrm{max}_{i,j}\big\|\frac{\partial^2}{\partial x_i\partial x_j}h\big\|$ and $|h|_3 := \mathrm{max}_{i,j,k}\big\|\frac{\partial^3}{\partial x_i\partial x_j\partial x_k}h\big\|$, provided these quantities are finite. Here (and elsewhere) $\|\cdot\|:=\|\cdot\|_\infty$ denotes the usual supremum norm of a real-valued function.  For a Lipschitz function $h:\mathbb{R}^d\rightarrow\mathbb{R}$ we denote
\begin{equation*}\|h\|_{\mathrm{Lip}}=\sup_{\boldsymbol{x}\not=\boldsymbol{y}}\frac{|h(\boldsymbol{x})-h(\boldsymbol{y})|}{|\boldsymbol{x}-\boldsymbol{y}|}.
\end{equation*}
With this notation in place, taking 
\begin{align}
\label{class_functions}
\mathcal{H}_{\mathrm{K}}&=\{\mathbf{1}(\cdot\leq \boldsymbol{z})\,|\,\boldsymbol{z}\in\mathbb{R}^d\}, \\
\mathcal{H}_{\mathrm{W}}&=\{h:\mathbb{R}^d\rightarrow\mathbb{R}\,|\,\text{$h$ is Lipschitz, $\|h\|_{\mathrm{Lip}}\leq1$}\}, \\
\mathcal{H}_{\mathrm{bW}}&=\{h:\gdr{\mathbb{R}^d}\rightarrow\mathbb{R}\,|\,\text{$h$ is Lipschitz, $\|h\|\leq1$ and $\|h\|_{\mathrm{Lip}}\leq1$}\}, \\
\mathcal{H}_{1,2}&=\{h:\mathbb{R}^d\rightarrow\mathbb{R}\,|\,\text{$h\in C^2(\mathbb{R}^d)$ with $|h|_j\leq1$, $j=1,2$}\}, \\
\mathcal{H}_{0,1,2,3}&=\{h:\mathbb{R}^d\rightarrow\mathbb{R}\,|\,\text{$h\in C^3(\mathbb{R}^d)$ with $\|h\|\leq1$ and $|h|_j\leq1$, $j=1,2,3$}\}
\end{align}
in (\ref{dhdef}) gives the Kolmogorov, \gr{1-}Wasserstein and bounded Wasserstein distances, which we denote by $d_{\mathrm{K}}$, $d_{\mathrm{W}}$ and $d_{\mathrm{bW}}$, respectively, as well as smooth test function metrics, which we denote by $d_{1,2}$ and $d_{0,1,2,3}$.  
In all the above notation, we supress the dependence on the dimension $d$.  Of the works mentioned in the Introduction, the results of \cite{p17} are given in the Kolmogorov metric, \cite{al15} and \cite{ar15} work in the bounded Wasserstein metric, and \cite{a18} works in the smooth test function $d_{0,1,2,3}$ metric.  It is evident that $d_{\mathrm{bW}}$ and $d_{0,1,2,3}$ are weaker than the $d_{\mathrm{W}}$ metric.

We now note the following important relations between the Kolmogorov metric and the \gr{1-}Wasserstein and bounded Wasserstein metrics, respectively. Let $Y$ be any real-valued random variable and $Z\sim {\rm N}(0,1)$.  Then by \cite[Proposition 1.2]{ross} (see also \cite[Theorem 3.3]{chen}) and \cite[Proposition 2.4]{pike}, we have that 
\begin{align}\label{2020a}d_{\mathrm{K}}(Y,Z)&\leq \bigg(\frac{2}{\pi}\bigg)^{1/4}\sqrt{d_{\mathrm{W}}(Y,Z)},\\
\label{2020b} d_{\mathrm{K}}(Y,Z)&\leq \bigg(1+\frac{1}{2\sqrt{2\pi}}\bigg)\sqrt{d_{\mathrm{bW}}(Y,Z)}.
\end{align}
These bounds in terms of $d_{\mathrm{W}}(Y,Z)$ and $d_{\mathrm{bW}}(Y,Z)$, respectively, are best possible up to a constant factor \cite[p$.$ 1026]{pm16}.  Hence, our forthcoming $\mathcal{O}(n^{-1/2})$ \gr{1-}Wasserstein distance bounds for the asymptotic normality of the single-parameter MLE and $\mathcal{O}(n^{-1/2})$ bounded Wasserstein distance bounds both yield $\mathcal{O}(n^{-1/4})$ Kolmogorov distance bounds via (\ref{2020a}) and (\ref{2020b}), respectively. \gr{As $d_{\mathrm{W}}\leq d_{\mathrm{W}_p}$ for $p>1$,  bounds given with respect to the $p$-Wasserstein distance similarly imply such bounds.}

For the multi-parameter case, the following generalisation of (\ref{2020a}) due to \cite{k19} is available. Let $\boldsymbol{Z}\sim \mathrm{MVN}(\mathbf{0},I_d)$, $d\geq1$.  Then, for any $\mathbb{R}^d$-valued random vector $\boldsymbol{Y}$,
\begin{equation}\label{dkdwnor}d_{\mathrm{K}}(\boldsymbol{Y},\boldsymbol{Z})\leq \sqrt{2(\sqrt{2\log d} +2)}\sqrt{d_{\mathrm{W}}(\boldsymbol{Y},\boldsymbol{Z})}.
\end{equation}
A similar bound with the slightly bigger multiplicative constant of $3(\log(d+1))^{1/4}$ had previously been obtained by \cite{app16}. For an analogous relationship between the \gr{1-}Wasserstein and convex distances in $\mathbb{R}^d$ see \cite{npy20}. 

\begin{remark}\label{2020rem}In the univariate case, the same argument to that used in the proof of Corollary 4.2 of \cite{gpr17} can be used to show that there exists a universal constant $C$ (which can be found explicitly) such that $d_{\mathrm{K}}(Y,Z)\leq C\big(d_{0,1,2,3}(Y,Z)\big)^{1/4}$. Using the approach of \cite{app16} with a multivariate analogue of the smoothing function of \cite{gpr17} would also lead to a bound of the form $d_{\mathrm{K}}(\boldsymbol{Y},\boldsymbol{Z})\leq C\big(d_{0,1,2,3}(\boldsymbol{Y},\boldsymbol{Z})\big)^{1/4}$, for $d\geq1$. Consequently, the $\mathcal{O}(n^{-1/2})$ bounds in the $d_{0,1,2,3}$ metric of \cite{a18} for the multivariate normal approximation of the multi-parameter MLE only yield $\mathcal{O}(n^{-1/8})$ bounds in the Kolmogorov metric, whilst our $\mathcal{O}(n^{-1/2})$ \gr{1-}Wasserstein distance bounds lead to $\mathcal{O}(n^{-1/4})$ Kolmogorov distance bounds.
\end{remark}


\subsection{Wasserstein distance bounds by Stein's method}

Optimal order $\mathcal{O}(n^{-1/2})$ \gr{1-}Wasserstein distance bounds for the normal approximation of sums of independent random variables via Stein's method date as far back as \cite{e74}.  We shall make use of the following result.

\begin{theorem}[Reinert \cite{Reinert_coupling}]\label{thmapw} Let $\xi_1,\ldots,\xi_n$ be i.i.d$.$ random variables with $\mathbb{E}[\xi_1]=0$, $\mathrm{Var}(\xi_1)=1$ and $\mathbb{E}[|\xi_1|^3]<\infty$.  Denote $W=\frac{1}{\sqrt{n}}\sum_{i=1}^n\xi_i$ and let $Z\sim \mathrm{N}(0,1)$.  Then
\begin{equation*}
d_{\mathrm{W}}(W,Z)\leq\frac{1}{\sqrt{n}}\big(2+\mathbb{E}[|\xi_1|^3]\big).
\end{equation*}
\end{theorem}

Only very recently have optimal order Wasserstein distance bounds been obtained for multivariate normal approximation of independent random vectors.  There has been quite a lot of activity on this topic over the last few years, and amongst the bounds from this literature we \gr{use a bound of \cite{bonis} given in Theorem \ref{bonisthm} below}.
This is on account of the weak conditions, simplicity, and good dependence on the dimension $d$ that is sufficient for our purposes. It should be noted, however, that there are bounds in the literature that have a better dependence on the dimension $d$; see, for example, \cite{cfp17}, in which\gr{, in the case of the 2-Wasserstein distance,} the fourth moment condition of \cite{bonis} is replaced by a Poincar\'{e} inequality condition. \gr{If we were to use such a bound with improved dependence on $d$ in the derivation of our general bounds of Theorems \ref{Theorem_multi} and \ref{thmwasp}, it would, however, make no difference to the overall dependence of the bound on the dimension $d$. We also note that in the univariate case, optimal order $n^{-1/2}$ $p$-Wasserstein distance bounds have been obtained for the normal approximation of sums of independent random variables without the use of Stein's method; see \cite{r09} and references therein.}

The bound (\ref{bonisbound2}) below is not stated in \cite{bonis}, but is easily obtained from the bound (\ref{bonisbound}) (which is given in \cite{bonis}) by an application of H\"older's inequality. \gr{The bound (\ref{bonisbound200}) is also not stated in \cite{bonis}, but is again easily obtained from the bound (\ref{bonisbound00}) (which is given in \cite{bonis}) by this time applying the basic inequality $\big(\sum_{j=1}^d a_j\big)^{r}\leq d^{r-1}\sum_{j=1}^da_j^r$, where $a_1,\ldots,a_d\geq0$ and $r\geq2$.} 

For a $d\times d$ matrix $A$, let $\|A\|_F=\sqrt{\sum_{i=1}^d\sum_{j=1}^d |a_{i,j}|^2}$ be the Frobenius norm.

\begin{theorem}[Bonis \cite{bonis}]\label{bonisthm}  Let $\boldsymbol{\xi}_1,\ldots,\boldsymbol{\xi}_n$ be i.i.d$.$ random vectors in $\mathbb{R}^d$ with $\mathbb{E}[\boldsymbol{\xi}_1] = \boldsymbol{0}$ and $\mathbb{E}[\boldsymbol{\xi}_1\boldsymbol{\xi}_1^{\intercal}] =I_d$.  Let $\boldsymbol{W}=\frac{1}{\sqrt{n}}\sum_{i=1}^n\boldsymbol{\xi}_i$ and let $\boldsymbol{Z}\sim\mathrm{MVN}(\boldsymbol{0},I_d)$.  Suppose that $\mathbb{E}[|\boldsymbol{\xi}_1|^4]<\infty$.  Then
\begin{align}\label{bonisbound}d_{\gr{\mathrm{W}_2}}(\boldsymbol{W},\boldsymbol{Z})&\leq \frac{14 d^{1/4}}{\sqrt{n}}\sqrt{\|\mathbb{E}[\boldsymbol{\xi}_1\boldsymbol{\xi}_1^{\intercal}|\boldsymbol{\xi}_1|^2]\|_F}\\
\label{bonisbound2}&\leq \frac{14d^{5/4}}{\sqrt{n}}\mathrm{max}_{1\leq j \leq d}\sqrt{\mathbb{E}[\xi_{1,j}^4]},
\end{align}
where $\xi_{1,j}$ is the $j$-th component of $\boldsymbol{\xi}_1$.

\gr{Suppose now that $\mathbb{E}[|\boldsymbol{\xi}_1|^{p+2}]<\infty$ for $p\geq2$. Then there exists a constant $C_p>0$ depending only on $p$ such that
\begin{align}\label{bonisbound00}d_{\mathrm{W}_p}(\boldsymbol{W},\boldsymbol{Z})&\leq \frac{C_p}{\sqrt{n}}\Big(\sqrt{\|\mathbb{E}[\boldsymbol{\xi}_1\boldsymbol{\xi}_1^{\intercal}|\boldsymbol{\xi}_1|^2]\|_F}+\big(\mathbb{E}[|\boldsymbol{\xi}_1|^{p+2}]\big)^{1/p}\Big)\\
\label{bonisbound200}&\leq \frac{C_p}{\sqrt{n}}\Big(d^{5/4}\mathrm{max}_{1\leq j \leq d}\sqrt{\mathbb{E}[\xi_{1,j}^4]}+d^{1/2+1/p}\mathrm{max}_{1\leq j \leq d}\big(\mathbb{E}[|\xi_{1,j}|^{p+2}]\big)^{1/p}\Big).
\end{align}}
\end{theorem}

\section{Main results and proofs}\label{sec3}


For ease of presentation, let us now introduce the following notation:
\begin{equation}
\begin{aligned}
\label{cm}
\boldsymbol{W}&=\sqrt{n}[I(\boldsymbol{\theta}_0)]^{1/2}\big(\hat{\boldsymbol{\theta}}_n(\boldsymbol{X}) - \boldsymbol{\theta}_0\big),\\
 Q_{j}&=Q_{j}(\boldsymbol{X},\boldsymbol{\theta}_0) := \hat{\theta}_n(\boldsymbol{X})_{j} - \theta_{0,j}, \quad  j \in \left\lbrace 1,2,\ldots,d \right\rbrace,\\
 T_{lj} &= T_{lj}\left(\boldsymbol{\theta}_0,\boldsymbol{X}\right) = \frac{\partial^2}{\partial\theta_l\partial\theta_j}\ell(\boldsymbol{\theta}_0;\boldsymbol{X}) + n[I(\boldsymbol{\theta}_0)]_{lj}, \quad j,l \in \left\lbrace 1,2,\ldots,d \right\rbrace,\\
 \tilde{V}& = \tilde{V}(n,\boldsymbol{\theta}_0) := \left[I(\boldsymbol{\theta}_0)\right]^{-1/2},\\
 \xi_{ij} &= \sum_{k=1}^{d}\tilde{V}_{jk}\frac{\partial}{\partial \theta_k}\log(f(\boldsymbol{X}_i|\boldsymbol{\theta}_0)), \quad i \in \left\lbrace 1,2,\ldots,n \right\rbrace,\; j \in \left\lbrace 1,2,\ldots,d \right\rbrace.
\end{aligned}
\end{equation}
Notice that, using condition (R.C.3), $\EE\left[T_{lj}\right] = 0$ for all $j,l\in\{1,2,\ldots,d\}$.  

\gr{A general $1$-Wasserstein distance error bound for the multivariate normal approximation of the multi-parameter MLE is given in the following theorem.}

\begin{theorem}
\label{Theorem_multi}
Let $\boldsymbol{X}=(\boldsymbol{X}_1, \boldsymbol{X}_2, \ldots, \boldsymbol{X}_n)$ be i.i.d$.$ $\mathbb{R}^t$-valued, $t \in \mathbb{Z}^+$, random vectors with probability density (or mass) function $f(\boldsymbol{x}_i|\boldsymbol{\theta})$, for which the true parameter value is $\boldsymbol{\theta}_0$ and the parameter space $\Theta$ is an open subset of $\mathbb{R}^d$. Assume that the MLE exists and is unique and that (R.C.1)--(R.C.3), (R.C.4'') are satisfied. In addition, for $\tilde{V}$ as in \eqref{cm}, assume that $\E[|\tilde{V}\nabla\left(\log\left(f(\boldsymbol{X}_1|\boldsymbol{\theta}_0)\right)\right)|^4] < \infty$, where  $\nabla=\big(\frac{\partial}{\partial \theta_1},\ldots,\frac{\partial}{\partial \theta_d}\big)^\intercal$. Also, assume that $\mathbb{E}[Q_l^2]<\infty$ for all $l\in\{1,2,\ldots,d\}$ and $\mathbb{E}[T_{lj}^2]<\infty$ for all $j,l\in\{1,2,\ldots,d\}$.  
 Then
\begin{align}
\label{final_bound_regression}
 d_{\mathrm{W}}(\boldsymbol{W}, \boldsymbol{Z}) \leq \frac{1}{\sqrt{n}}\big(K_1(\boldsymbol{\theta}_0)+K_2(\boldsymbol{\theta}_0)+K_3(\boldsymbol{\theta}_0)\big),
\end{align}
where
\begin{align}
\label{notation_Ki}
\nonumber K_1(\boldsymbol{\theta}_0)&= 14d^{5/4}\max_{1\leq j \leq d}\sqrt{\E[\xi_{1,j}^4]},\\
\nonumber K_2(\boldsymbol{\theta}_0)&=\sum_{k=1}^{d}\sum_{j=1}^{d}|\tilde{V}_{kj}|\sum_{l=1}^{d}\sqrt{\E[Q_l^2]}\sqrt{\E[T_{lj}^2]},\\
K_3(\boldsymbol{\theta}_0)&=\frac{1}{2}\sum_{k=1}^{d}\sum_{j=1}^{d}|\tilde{V}_{kj}|\sum_{l=1}^{d}\sum_{q=1}^d\sum_{\substack{\tilde{\theta}_m \in \left\lbrace\hat{\theta}_n(\boldsymbol{X})_m, \theta_{0,m}\right\rbrace\\m \in \left\lbrace 1,2,\ldots, d\right\rbrace}}\E\big|Q_lQ_qM_{qlj}(\boldsymbol{\tilde{\theta}};\boldsymbol{X})\big|.
\end{align}
\end{theorem}


The following theorem is a simplification of Theorem \ref{Theorem_multi} for the single-parameter MLE.

\begin{theorem}
\label{Theoremnoncan}
Let $\boldsymbol{X}=(X_1, X_2, \ldots, X_n)$ be i.i.d$.$ random variables with probability density (or mass) function $f(x_i|\theta)$, for which the true parameter value is $\theta_0$ and the parameter space $\Theta$ is an open subset of $\mathbb{R}$.  Assume that the regularity conditions (R1)--(R3), (R.C.4'') are satisfied and that the MLE, $\hat{\theta}_n(\boldsymbol{X})$, exists and is unique. Assume that $\mathbb{E}\big[\left|\frac{\mathrm{d}}{\mathrm{d}\theta}{\rm log}f(X_1|\theta_0)\right|^3\big] < \infty$, ${\rm Var}\big(\frac{\mathrm{d}^2}{\mathrm{d}\theta^2} \log f(X_1|\theta_0)\big)<\infty$ and $\EE[(\hat{\theta}_n(\boldsymbol{X})- \theta_0)^2]<\infty$.
 Let $Z \sim {\rm N}(0,1)$. Then
\begin{align}
\label{boundTHEOREM} &d_{\mathrm{W}}(W,Z) \leq \frac{1}{\sqrt{n}}\bigg\{2 + \frac{1}{[i(\theta_0)]^{3/2}}\EE\bigg[\bigg|\frac{\mathrm{d}}{\mathrm{d}\theta}{\rm log}f(X_1|\theta_0)\bigg|^3\bigg]\nonumber\\
&\quad+\frac{1}{\sqrt{i(\theta_0)}}\sqrt{n{\rm Var}\left(\frac{\mathrm{d}^2}{\mathrm{d}\theta^2} \log f(X_1|\theta_0)\right)}\sqrt{\EE\big[(\hat{\theta}_n(\boldsymbol{X})- \theta_0)^2\big]}\nonumber\\
&\quad+\frac{1}{2\sqrt{i(\theta_0)}}\Big(\EE\big|(\hat{\theta}_n(\boldsymbol{X}) - \theta_0)^2M(\theta_0;\boldsymbol{X})\big|+ \EE\big|(\hat{\theta}_n(\boldsymbol{X}) - \theta_0)^2M(\hat{\theta}_n(\boldsymbol{X});\boldsymbol{X})\big|\Big)\bigg\}.
\end{align}
\end{theorem}

\gr{$p$-Wasserstein distance analogues of the bounds of the above two theorems are given in the following theorem.}

\gr{\begin{theorem}\label{thmwasp} Let $p\geq2$. Let $\boldsymbol{X}=(\boldsymbol{X}_1, \boldsymbol{X}_2, \ldots, \boldsymbol{X}_n)$ be i.i.d$.$ $\mathbb{R}^t$-valued, $t \in \mathbb{Z}^+$, random vectors with probability density (or mass) function $f(\boldsymbol{x}_i|\boldsymbol{\theta})$, for which the true parameter value is $\boldsymbol{\theta}_0$ and the parameter space $\Theta$ is an open subset of $\mathbb{R}^d$. Assume that the MLE exists and is unique and that (R.C.1)--(R.C.3), (R.C.4''($p$)) are satisfied. In addition, for $\tilde{V}$ as in \eqref{cm}, assume that $\E[|\tilde{V}\nabla\left(\log\left(f(\boldsymbol{X}_1|\boldsymbol{\theta}_0)\right)\right)|^{p+2}] < \infty$. Also, assume that $\mathbb{E}[|Q_l|^{2p}]<\infty$ for all $l\in\{1,2,\ldots,d\}$ and $\mathbb{E}[|T_{lj}|^{2p}]<\infty$ for all $j,l\in\{1,2,\ldots,d\}$.  
 Then
\begin{align}
\label{final_bound_regression0}
 d_{\mathrm{W}_p}(\boldsymbol{W}, \boldsymbol{Z}) \leq \frac{1}{\sqrt{n}}\big(K_{1,p}(\boldsymbol{\theta}_0)+K_{2,p}(\boldsymbol{\theta}_0)+K_{3,p}(\boldsymbol{\theta}_0)\big),
\end{align}
where
\begin{align}
\label{notationWp}
\nonumber K_{1,p}(\boldsymbol{\theta}_0)&= C_p\Big(d^{5/4}\max_{1\leq j \leq d}\sqrt{\E[\xi_{1,j}^4]}+d^{1/2+1/p}\big(\mathbb{E}[|\xi_{1,j}|^{p+2}]\big)^{1/p}\Big),\\
\nonumber K_{2,p}(\boldsymbol{\theta}_0)&=d^{3-3/p}\Bigg(\sum_{k=1}^{d}\sum_{j=1}^{d}|\tilde{V}_{kj}|^p\sum_{l=1}^{d}\sqrt{\E[|Q_l|^{2p}]}\sqrt{\E[|T_{lj}|^{2p}]}\Bigg)^{1/p},\\
K_{3,p}(\boldsymbol{\theta}_0)&=\frac{d^{4-4/p}}{2}\Bigg(\sum_{k=1}^{d}\sum_{j=1}^{d}|\tilde{V}_{kj}|^p\sum_{l=1}^{d}\sum_{q=1}^d\sum_{\substack{\tilde{\theta}_m \in \left\lbrace\hat{\theta}_n(\boldsymbol{X})_m, \theta_{0,m}\right\rbrace\\m \in \left\lbrace 1,2,\ldots, d\right\rbrace}}\E\big[\big|Q_lQ_qM_{qlj}(\boldsymbol{\tilde{\theta}};\boldsymbol{X})\big|^p\big]\bigg)^{1/p},
\end{align}
and $C_p>0$ is a constant depending only on $p$.

In the case $p=2$, we have the following simpler bound with an explicit constant:
\begin{align}
\label{final_bound_regression0z}
 d_{\mathrm{W}_2}(\boldsymbol{W}, \boldsymbol{Z}) \leq \frac{1}{\sqrt{n}}\big(K_1(\boldsymbol{\theta}_0)+K_{2,2}(\boldsymbol{\theta}_0)+K_{3,2}(\boldsymbol{\theta}_0)\big),
\end{align}
where $K_1(\boldsymbol{\theta}_0)$ is defined as in Theorem \ref{Theorem_multi}.
\end{theorem}}

\begin{remark}\label{rem12}
{\textbf{(1)}} Let us demonstrate that the bound (\ref{final_bound_regression}) of Theorem \ref{Theorem_multi} is of the optimal order $\mathcal{O}(n^{-1/2})$; similar considerations show that the \gr{bounds of Theorems \ref{Theoremnoncan} and \ref{thmwasp} are} $\mathcal{O}(n^{-1/2})$.  Firstly, we have that for all $j=1,2,\ldots,d$, $\mathbb{E}[\xi_{1,j}^4]=\mathcal{O}(1)$, and therefore $K_1(\boldsymbol{\theta}_0)=\mathcal{O}(1)$. Here and throughout the paper, $\mathcal{O}(1)$ is understood as
smaller than a constant which does not depend on $n$, but may depend on the dimension $d$.  
Assuming that $[I(\boldsymbol{\theta}_0)]^{-1}=\mathcal{O}(1)$, we have that $\mathbb{E}[Q_l^2]=\mathcal{O}(n^{-1})$ for all $l=1,2,\ldots,d$. To see this, note that because $\boldsymbol{W}$ is asymptotically standard multivariate normally distributed, it follows that, as $n\rightarrow\infty$,
\begin{equation*}\mathrm{Cov}(\boldsymbol{W})=[I(\boldsymbol{\theta}_0)]^{1/2}\mathrm{Cov}(\hat{\boldsymbol{\theta}}_n(\boldsymbol{X}))[I(\boldsymbol{\theta}_0)]^{1/2}\rightarrow I_{d},
\end{equation*}
and therefore $\mathrm{Cov}(\hat{\boldsymbol{\theta}}_n(\boldsymbol{X}))\rightarrow \frac{1}{n}[I(\boldsymbol{\theta}_0)]^{-1}$, as $n\rightarrow\infty$, from which we read off that $\mathbb{E}[Q_l^2]=\mathcal{O}(n^{-1})$ for all $l=1,2,\ldots,d$.
Also, using condition (R.C.3) and that $\boldsymbol{X}_1,\boldsymbol{X}_2,\ldots,\boldsymbol{X}_n$ are independent we have that
\begin{equation*}\mathbb{E}[T_{lj}^2]=\sum_{i=1}^n\mathrm{Var}\bigg(\frac{\partial^2}{\partial\theta_l\partial\theta_j}\log(f(\boldsymbol{X}_i|\boldsymbol{\theta}_0))\bigg)=\mathcal{O}(n).
\end{equation*}
Therefore $K_2(\boldsymbol{\theta}_0)=\mathcal{O}(1)$.  Since $\ell(\boldsymbol{\theta};\boldsymbol{x})=\sum_{i=1}^n\log(f(\boldsymbol{x}_i|\boldsymbol{\theta}))$, we have that
$\frac{\partial^3}{\partial\theta_q\partial\theta_l\partial\theta_j}\ell(\boldsymbol{\theta};\boldsymbol{x})=\mathcal{O}(n)$ and therefore $M_{qlj}(\boldsymbol{\theta};\boldsymbol{x})=\mathcal{O}(n)$.  As we  also have that $\mathbb{E}[Q_l^2]=\mathcal{O}(n^{-1})$ (and so $\E|Q_lQ_q|=\mathcal{O}(n^{-1})$ by the Cauchy-Schwarz inequality) it seems intuitive that $\E|Q_lQ_qM_{qlj}(\boldsymbol{\tilde{\theta}};\boldsymbol{X})|=\mathcal{O}(1)$.  However, this cannot be guaranteed because $M_{qlj}(\boldsymbol{\tilde{\theta}};\boldsymbol{X})$ is random.
If we additionally assume that $\mathbb{E}[Q_l^4]<\infty$ for all $l=1,2,\ldots,d$ and 
\[\mathrm{max}_{\substack{\tilde{\theta}_m \in \left\lbrace\hat{\theta}_n(\boldsymbol{X})_m, \theta_{0,m}\right\rbrace\\m \in \left\lbrace 1,2,\ldots, d\right\rbrace}}\mathbb{E}[(M_{qlj}(\tilde{\boldsymbol{\theta}};\boldsymbol{X}))^2]<\infty\] 
for all $j,l,q\in\{1,2,\ldots,d\}$ then we are guaranteed that $\E|Q_lQ_qM_{qlj}(\boldsymbol{\tilde{\theta}};\boldsymbol{X})|=\mathcal{O}(1)$, meaning that $K_3(\boldsymbol{\theta}_0)=\mathcal{O}(1)$.
This is because $M_{qlj}(\boldsymbol{\theta};\boldsymbol{x})=\mathcal{O}(n)$, and $\mathbb{E}[Q_l^4]=\mathcal{O}(n^{-2})$ for all $l=1,2,\ldots,d$, provided $[I(\boldsymbol{\theta}_0)]^{-1}=\mathcal{O}(1)$.  To see this, note that, by the asymptotic normality of the MLE, we have that, for all $l=1,2,\ldots,d$, $\hat{\theta}_n(\boldsymbol{X})_{l} - \theta_{0,l}\stackrel{d}{\rightarrow} \mathrm{N}(0,\frac{1}{n}I_*)$, as $n\rightarrow\infty$, where $I_*=\sum_{j=1}^d([I(\boldsymbol{\theta}_0)]^{-1})_{lj}$. Hence, $\mathbb{E}[Q_l^4]=\mathbb{E}[(\hat{\theta}_n(\boldsymbol{X})_{l}- \theta_{0,l})^4]\rightarrow\frac{3}{n^2}I_*^2$, as $n\rightarrow\infty$. Here we used that, for $Y\sim \mathrm{N}(0,\sigma^2)$, $\mathbb{E}[Y^4]=3\sigma^4$.
 Two applications of the Cauchy-Schwarz inequality then give 
\[\E\big|Q_lQ_qM_{qlj}(\boldsymbol{\tilde{\theta}};\boldsymbol{X})\big|\leq\big(\E[Q_l^4]\E[Q_q^4]\big)^{1/4}\big(\E[(M_{qlj}(\tilde{\boldsymbol{\theta}};\boldsymbol{X}))^2]\big)^{1/2}=\mathcal{O}(1).\]
Since $K_1(\boldsymbol{\theta}_0)$, $K_2(\boldsymbol{\theta}_0)$ and $K_3(\boldsymbol{\theta}_0)$ are all $\mathcal{O}(1)$ as $n\rightarrow\infty$, it follows that the bound in Theorem \ref{Theorem_multi} is $\mathcal{O}(n^{-1/2})$.

\vspace{2mm}

\noindent{{\textbf{(2)}}} In general $\ell(\boldsymbol{\theta};\boldsymbol{x})$ and $\log(f(\boldsymbol{x}_i|\boldsymbol{\theta}_0))$ will depend on the dimension $d$ (and therefore so will $\tilde{V}_{kj}$ and $M_{qlj}(\boldsymbol{\theta};\boldsymbol{x})$, for example), and therefore it is difficult to make precise general statements regarding the dependence of the bound  (\ref{final_bound_regression}) of Theorem \ref{Theorem_multi} on the dimension $d$.
However, it is clear that the term $K_3(\boldsymbol{\theta}_0)$ has a very poor dependence on the dimension $d$. Assuming that $M_{qlj}(\boldsymbol{\theta};\boldsymbol{x})=\mathcal{O}(1)$ and $\tilde{V}_{kj}=\mathcal{O}(1)$, we have that $K_3(\boldsymbol{\theta}_0)=\mathcal{O}(d^{4}2^d)$.

This poor dependence on the dimension is a consequence of the crude inequality (\ref{maxbound}) used in the proof of Theorem \ref{Theorem_multi} below, which we now state:
\begin{align}\label{maxbound0}\E\left|Q_lQ_qM_{qlj}(\boldsymbol{\theta}_0^*;\boldsymbol{X})\right|\leq\sum\nolimits_{\substack{\tilde{\theta}_m \in \left\lbrace\hat{\theta}_n(\boldsymbol{x})_m, \theta_{0,m}\right\rbrace\\m \in \left\lbrace 1,2,\ldots, d\right\rbrace}}\E\big|Q_lQ_qM_{qlj}(\boldsymbol{\tilde{\theta}};\boldsymbol{X})\big|,
\end{align}
where $\boldsymbol{\theta}_0^*=(\theta_{0,1}^*,\theta_{0,2}^*,\ldots,\theta_{0,d}^*)^\intercal$, and $\theta_{0,j}^*:=\theta_{0,j}^*(\boldsymbol{x})=\alpha_j\theta_{0,j}+(1-\alpha_j)\hat{\theta}(\boldsymbol{x})_j$, $\alpha_j\in(0,1)$, $j=1,2,\ldots,d$. \gr{(We also introduce the monotonicity assumption on $M_{qlj}$ to obtain inequality (\ref{maxbound}).)}  Inequality (\ref{maxbound0}) is useful in that the expectations in the sum are easier to bound directly than the quantity $\E|Q_lQ_qM_{qlj}(\boldsymbol{\theta}_0^*;\boldsymbol{X})|$, but this comes at the cost of having a sum with $2^d$ terms, resulting in a poor dependence on the dimension $d$.  \gr{However,} as is demonstrated in the examples of Section \ref{sec4}, when the number of dimensions is low, inequality (\ref{maxbound0}) (which leads to the term $K_3(\boldsymbol{\theta}_0)$) is very useful as the computation of the expectations in the sum are often straightforward.  
\end{remark}

\begin{remark}\label{remcomparison}Theorem 2.1 of \cite{ar15} gives a bounded Wasserstein bound on the distance between the distribution of the single-parameter MLE and the normal distribution, and Theorem 2.1 of \cite{a18} gives a bound on the distance between the distribution of multi-parameter MLE and the multivariate normal distribution with respect to the $d_{0,1,2,3}$ metric.  Both bounds are of the optimal $\mathcal{O}(n^{-1/2})$ order.  We now give further comparisons between our bounds and those of \cite{ar15} and \cite{a18}.

Theorem 2.1 of \cite{ar15} holds under the same regularity conditions as our Theorem \ref{Theorem_multi}, but with condition (R.C.4') instead of (R.C.4'')
Condition (R.C.4') introduces a constant $\epsilon$.  This causes two complications in the bound of \cite{ar15}. Firstly, some additional conditional expectations (which involve $\epsilon$) must be estimated; secondly, $\epsilon$ appears in other terms in the bound and so in applications of the bound $\epsilon$ must later be optimised.  Our bound (\ref{boundTHEOREM}) has no such complications and in most applications we would expect that the expectations that must be estimated in our bound are easier to work with than those of \cite{ar15}, and ultimately lead to better bounds (even when given in a stronger metric).  Indeed, in Section \ref{sec:one_parameter} we apply Theorem \ref{Theorem_multi} to derive \gr{1-}Wasserstein distance bounds for the normal approximation of the MLE of the exponential distribution in the canonical and non-canonical parametrisations, and we find that in both cases our bounds outperform those that were obtained by \cite{ar15}.

Theorem 2.1 of \cite{a18} also holds under the same regularity conditions as our Theorem \ref{Theorem_multi}, but with condition (R.C.4') instead of (R.C.4'').  The bound of \cite{a18} therefore has similar complications to the bound of \cite{ar15}, and overall the bound of \cite{a18} takes a more complicated form than our bound (\ref{final_bound_regression}) in Theorem \ref{Theorem_multi}.  For small dimension $d$, we would therefore expect our bound to be preferable to that of \cite{a18} and lead to better bounds in applications.  However, as noted in Remark \ref{rem12}, the term $K_3(\boldsymbol{\theta}_0)$ of bound (\ref{final_bound_regression}) has a very poor dependence on the dimension $d$; much worse than the bound of \cite{a18}.  In applications in which the dependence on the dimension is more important than the choice of metric, the bound of \cite{a18} may \gr{therefore} be preferable to our bound (\ref{final_bound_regression}).
\end{remark}


\noindent{\emph{Proof of Theorem \ref{Theorem_multi}.}} By the triangle inequality we have that
\begin{align}
\nonumber  d_{\mathrm{W}}(\boldsymbol{W}, \boldsymbol{Z})
& \leq d_{\mathrm{W}}\bigg(\frac{1}{\sqrt{n}}\tilde{V}\nabla\left(\ell(\boldsymbol{\theta}_0;\boldsymbol{X})\right), \boldsymbol{Z}\bigg) + d_{\mathrm{W}}\bigg(\boldsymbol{W}, \frac{1}{\sqrt{n}}\tilde{V}\nabla\left(\ell\left(\boldsymbol{\theta}_0;\boldsymbol{X}\right)\right)\bigg)\\
\label{r111}&=:R_1+R_2.
\end{align}
We now proceed to find upper bounds for the terms $R_1$ and $R_2$.

The term $R_1$ is readily bounded by an application of Theorem \ref{bonisthm}. We have $\nabla\left(\ell\left(\boldsymbol{\theta}_0;\boldsymbol{X}\right)\right) = \sum_{i=1}^{n}\nabla\left(\log \left(f(\boldsymbol{X}_i|\boldsymbol{\theta}_0)\right)\right)$ and we can write $\boldsymbol{S}  = \frac{1}{\sqrt{n}}\sum_{i=1}^{n}\boldsymbol{\xi}_i,$ for $\boldsymbol{\xi}_i = \tilde{V}\nabla\left(\log\left(f(\boldsymbol{X}_i|\boldsymbol{\theta}_0)\right)\right)$, $i=1,2,\ldots,n$, being i.i.d$.$ random vectors in $\mathbb{R}^{d}$. From the regularity condition (R.C.3), it follows that $\E[\boldsymbol{\xi}_1] = \boldsymbol{0}$. In addition, using (R.C.3), we have that due to the symmetry of $\tilde{V}$,
\begin{equation}
\nonumber {\rm Var}\left(\boldsymbol{S}\right) = \frac{1}{n}\tilde{V}\sum_{i=1}^{n}\left\lbrace {\rm Var}\left(\nabla(\log(f(\boldsymbol{X}_i|\boldsymbol{\theta}_0)))\right)\right\rbrace\tilde{V}  = \tilde{V}I(\boldsymbol{\theta}_0)\tilde{V} = I_{d}.
\end{equation}
Therefore, from Theorem \ref{bonisthm} \gr{(using that $d_{\mathrm{W}}\leq d_{\mathrm{W}_2}$)} we have that
\begin{equation*}
R_1 \leq \frac{14d^{5/4}}{\sqrt{n}}\max_{1\leq j \leq d}\sqrt{\E[\xi_{1,j}^4]}=\frac{K_1(\boldsymbol{\theta}_0)}{\sqrt{n}},
\end{equation*}
where $\xi_{1j} = \sum_{k=1}^{d}\tilde{V}_{j,k}\frac{\partial}{\partial\theta_k}\left(\log\left(f(\boldsymbol{X}_1|\boldsymbol{\theta}_0)\right)\right)$.

Now we turn our attention to the more involved part of the proof, that of bounding $R_2$.  We begin by obtaining a useful expression for $\boldsymbol{W}=\sqrt{n}[I(\boldsymbol{\theta}_0)]^{1/2}(\hat{\boldsymbol{\theta}}_n(x) - \boldsymbol{\theta}_0)$.
From the definition of the MLE we have that $\frac{\partial}{\partial\theta_k}\ell(\hat{\boldsymbol{\theta}}_n(\boldsymbol{x});\boldsymbol{x}) = 0$ for all $k=1,2,\ldots,d$. A second order Taylor expansion of $\frac{\partial}{\partial\theta_k}\ell(\hat{\boldsymbol{\theta}}_n(\boldsymbol{x}); \boldsymbol{x})$ around $\boldsymbol{\theta}_0$ gives that
\begin{equation}
\label{Taylor1}
\sum_{j=1}^{d}Q_j\frac{\partial^2}{\partial\theta_k\partial\theta_j}\ell(\boldsymbol{\theta}_0;\boldsymbol{x}) = -\frac{\partial}{\partial\theta_k}\ell(\boldsymbol{\theta}_0;\boldsymbol{x}) - \frac{1}{2}\sum_{j=1}^{d}\sum_{q=1}^{d}Q_jQ_q\frac{\partial^3}{\partial\theta_k\partial\theta_j\partial\theta_q}\ell\left(\boldsymbol{\theta};\boldsymbol{x}\right)\Big|_{\substack{\boldsymbol{\theta} = \boldsymbol{\theta}_0^{*}}}.
\end{equation}
Here $\boldsymbol{\theta}_0^*=(\theta_{0,1}^*,\theta_{0,2}^*,\ldots,\theta_{0,d}^*)^\intercal$, where $\theta_{0,j}^*:=\theta_{0,j}^*(\boldsymbol{x})=\alpha_j\theta_{0,j}+(1-\alpha_j)\hat{\theta}(\boldsymbol{x})_j$, $\alpha_j\in(0,1)$, $j=1,2,\ldots,d$. Adding now $\sum_{j=1}^{d}n[I(\boldsymbol{\theta}_0)]_{kj}Q_j$ on both sides of \eqref{Taylor1}, we obtain
\begin{equation}
\label{Taylor2}
\nonumber \sum_{j=1}^{d}n[I(\boldsymbol{\theta}_0)]_{kj}Q_j = \frac{\partial}{\partial \theta_k}\ell(\boldsymbol{\theta}_0;\boldsymbol{x}) + \sum_{j=1}^{d}Q_jT_{kj} + \frac{1}{2}\sum_{j=1}^{d}\sum_{q=1}^{d}Q_jQ_q\frac{\partial^3}{\partial\theta_k\partial\theta_j\partial\theta_q}\ell\left(\boldsymbol{\theta};\boldsymbol{x}\right)\Big|_{\substack{\boldsymbol{\theta} = \boldsymbol{\theta}_0^{*}}}.
\end{equation}
The equality above holds for all $k = 1,2,\ldots, d$, which means that, for $\left[I(\boldsymbol{\theta}_0)\right]_{[j]}$ denoting the $j$-th column of the matrix $I(\boldsymbol{\theta}_0)$,
\begin{align}
\label{Taylor_multi}
\nonumber\boldsymbol{W}&= \sqrt{n}[I(\boldsymbol{\theta}_0)]^{1/2}\big(\hat{\boldsymbol{\theta}}_n(\boldsymbol{x}) - \boldsymbol{\theta}_0\big)\\
 \nonumber&=  \frac{1}{\sqrt{n}}\tilde{V}\bigg\{\nabla\left(\ell\left(\boldsymbol{\theta}_0;\boldsymbol{x}\right)\right) + \sum_{j=1}^{d}Q_j\left(\nabla\left(\frac{\partial}{\partial\theta_j}\ell(\boldsymbol{\theta}_0;\boldsymbol{x})\right) + n[I(\boldsymbol{\theta}_0)]_{[j]}\right)\\
& \qquad\quad\quad +  \frac{1}{2}\sum_{j=1}^{d}\sum_{q=1}^{d}Q_jQ_q\nabla\left(\frac{\partial^2}{\partial\theta_j\partial\theta_q}\ell\left(\boldsymbol{\theta};\boldsymbol{x}\right)\Big|_{\substack{\boldsymbol{\theta} = \boldsymbol{\theta_0^{*}}}}\right)\bigg\},
\end{align}
where we multiplied both sides by $\frac{1}{\sqrt{n}}[I(\boldsymbol{\theta}_0)]^{-1/2} = \frac{1}{\sqrt{n}}\tilde{V}$.

Now, from the \gr{integral probability metric representation of the 1-Wasserstein} distance we have that
\begin{equation*}R_2=\sup_{h\in\mathcal{H}_{\mathrm{W}}}|\E[h(\boldsymbol{W})]-\E[h(n^{-1/2}\tilde{V}\nabla\left(\ell(\boldsymbol{\theta}_0;\boldsymbol{X})\right))]|\gr{.}
\end{equation*}
Let $h\in\mathcal{H}_{\mathrm{W}}$.  Then, by \eqref{Taylor_multi},
\begin{align*}
&\big|\E[h(\boldsymbol{W})]-\E[h(n^{-1/2}\tilde{V}\nabla\left(\ell(\boldsymbol{\theta}_0;\boldsymbol{X})\right))]\big| \\
&\leq \|h\|_{\mathrm{Lip}}\E\bigg|\frac{1}{\sqrt{n}}\tilde{V}\bigg\{ \sum_{j=1}^{d}Q_j\left(\nabla\left(\frac{\partial}{\partial\theta_j}\ell(\boldsymbol{\theta}_0;\boldsymbol{x})\right) + n[I(\boldsymbol{\theta}_0)]_{[j]}\right)\\
& \qquad\quad\quad +  \frac{1}{2}\sum_{j=1}^{d}\sum_{q=1}^{d}Q_jQ_q\nabla\left(\frac{\partial^2}{\partial\theta_j\partial\theta_q}\ell\left(\boldsymbol{\theta};\boldsymbol{x}\right)\Big|_{\substack{\boldsymbol{\theta} = \boldsymbol{\theta_0^{*}}}}\right)\bigg\}\bigg|,
\end{align*}
and, by the triangle inequality,
\begin{align}
\label{result_term_2}
\nonumber &\big|\E[h(\boldsymbol{W})]-\E[h(n^{-1/2}\tilde{V}\nabla\left(\ell(\boldsymbol{\theta}_0;\boldsymbol{X})\right))]\big| \\
\nonumber  & \leq \frac{\|h\|_{\mathrm{Lip}}}{\sqrt{n}}\bigg\{\sum_{k=1}^{d}\sum_{j=1}^{d}\sum_{l=1}^{d}|\tilde{V}_{kj}|\mathbb{E}|Q_lT_{lj}| + \frac{1}{2}\sum_{k=1}^{d}\sum_{j=1}^{d}|\tilde{V}_{kj}|\sum_{l=1}^{d}\sum_{q=1}^{d}\E\bigg|Q_lQ_q\frac{\partial^3}{\partial\theta_q\partial\theta_l\partial\theta_j}\ell\left(\boldsymbol{\theta};\boldsymbol{X}\right)\Big|_{\substack{\boldsymbol{\theta} = \boldsymbol{\theta}_0^{*}}}\bigg|\bigg\} \\
 & \leq \frac{\|h\|_{\mathrm{Lip}}}{\sqrt{n}}\sum_{k=1}^{d}\sum_{j=1}^{d}|\tilde{V}_{kj}|\bigg\{\sum_{l=1}^{d}\sqrt{\E[Q_l^2]}\sqrt{\E[T_{lj}^2]} + \frac{1}{2}\sum_{l=1}^{d}\sum_{q=1}^d\E\left|Q_lQ_qM_{qlj}(\boldsymbol{\theta}_0^*;\boldsymbol{X})\right|\bigg\},
\end{align}
where $M_{qlj}(\boldsymbol{\theta};\boldsymbol{x})$ is as in the condition (R.C.4''). In obtaining the final inequality we used the Cauchy-Schwarz inequality.


Let us now focus on bounding $\E|Q_lQ_qM_{qlj}(\boldsymbol{\theta}_0^*;\boldsymbol{X})|$.
As $M_{qlj}$ is a monotonic function in the sense defined in (\ref{monotonicity}), we have that, for all $\boldsymbol{x}\in\mathbb{X}$,
\begin{equation*}M_{qlj}(\boldsymbol{\theta}_0^*(\boldsymbol{x});\boldsymbol{x})\leq \mathrm{max}_{\substack{\tilde{\theta}_m \in \left\lbrace\hat{\theta}_n(\boldsymbol{x})_m, \theta_{0,m}\right\rbrace\\m \in \left\lbrace 1,2,\ldots, d\right\rbrace}}M_{qlj}(\boldsymbol{\tilde{\theta}};\boldsymbol{x}).
\end{equation*}
Therefore
\begin{align}\E\left|Q_lQ_qM_{qlj}(\boldsymbol{\theta}_0^*;\boldsymbol{X})\right|&\leq \E\bigg|Q_lQ_q\mathrm{max}_{\substack{\tilde{\theta}_m \in \left\lbrace\hat{\theta}_n(\boldsymbol{\gr{X}})_m, \theta_{0,m}\right\rbrace\\m \in \left\lbrace 1,2,\ldots, d\right\rbrace}}M_{qlj}(\boldsymbol{\tilde{\theta}};\boldsymbol{X})\bigg|\nonumber\\
&\leq\E\bigg|Q_lQ_q\sum\nolimits_{\substack{\tilde{\theta}_m \in \left\lbrace\hat{\theta}_n(\boldsymbol{\gr{X}})_m, \theta_{0,m}\right\rbrace\\m \in \left\lbrace 1,2,\ldots, d\right\rbrace}}M_{qlj}(\boldsymbol{\tilde{\theta}};\boldsymbol{X})\bigg|\nonumber\\
\label{maxbound}&=\sum\nolimits_{\substack{\tilde{\theta}_m \in \left\lbrace\hat{\theta}_n(\boldsymbol{\gr{X}})_m, \theta_{0,m}\right\rbrace\\m \in \left\lbrace 1,2,\ldots, d\right\rbrace}}\E\big|Q_lQ_qM_{qlj}(\boldsymbol{\tilde{\theta}};\boldsymbol{X})\big|.
\end{align}
Applying inequality (\ref{maxbound}) to (\ref{result_term_2}) gives the bound
\begin{align}\nonumber &\big|\E[h(\boldsymbol{W})]-\E[h(n^{-1/2}\tilde{V}\nabla\left(\ell(\boldsymbol{\theta}_0;\boldsymbol{X})\right))]\big| \\
\nonumber & \leq \frac{\|h\|_{\mathrm{Lip}}}{\sqrt{n}}\sum_{k=1}^{d}\sum_{j=1}^{d}|\tilde{V}_{kj}|\Bigg\{\sum_{l=1}^{d}\sqrt{\E[Q_l^2]}\sqrt{\E[T_{lj}^2]} + \frac{1}{2}\sum_{l=1}^{d}\sum_{q=1}^d\sum_{\substack{\tilde{\theta}_m \in \left\lbrace\hat{\theta}_n(\boldsymbol{X})_m, \theta_{0,m}\right\rbrace\\m \in \left\lbrace 1,2,\ldots, d\right\rbrace}}\E\big|Q_lQ_qM_{qlj}(\boldsymbol{\tilde{\theta}};\boldsymbol{X})\big|\Bigg\}\\
\nonumber&=\frac{\|h\|_{\mathrm{Lip}}}{\sqrt{n}}\big(K_2(\boldsymbol{\theta}_0)+K_3(\boldsymbol{\theta}_0)\big),
\end{align}
Since $h\in\mathcal{H}_{\mathrm{W}}$ we have that $\|h\|_{\mathrm{Lip}}\leq1$, and therefore $R_2\leq \frac{1}{\sqrt{n}}\big(K_2(\boldsymbol{\theta}_0)+K_3(\boldsymbol{\theta}_0)\big)$.
Finally, combining our bounds for $R_1$ and $R_2$ yields inequality \eqref{final_bound_regression}. \hfill $\Box$

\vspace{3mm}

\noindent{\emph{Proof of Theorem \ref{Theoremnoncan}}.}  The proof is exactly the same as that of Theorem \ref{Theorem_multi} with the exception that the term $R_1$ in (\ref{r111}) is bounded using Theorem \ref{thmapw}, rather than Theorem \ref{bonisthm}. \hfill $\Box$

\vspace{3mm}

\gr{\noindent{\emph{Proof of Theorem \ref{thmwasp}}.} The proof is similar to that of Theorem \ref{Theorem_multi}. Let $p\geq2$. By the triangle inequality we have that
\begin{align}
\nonumber  d_{\mathrm{W}_p}(\boldsymbol{W}, \boldsymbol{Z})
& \leq d_{\mathrm{W}_p}\bigg(\frac{1}{\sqrt{n}}\tilde{V}\nabla\left(\ell(\boldsymbol{\theta}_0;\boldsymbol{X})\right), \boldsymbol{Z}\bigg) + d_{\mathrm{W}_p}\bigg(\boldsymbol{W}, \frac{1}{\sqrt{n}}\tilde{V}\nabla\left(\ell\left(\boldsymbol{\theta}_0;\boldsymbol{X}\right)\right)\bigg)\\
&=:R_{1,p}+R_{2,p}.\nonumber
\end{align}
The term $R_{1,p}$ can be bounded similarly to how we bounded $R_1$ in the proof of Theorem \ref{Theorem_multi}. In the case $p=2$ we obtain the same bound $K_1(\boldsymbol{\theta}_0)$ for $R_{1,2}$, and for the case $p\geq2$ the only way our argument changes is that we apply inequality (\ref{bonisbound200}), rather than inequality (\ref{bonisbound2}).

To bound $R_{2,p}$, we note that the random vectors $\boldsymbol{W}$ and $\frac{1}{\sqrt{n}}\tilde{V}\nabla\left(\ell\left(\boldsymbol{\theta}_0;\boldsymbol{X}\right)\right)$ are defined on the same probability space and thus provide a coupling of them. It therefore follows from the definition of the $p$-Wasserstein distance that
\begin{equation*}R_{2,p}=d_{\mathrm{W}_p}\bigg(\boldsymbol{W}, \frac{1}{\sqrt{n}}\tilde{V}\nabla\left(\ell\left(\boldsymbol{\theta}_0;\boldsymbol{X}\right)\right)\bigg)\leq\Bigg(\mathbb{E}\bigg[\bigg|\boldsymbol{W}-\frac{1}{\sqrt{n}}\tilde{V}\nabla\left(\ell\left(\boldsymbol{\theta}_0;\boldsymbol{X}\right)\right)\bigg|^p\bigg]\Bigg)^{1/p}.
\end{equation*}
Substituting (\ref{Taylor_multi}) into this bound and using the triangle inequality now gives that
\begin{align*}R_{2,p}&\leq\frac{1}{\sqrt{n}}\Bigg(\mathbb{E}\bigg[\bigg(\sum_{k=1}^d\sum_{j=1}^d\sum_{l=1}^d|\tilde{V}_{kj}Q_lT_{lj}|\\
&\quad+\frac{1}{2}\sum_{k=1}^d\sum_{j=1}^d\sum_{l=1}^d\sum_{q=1}^d\bigg|\tilde{V}_{kj}Q_lQ_q\frac{\partial^3}{\partial\theta_q\partial\theta_l\partial\theta_j}\ell\left(\boldsymbol{\theta}_0^{*};\boldsymbol{X}\right)\bigg|\bigg)^p\bigg]\Bigg)^{1/p}\\
&\leq\frac{1}{\sqrt{n}}\Bigg\{\Bigg(\mathbb{E}\bigg[\bigg(\sum_{k=1}^d\sum_{j=1}^d\sum_{l=1}^d|\tilde{V}_{kj}Q_lT_{lj}|\bigg)^p\bigg]\Bigg)^{1/p}\\
&\quad+\frac{1}{2}\Bigg(\mathbb{E}\bigg[\bigg(\sum_{k=1}^d\sum_{j=1}^d\sum_{l=1}^d\sum_{q=1}^d\bigg|\tilde{V}_{kj}Q_lQ_q\frac{\partial^3}{\partial\theta_q\partial\theta_l\partial\theta_j}\ell\left(\boldsymbol{\theta}_0^{*};\boldsymbol{X}\right)\bigg|\bigg)^p\bigg]\Bigg)^{1/p}\Bigg\}.
\end{align*} 
We now apply the inequality  $\big(\sum_{j=1}^d a_j\big)^{r}\leq d^{r-1}\sum_{j=1}^da_j^r$, where $a_1,\ldots,a_d\geq0$ and $r\geq2$, to get 
\begin{align*}R_{2,p}&\leq \frac{1}{\sqrt{n}}\Bigg\{d^{3-3/p}\Bigg(\sum_{k=1}^d\sum_{j=1}^d|\tilde{V}_{kj}|^p\sum_{l=1}^d\mathbb{E}\big[|Q_lT_{lj}|^p\big]\Bigg)^{1/p}\\
&\quad+\frac{d^{4-4/p}}{2}\Bigg(\sum_{k=1}^d\sum_{j=1}^d|\tilde{V}_{kj}|^p\sum_{l=1}^d\sum_{q=1}^d\mathbb{E}\bigg[\bigg|Q_lQ_q\frac{\partial^3}{\partial\theta_q\partial\theta_l\partial\theta_j}\ell\left(\boldsymbol{\theta}_0^{*};\boldsymbol{X}\right)\bigg|^p\bigg]\Bigg)^{1/p}\Bigg\}\\
&\leq\frac{1}{\sqrt{n}}\Bigg\{d^{3-3/p}\Bigg(\sum_{k=1}^{d}\sum_{j=1}^{d}|\tilde{V}_{kj}|^p\sum_{l=1}^{d}\sqrt{\E[|Q_l|^{2p}]}\sqrt{\E[|T_{lj}|^{2p}]}\Bigg)^{1/p}\\
&\quad+\frac{d^{4-4/p}}{2}\Bigg(\sum_{k=1}^{d}\sum_{j=1}^{d}|\tilde{V}_{kj}|^p\sum_{l=1}^{d}\sum_{q=1}^d\sum_{\substack{\tilde{\theta}_m \in \left\lbrace\hat{\theta}_n(\boldsymbol{X})_m, \theta_{0,m}\right\rbrace\\m \in \left\lbrace 1,2,\ldots, d\right\rbrace}}\E\big[\big|Q_lQ_qM_{qlj}(\boldsymbol{\tilde{\theta}};\boldsymbol{X})\big|^p\big]\bigg)^{1/p}\Bigg\}\\
&=\frac{1}{\sqrt{n}}\big(K_{2,p}(\boldsymbol{\theta}_0)+K_{3,p}(\boldsymbol{\theta}_0)\big),
\end{align*}
where in obtaining the second inequality we used the Cauchy-Schwarz inequality and a similar argument to the one used to obtain inequality (\ref{maxbound}). Summing up our bounds for $R_{1,p}$ and $R_{2,p}$, in the cases $p\geq2$ and $p=2$, yields the desired bounds (\ref{final_bound_regression0}) and (\ref{final_bound_regression0z}), respectively. \hfill $\Box$}

\section{Examples}\label{sec4}

In this section, we apply the general theorems of Section \ref{sec3} to obtain explicit optimal $\mathcal{O}(n^{-1/2})$ Wasserstein distance bounds for the multivariate normal approximation of the MLE in several important settings.  Each of the examples given is of interest in its own right and taken together the examples provide a useful demonstration of the application of the general theorems to derive explicit bounds for particular MLEs of interest. \gr{Our focus in this section is mostly on obtaining bounds with respect to the 1-Wasserstein metric, although we do derive some bounds with respect to the $2$-Wasserstein metric. It should be noted, however, that $p$-Wasserstein ($p\geq1$) analogues of each of the bounds derived in this section can be obtained through an application of Theorem \ref{thmwasp}; see Corollary \ref{corollaryexponential} for a 2-Wasserstein distance bound for the normal approximation of the exponential distribution under canonical parametrisation.} \aaa{In Section Proposition \ref{prop_implicit_MLE} we provide an upper bound \gdr{with respect to the bounded Wasserstein distance} for cases where the MLE cannot be expressed analytically.}

\subsection{Single-parameter exponential families}
\label{sec:one_parameter}

The distribution of a random variable, $X$, is said to be a \emph{single-parameter exponential family distribution} if the probability density (or mass) function is of the form
\begin{equation}
\label{density_exponential_family}
f(x|\theta) = {\exp}\left\lbrace k(\theta)T(x) - A(\theta) + S(x)\right\rbrace\mathbf{1}_{\{x \in B\}},
\end{equation}
where the set $B = \left\lbrace x:f(x|\theta)>0 \right\rbrace$ is the support of $X$ and does not depend on $\theta$; $k(\theta)$ and $A(\theta)$ are functions of the parameter; $T(x)$ and $S(x)$ are functions only of the data.  Many popular distributions are members of the exponential family, including the normal, gamma and beta distributions.

 The choice of the functions $k(\theta)$ and $T(X)$ is not unique. If $k(\theta) = \theta$ we have the so-called \emph{canonical case}. In this case $\theta$ and $T(X)$ are called the \emph{natural} \emph{parameter} and \emph{natural} \emph{observation} \cite{cb02}.  
 It is often of interest to work under the canonical parametrisation due to appealing theoretical properties that can, for example, simplify the theory and computational complexity in generalised
linear models.  In fact, as noted in Remark \ref{REMARK_CANONICAL_R2} below, our general (\ref{noncanexponential}) bound in Corollary \ref{Theoremnoncanexp} for the normal approximation of the MLE for exponential family distributions simplifies in the canonical case.  Canonical parametrisations are important in, amongst other examples, Gaussian graphical models \cite{l96} and precision matrix estimation \cite{m18}.

\gr{\begin{corollary}
\label{Theoremnoncanexp}
Let $X_1, X_2, \ldots, X_n$ be i.i.d$.$ random variables with the probability density (or mass) function of a single-parameter exponential family distribution, as given in (\ref{density_exponential_family}). Assume that (R1)--(R3) are satisfied and that the MLE exists. Assuming that $k'(\theta_{0}) \neq 0$ and denoting by $D(\theta) = \frac{A'(\theta)}{k'(\theta)}$, then with $W=\sqrt{n\:i(\theta_0)}(\hat{\theta}_n(\boldsymbol{x})- \theta_0)$ and $Z \sim {\rm N}(0,1)$, it holds that

\vspace{3mm}

\noindent{\textbf{(1)}} If (R.C.4'') is satisfied and for $M(\theta,\boldsymbol{x})$ as in (R.C.4''), then
\begin{align}
\label{noncanexponential}
\nonumber & d_{\mathrm{W}}(W,Z) \leq \frac{1}{\sqrt{n}}\bigg[2 + \frac{{\EE}[|T(X_1)-D(\theta_0)|^3]}{\left[{\rm Var}(T(X_1))\right]^{3/2}} + \frac{|k''(\theta_0)|}{\sqrt{i(\theta_0)}}\sqrt{n{\rm Var}\left(T(X_1)\right)}\sqrt{{\EE}\big[(\hat{\theta}_n(\boldsymbol{X})- \theta_0)^2\big]}\\
&\quad + \frac{1}{2\sqrt{i(\theta_0)}}\Big(\EE\big|(\hat{\theta}_n(\boldsymbol{X}) - \theta_0)^2M(\theta_0;\boldsymbol{X})\big|+ \EE\big|(\hat{\theta}_n(\boldsymbol{X}) - \theta_0)^2M(\hat{\theta}_n(\boldsymbol{X});\boldsymbol{X})\big|\Big)\bigg].
\end{align}
\noindent{\textbf{(2)}} If (R.C.4''(2)) is satisfied and for $M(\theta,\boldsymbol{x})$ as in (R.C.4''(2)), then
\begin{align}
\label{noncanexponential2}
\nonumber & d_{\mathrm{W}_2}(W,Z) \leq \frac{1}{\sqrt{n}}\bigg[\frac{14}{i(\theta_0)}[k'(\theta_0)]^2\sqrt{\E\left[(T(X_1) - D(\theta_0))^4\right]}\\
\nonumber & \quad + \frac{|k''(\theta_0)|}{\sqrt{i(\theta_0)}}\Big(\E\big[(\hat{\theta}_n(\boldsymbol{X}) - \theta_0)^4\big]\Big)^{1/4}\bigg(\E\bigg[\bigg(\sum_{i=1}^n\lbrace T(X_i) - \E[T(X_i)]\rbrace\bigg)^4\bigg]\bigg)^{1/4}\\
&\quad + \frac{1}{2\sqrt{i(\theta_0)}}\Big(\E\big[(\hat{\theta}_n(\boldsymbol{X}) - \theta_0)^4(M(\theta_0;\boldsymbol{X}))^2\big] + \E\big[(\hat{\theta}_n(\boldsymbol{X}) - \theta_0)^4(M(\hat{\theta}_n(\boldsymbol{X});\boldsymbol{X}))^2\big]\Big)^{1/2}\bigg].
\end{align}
In both {\textbf{(1)}} and {\textbf{(2)}} above, $i(\theta_0) = {\rm Var}\left(\frac{\mathrm{d}}{\mathrm{d}\theta} \log f(X_1|\theta_0)\right) = [k'(\theta_0)]^2{\rm Var}(T(X_1)) > 0$.
\end{corollary}
\begin{proof} 
{\textbf{(1)}}: We have that
\begin{align*}
{\EE}\bigg[\left|\frac{\mathrm{d}}{\mathrm{d}\theta}\log f(X_1|\theta_0)\right|^3\bigg]  = {\EE}\left[\left|k'(\theta_0)T(X_1)-A'(\theta_0)\right|^3\right]  = |k'(\theta_0)|^3{\EE}\left[|T(X_1)-D(\theta_0)|^3\right]
\end{align*}
and
\begin{equation}
\nonumber {\rm Var}\left(\frac{\mathrm{d}^2}{\mathrm{d}\theta^2} \log f(X_1|\theta_0)\right) = {\rm Var}\left(k''(\theta_0)T(X_1)-A''(\theta_{0})\right) = \left[k''(\theta_0)\right]^2{\rm Var}\left(T(X_1)\right),
\end{equation}
and applying these formulas to the bound \eqref{boundTHEOREM} yields the bound (\ref{noncanexponential}).
\vspace{0.05in}
\\
{\textbf{(2)}}: Using the general result in \eqref{final_bound_regression0z} and the expression of $K_1(\theta_0)$ as in \eqref{notation_Ki}, we have in this specific case for $d=1$ that
\begin{equation}
\label{K1W2}
K_1(\theta_0) = \frac{14}{i(\theta_0)}\sqrt{\E\left[\left(\frac{\mathrm{d}}{\mathrm{d}\theta}\log f(X_1|\theta_0)\right)^4\right]} = \frac{14[k'(\theta_0)]^2}{i(\theta_0)}\sqrt{\E\big[\left(T(X_1) - D(\theta_0)\right)^4\big]}.
\end{equation}
With respect to $K_{2,2}(\theta_0)$ as in \eqref{notationWp}, we have that
\begin{align}
\label{K2W2}
\nonumber K_{2,2}(\theta_0) &= \frac{1}{\sqrt{i(\theta_0)}}\Big(\E\big[(\hat{\theta}_n(\boldsymbol{X}) - \theta_0)^{4}\big]\Big)^{1/4}\Big(\E\big[(\ell''(\theta_0;\boldsymbol{X}) + n\,i(\theta_0))^{4}\big]\Big)^{1/4}\\
& = \frac{|k''(\theta_0)|}{\sqrt{i(\theta_0)}}\Big(\E\big[(\hat{\theta}_n(\boldsymbol{X}) - \theta_0)^{4}\big]\Big)^{1/4}\bigg(\E\bigg[\bigg(\sum_{i=1}^{n}\lbrace T(X_i) - \E[T(X_i)]\rbrace\bigg)^{4}\bigg]\bigg)^{1/4}.
\end{align}
Combining \eqref{K1W2} and \eqref{K2W2} with the general result of \eqref{final_bound_regression0z} leads to the upper bound in \eqref{noncanexponential2}.
\end{proof}
}
\begin{remark}
\label{REMARK_CANONICAL_R2}
In the canonical case, $k''(\theta_0) \equiv 0$ and the second term of the bound\gr{s in \eqref{noncanexponential} and \eqref{noncanexponential2}} vanishes. Also\gr{, in this specific case,} $\frac{{\mathrm{d}}^2}{{\mathrm{d}}\theta^2}\log f(x|\theta) = -A''(\theta)$ and $i(\theta_0) = A''(\theta_0)$. In addition, $\frac{{\mathrm{d}}^3}{{\mathrm{d}}\theta^3}\log f(x|\theta) = -A^{(3)}(\theta)$ is independent of the random variables. This will make it easier to find a monotonic function $M(\theta)$ as in (R.C.4'') \gr{and (R.C.4''(2))}, which will be a bound for $n|A^{(3)}(\theta)|$.
\end{remark}

We give two examples using the exponential distribution, firstly, in its canonical form, and then, in Appendix \ref{expapp} under a change of parametrisation.  The example given in the appendix is given for purely illustrative purposes, as an improved bound can be obtained directly by Stein's method.


In the case of $X_1, X_2, \ldots, X_n$ exponentially distributed $\mathrm{Exp}(\theta)$, i.i.d$.$ random variables where $\theta>0$, the probability density function is
\begin{equation*}
f(x|\theta) = \theta{\rm exp}\{-\theta x\}\mathbf{1}_{\{x >0\}} = {\rm exp}\{\log {\theta} - \theta x\}\mathbf{1}_{\{x>0\}} = {\rm exp}\left\lbrace k(\theta)T(x) - A(\theta) + S(x)\right\rbrace\mathbf{1}_{\{x \in B\}},
\end{equation*}
where $B = (0,\infty)$, $\theta \in \Theta = (0,\infty)$, $T(x) = -x$, $k(\theta)=\theta$, $A(\theta) = -\log{\theta}$ and $S(x)=0$. Hence $\mathrm{Exp}(\theta)$ is a single-parameter canonical exponential family distribution.  The MLE is unique and given by $\hat{\theta}_n(\boldsymbol{X}) = \frac{1}{\bar{X}}$.
\gr{\begin{corollary}
\label{corollaryexponential}
Let $X_1, X_2, \ldots, X_n$ be i.i.d$.$ random variables that follow the $\mathrm{Exp}(\theta_0)$ distribution.  Let $W=\sqrt{n\:i(\theta_0)}(\hat{\theta}_n(\boldsymbol{x})- \theta_0)$ and $Z\sim\mathrm{N}(0,1)$. Then,
\begin{itemize}
\item[\textbf{(1)}] For $n>2$,
\begin{equation}
\label{boundexponential3}
d_{\mathrm{W}}(W,Z) < \frac{5.41456}{\sqrt{n}} + \frac{\sqrt{n}(n+2)}{(n-1)(n-2)}+\frac{2}{n^{3/2}}.
\end{equation}
\item[\textbf{(2)}]For $n>4$,
\begin{equation}
\label{boundexponential4}
d_{\mathrm{W}_2}(W,Z) \leq \frac{42}{\sqrt{n}} + \frac{1}{2\sqrt{n}}\left[\frac{1144n^4 + 2028n^3 + 1576n^2 + 480n}{(n-1)(n-2)(n-3)(n-4)}\right]^{1/2}.
\end{equation}
\end{itemize} 
\end{corollary}
}
\begin{remark}\label{remcanex1} The rate of convergence of the bound\gr{s \eqref{boundexponential3} and \eqref{boundexponential4} is $n^{-1/2}$} and the bound\gr{s do} not depend on the value of $\theta_0$. A bound with such properties was also obtained by \cite{ar15} in the bounded Wasserstein metric. \gr{Despite working in a stronger metric, in the case of the $1$-Wasserstein metric result of \eqref{boundexponential3}, we are able to give smaller numerical constants than \cite{ar15}.}

It should be noted that the exact \gr{values for $d_{\mathrm{W}}(W,Z)$ and $d_{\mathrm{W}_2}(W,Z)$ do not depend on $\theta_0$.}  This is because a simple scaling argument using the fact that $i(\theta_0)=\frac{1}{\theta_0^2}$ shows that the distribution of $W=\sqrt{n\:i(\theta_0)}(\hat{\theta}_n(\boldsymbol{x})- \theta_0)$ does not involve $\theta_0$.  Hence, it is a desirable feature of our bound\gr{s that they do} not depend on $\theta_0$.
\end{remark}
\begin{proof}
Straightforward steps can be followed in order to prove that the assumptions (R1)--(R3), (R.C.4''), \gr{and (R.C.4''(2))} hold for this example. We will not show that here. The log-likelihood function is
\begin{align}
\nonumber  \ell(\theta_0;\boldsymbol{x}) = -nA(\theta_0) + k(\theta_0)\sum_{i=1}^{n}T(x_i) = n(\log\theta_0 - \theta_0\bar{x}),
\end{align}
and its third derivative is given by  $\ell^{(3)}(\theta_0;\boldsymbol{x}) = -nA^{(3)}(\theta_0) = \frac{2n}{\theta_0^3}$.
We see that $|\ell^{(3)}(\theta;\boldsymbol{x})| = \frac{2n}{\theta^3}$, which is a decreasing function with respect to $\theta$, and therefore conditions (R.C.4'') \gr{and (R.C.4''(2)) that are necessary for the results in \eqref{boundexponential3} and \eqref{boundexponential4}, respectively,} are satisfied with $M(\theta,\boldsymbol{x}) = \frac{2n}{\theta^3}$. \gr{We now proceed to separately prove results {\textbf{(1)}} and {\textbf{(2)}} of Corollary \ref{corollaryexponential}.
\vspace{3mm}
\\
{\textbf{For (1):}}} Basic calculations of integrals show that $\EE[|T(X_1)-D(\theta_0)|^3] = \EE\big[\big|\frac{1}{\theta_0} - X_1\big|^3\big] < \frac{2.41456}{\theta_0^3}$.  In addition, since $T(x)=x$, we have that ${\rm Var}(T(X_1)) = {\rm Var}(X_1) = \frac{1}{\theta_0^2}$ and therefore for the first term of the upper bound in \eqref{noncanexponential}, we have that
\begin{align}
\label{upperbound_first_term_can_exponential}
& \frac{1}{\sqrt{n}}\bigg(2 + \frac{{\EE}[|T(X_1)-D(\theta_0)|^3]}{\left[{\rm Var}(T(X_1))\right]^{3/2}}\bigg) < \frac{4.41456}{\sqrt{n}}.
\end{align}
According to Remark \ref{REMARK_CANONICAL_R2}, the second term of the bound in \eqref{noncanexponential} vanishes. Finally, we consider the third term. Recall that we can take $M(\theta,\boldsymbol{x}) = \frac{2n}{\theta^3}$. We know that since $X_i \sim {\rm Exp}(\theta_0)$, $i=1,2,\ldots,n$, we have that $\bar{X} \sim {\rm G}(n,n\theta_0)$, with ${\rm G}(\alpha, \beta)$ being the gamma distribution with shape parameter $\alpha$ and rate parameter $\beta$. Using now the fact that $\hat{\theta}_n(\boldsymbol{x}) = \frac{1}{\bar{x}}$, the results in pp$.$ 70--73 of \cite{DistributionTheory} give that, for $n>2$,
\begin{align}
\label{M_for_theta0}
& \EE\big|(\hat{\theta}_n(\boldsymbol{X}) - \theta_0)^2M(\theta_0;\boldsymbol{X})\big| = \frac{2n}{\theta_0^3}\EE\bigg[\left(\frac{1}{\bar{X}} - \theta_0\right)^2 \bigg] = \frac{2n(n+2)}{\theta_0(n-1)(n-2)}
\end{align}
and
\begin{align}
\label{M_for_thetahat}
\nonumber  \EE\big|(\hat{\theta}_n(\boldsymbol{X}) - \theta_0)^2M(\hat{\theta}_n(\boldsymbol{X});\boldsymbol{X})\big| &= 2n\EE\bigg[\bar{X}^3\left(\frac{1}{\bar{X}} - \theta_0\right)^2\bigg] = 2n\EE\big[\bar{X} + \theta_0^2\bar{X}^3 - 2\theta_0\bar{X}^2\big]\\
& = \frac{2n}{\theta_0}\left(1 + \frac{(n+1)(n+2)}{n^2} - 2\frac{n+1}{n}\right) = \frac{2(n+2)}{n\theta_0}.
\end{align}
Applying the results of \eqref{upperbound_first_term_can_exponential} , \eqref{M_for_theta0} and \eqref{M_for_thetahat} to \eqref{noncanexponential} and using that $i(\theta_0) = \frac{1}{\theta_0^2}$, yields \gr{result {\textbf{(1)}}} of the corollary.
\vspace{3mm}
\\
\gr{{\textbf{For (2):}} For the first term of the upper bound in \eqref{noncanexponential2} we have, since $i(\theta_0) = \frac{1}{\theta_0^2}$ and $k(\theta_0) = \theta_0$, that
\begin{equation}
\label{first_term_exp_W2}
\frac{14}{\sqrt{n}\,i(\theta_0)}[k'(\theta_0)]^2\sqrt{\E\big[\left(T(X_1) - D(\theta_0)\right)^4\big]} = \frac{14\theta_0^2}{\sqrt{n}}\sqrt{\E\bigg[\bigg(X_1 - \frac{1}{\theta_0}\bigg)^4\bigg]} = \frac{42}{\sqrt{n}},
\end{equation}
where we used that the fourth central moment of  $X \sim {\rm Exp}(\theta_0)$ is given by $\mathbb{E}[(X-\frac{1}{\theta_0})^4]=\frac{9}{\theta_0^4}$ .
The second term in \eqref{noncanexponential2} vanishes due to $k''(\theta_0) = 0$. With respect to the third term, since $\ell^{(3)}(\theta_0;\boldsymbol{x}) = \frac{2n}{\theta_0^3}$, we take $M(\theta_0;\boldsymbol{x}) = \frac{2n}{\theta_0^3}$. We have already mentioned that $\hat{\theta}_n(\boldsymbol{X}) = \frac{1}{\bar{X}}$ and $\bar{X} \sim G(n,n\theta_0)$. Therefore, simple calculations yield
\begin{equation}
\label{third_term_exp_W2_1}
\E\big[(\hat{\theta}_n(\boldsymbol{X}) - \theta_0)^4(M(\theta_0;\boldsymbol{X}))^2\big] = \frac{4n^2}{\theta_0^6}\E\bigg[\bigg(\frac{1}{\bar{X}} - \theta_0\bigg)^4\bigg] = \frac{4n^2(3n^2 + 46n + 24)}{\theta_0^2(n-1)(n-2)(n-3)(n-4)}
\end{equation}
and
\begin{align}
\label{third_term_exp_W2_2}
\nonumber \E\big[(\hat{\theta}_n(\boldsymbol{X}) - \theta_0)^4(M(\hat{\theta}_n(\boldsymbol{X});\boldsymbol{X}))^2\big] & = 4n^2\E\bigg[\bigg(\frac{1}{\bar{X}} - \theta_0\bigg)^4\left(\bar{X}\right)^6\bigg]
 \\
\nonumber & = 4n^2\E\left[\theta_0^4(\bar{X})^6 - 4\theta_0^3(\bar{X})^5 + 6\theta_0^2(\bar{X})^4 - 4\theta_0(\bar{X})^3 + (\bar{X})^2\right]\\
& = \frac{4(283n^4 + 461n^3 + 370n^2 + 120n)}{n^4\theta_0^2}.
\end{align}
Applying now the results of \eqref{first_term_exp_W2}, \eqref{third_term_exp_W2_1} and \eqref{third_term_exp_W2_2} to \eqref{noncanexponential2} and using that the second term of the bound in \eqref{noncanexponential2} vanishes, yields result {\textbf{(2)}} of the corollary. Note that the inequality $n^{-4} \leq [(n-1)(n-2)(n-3)(n-4)]^{-1}$, for any $n>4$, has also been used.}
\end{proof}

\subsection{The normal distribution under canonical parametrisation}
\label{subsec::can_normal}
The distribution of a random variable $X$ is said to be a canonical \emph{multi-parameter exponential family distribution} if, for $\boldsymbol{\eta} \in \mathbb{R}^d$, the probability density (or mass) function takes the form
\begin{equation}
\nonumber f(x|\boldsymbol{\eta}) = \mathrm{exp}\bigg\{ \sum_{j=1}^{d}\eta_jT_j(x) - A(\boldsymbol{\eta}) + S(x)\bigg\}\mathbf{1}_{\{x \in B\}},
\end{equation}
where $B = \left\lbrace x:f(x|\boldsymbol{\eta})>0 \right\rbrace$, the support of $X$, does not depend on $\boldsymbol{\eta}$; $A(\boldsymbol{\eta})$ is a function of the parameter $\boldsymbol{\eta}$; and $T_j(x)$ and $S(x)$ are functions of only the data. 


Here, we apply Theorem \ref{Theorem_multi} in the case that $X_1, X_2, \ldots, X_n$ are i.i.d$.$ random variables following the $\mathrm{N}(\mu,\sigma^2)$ distribution, an exponential family distribution. Let
\begin{equation}
\label{natural_parameter}
\boldsymbol{\eta}_0 = \left(\eta_1,\eta_2\right)^\intercal= \left(\frac{1}{2\sigma^2},\frac{\mu}{\sigma^2}\right)^\intercal,
\end{equation}
be the natural parameter vector. The MLE for $\boldsymbol{\eta}_0$ exists, it is unique and equal to 
\[\hat{\boldsymbol{\eta}}(\boldsymbol{X}) = \left(\hat{\eta}_1, \hat{\eta}_2\right)^{\intercal} = \frac{n}{\sum_{i=1}^{n}\left(X_i - \bar{X}\right)^2}\left(\frac{1}{2}, \bar{X}\right)^{\intercal}.\]
This can be seen from the invariance property of the MLE and the result of \cite[p$.$ 116]{Davison} in which the MLEs for $\mu$ and $\sigma^2$ are given.  In Corollary \ref{thmnorcan}, we give an explicit bound on \gr{the} \gr{$1$-}Wasserstein distance between the distribution of $\hat{\boldsymbol{\eta}}(\boldsymbol{X})$ and its limiting multivariate normal distribution.  As $\hat{\boldsymbol{\eta}}(\boldsymbol{X})$ is a non-linear statistic, this result demonstrates the power of our general theorems of Section \ref{sec3}; to the best of our knowledge no other such optimal order bounds have been given for multivariate normal approximation of non-linear statistics in the \gr{1-}Wasserstein metric.


\begin{corollary}\label{thmnorcan}
Let $X_1, X_2, \ldots, X_n$ be i.i.d$.$ $\mathrm{N}(\mu, \sigma^2)$ random variables. Let $\boldsymbol{\eta}_0$ be as in \eqref{natural_parameter}, and for ease of presentation we denote $\alpha := \alpha(\eta_1,\eta_2) = \eta_1(1+\sqrt{\eta_1})^2 + \eta_2^2$. Let $\boldsymbol{W}=\sqrt{n}[I(\boldsymbol{\eta}_0)]^{1/2}(\hat{\boldsymbol{\eta}}(\boldsymbol{X}) - \boldsymbol{\eta}_0)$ and $\boldsymbol{Z} \sim {\rm MVN}(\boldsymbol{0},I_{2})$.  Then, for $n>9$,
\begin{align}
\label{cannorfinalbd}d_{\mathrm{W}}(\boldsymbol{W},\boldsymbol{Z})&<\frac{189}{\alpha\sqrt{n}}\bigg(15(1+\sqrt{\eta_1})^4(\eta_1+\eta_2^2)^2+\frac{3\eta_2^6}{\eta_1}\bigg(10+\frac{3\eta_2^2}{\eta_1}\bigg)\bigg)^{1/2}\nonumber\\
&\quad+ \frac{1}{\sqrt{2\alpha n}}(3\eta_1+4\eta_1^2+3\eta_2^2)\bigg[\frac{206}{\sqrt{\eta_1}}+\frac{1286}{\eta_1}+\frac{393|\eta_2|}{\eta_1}+\frac{1792\eta_2^2}{\eta_1^2}\bigg].
\end{align}
\end{corollary}

\begin{remark}\label{norcanrem}A $\mathcal{O}(n^{-1/2})$ bound on the distance between $\boldsymbol{W}=\sqrt{n}[I(\boldsymbol{\eta}_0)]^{1/2}(\hat{\boldsymbol{\eta}}(\boldsymbol{X}) - \boldsymbol{\eta}_0)$ and $\boldsymbol{Z}$ in the weaker $d_{0,1,2,3}$ metric was given in \cite{a18}.  Aside from being given in a stronger metric, our bound has the advantage of taking a simpler form with a better dependence on the parameters $\eta_1$ and $\eta_2$.  The numerical constants in our bound and that of \cite{a18} are of the same magnitude.  In deriving the bound (\ref{cannorfinalbd}) we made no attempt to optimise the numerical constants and instead focused on giving a clear proof and simple final bound. 
\end{remark}

The following lemma will be used in the proof of Corollary \ref{thmnorcan}.  The proof is given in Appendix \ref{appnor4}.

\begin{lemma}\label{momentlem}Let $Q_i=\hat{\eta}_i-\eta_i$, $i=1,2$.  Then, for $n>9$,
\begin{align*}&\E[Q_1^2]\leq\frac{10\eta_1^2}{n}, \quad
\E[Q_2^2]<\frac{1}{n}(6\eta_1+10\eta_2^2), \quad \E[Q_1^4]<\frac{6958\eta_1^4}{n^2},\\
& \E[Q_2^4]<\frac{1}{n^2}(5886\eta_1^2+11700\eta_2^4),\quad \E[Q_1^2Q_2^2]<\frac{\eta_1^2}{n^2}(6400\eta_1+9023\eta_2^2), 
\end{align*}
and
\begin{align*}&\E[\hat{\eta}_1^{-8}]<\frac{31}{\eta_1^8},\quad
\E[\hat{\eta}_1^{-6}]<\frac{7}{\eta_1^6},\quad
\E[\hat{\eta}_1^{-4}]<\frac{2}{\eta_1^4},\\
&\E[\hat{\eta}_2^{2}]<\eta_1+3\eta_2^2, \quad
\E[\hat{\eta}_2^{4}]<69\eta_1^2+153\eta_2^4,\\
&\E\bigg[\frac{|\hat{\eta}_2|}{\hat{\eta}_1^3}\bigg]<\frac{|\eta_2|}{\eta_1^3}, \quad
\E\bigg[\frac{\hat{\eta}_2^2}{\hat{\eta}_1^6}\bigg]<\frac{1}{\eta_1^6}(\eta_1+2\eta_2^2), \quad
\E\bigg[\frac{\hat{\eta}_2^4}{\hat{\eta}_1^8}\bigg]<\frac{2}{\eta_1^8}(\eta_1^2+2\eta_2^4).
\end{align*}

\end{lemma}

\noindent{\emph{Proof of Corollary \ref{thmnorcan}.}} The first and second-order partial derivatives of the logarithm of the normal density function are given by
\begin{align}
\label{asdf}
\nonumber & \frac{\partial}{\partial \eta_1}\log f(x_1|\boldsymbol{\eta}_0) = -x_1^2 + \frac{1}{2\eta_1} + \frac{\eta_2^2}{4\eta_1^2},\qquad \frac{\partial}{\partial \eta_2}\log f(x_1|\boldsymbol{\eta}_0) = x_1 - \frac{\eta_2}{2\eta_1},\\
\nonumber &\frac{\partial^2}{\partial\eta_1^2}\log f(x_1|\boldsymbol{\eta}_0) = -\left(\frac{1}{2\eta_1^2} + \frac{\eta_2^2}{2\eta_1^3}\right), \qquad\; \frac{\partial^2}{\partial\eta_2^2} \log f(x_1|\boldsymbol{\eta}_0) = -\frac{1}{2\eta_1},\\
&\frac{\partial^2}{\partial\eta_1\partial\eta_2} \log f(x_1|\boldsymbol{\eta}_0) = \frac{\partial^2}{\partial\eta_2\partial\eta_1} \log f(x_1|\boldsymbol{\eta}_0) = \frac{\eta_2}{2\eta_1^2}.
\end{align}
Therefore, the expected Fisher information matrix for one random variable is
\begin{equation}
\label{multi_normal_FISHER1}
I(\boldsymbol{\eta}_0) = \frac{1}{2\eta_1}\begin{pmatrix}
\frac{1}{\eta_1} + \frac{\eta_2^2}{\eta_1^2} & -\frac{\eta_2}{\eta_1}\\
-\frac{\eta_2}{\eta_1} & 1
\end{pmatrix},
\end{equation}
and simple calculations give that
\begin{equation}
\label{multi_normal_FISHER}
\nonumber \left[I(\boldsymbol{\eta}_0)\right]^{-1/2} = \tilde{V} = \sqrt{\frac{2}{\alpha}}\begin{pmatrix}
\eta_1^{3/2}\left(1+\sqrt{\eta_1}\right) & \eta_1\eta_2\\
\eta_1\eta_2 & \eta_1\left(1+\sqrt{\eta_1}\right)+\eta_2^2
\end{pmatrix},
\end{equation}
where $\alpha = \eta_1\left(1+\sqrt{\eta_1}\right)^2 + \eta_2^2$ is defined as in the statement of the corollary.  
We now set about bounding $d_{\mathrm{W}}(\boldsymbol{W},\boldsymbol{Z})$ by applying the general bound (\ref{final_bound_regression}).  To this end, we first note that $K_2(\boldsymbol{\eta}_0)=0$ due to the fact that $\EE[T_{lj}^2] = 0$, for all $l,j \in \left\lbrace 1,2\right\rbrace$. This follows from the definition of $T_{kj}$ in \eqref{cm} and the results of \eqref{asdf} and \eqref{multi_normal_FISHER1}.

We now focus on bounding $K_1(\boldsymbol{\eta}_0)$.  Let
\[R_{1,j}=\E\bigg[\bigg(\sum_{k=1}^{d}\tilde{V}_{j,k}\frac{\partial}{\partial\theta_k}\log\left(f(\boldsymbol{X}_1|\boldsymbol{\eta}_0)\right)\bigg)^4\bigg], \quad j=1,2.\]
Then
\begin{align}R_{1,1}&=\E\bigg[\bigg(\sqrt{\frac{2}{\alpha}}\eta_1^{3/2}(1+\sqrt{\eta_1})\bigg(\frac{1}{2\eta_1}+\frac{\eta_2^2}{4\eta_1^2}-X_1^2\bigg)+\sqrt{\frac{2}{\alpha}}\eta_1\eta_2\bigg(X_1-\frac{\eta_2}{2\eta_1}\bigg)\bigg)^4\bigg]\nonumber \\
\label{r1aa}&\leq \frac{32}{\alpha^2}\bigg\{\eta_1^6(1+\sqrt{\eta_1})^4\E\bigg[\bigg(X_1^2-\frac{1}{2\eta_1}-\frac{\eta_2^2}{4\eta_1^2}\bigg)^4\bigg]+\eta_1^4\eta_2^4\E\bigg[\bigg(X_1-\frac{\eta_2}{2\eta_1}\bigg)^4\bigg]\bigg\},
\end{align}
where we used the inequality $(a+b)^4\leq 8(a^4+b^4)$. In terms of the parameters $\eta_1$ and $\eta_2$, we have that $\mu=\frac{\eta_2}{2\eta_1}$ and $\sigma^2=\frac{1}{2\eta_1}$ , so that $X_1\sim \mathrm{N}(\frac{\eta_2}{2\eta_1},\frac{1}{2\eta_1})$.  Therefore
\[\E\bigg[\bigg(X_1-\frac{\eta_2}{2\eta_1}\bigg)^4\bigg]=\frac{3}{4\eta_1^2},\]
and a longer calculation using standard formulas for the lower order moments of the normal distribution gives that
\begin{align*}\E\bigg[\bigg(X_1^2-\frac{1}{2\eta_1}-\frac{\eta_2^2}{4\eta_1^2}\bigg)^4\bigg]&=\E[(X_1^2-(\sigma^2+\mu^2))^4]\\
&=60\sigma^8+240\sigma^6\mu^2+48\sigma^4\mu^4=\frac{15}{4\eta_1^4}+\frac{30\eta_2^2}{\eta_1^5}+\frac{3\eta_2^4}{4\eta_1^6}.
\end{align*}
Substituting these formulas into (\ref{r1aa}) gives that
\begin{align*}R_{1,1}&\leq\frac{32}{\alpha^2}\bigg\{\eta_1^6(1+\sqrt{\eta_1})^4\bigg(\frac{30\eta_2^2}{\eta_1^5}+\frac{3\eta_2^4}{4\eta_1^6}\bigg)+\eta_1^4\eta_2^4\cdot\frac{3}{4\eta_1^2}\bigg\}\\
&< \frac{32}{\alpha^2}(1+\sqrt{\eta_1})^4\bigg(15\eta_1^2+30\eta_1\eta_2^2+\frac{3}{2}\eta_2^4\bigg).
\end{align*}
We bound $R_{1,2}$ similarly:
\begin{align*}R_{1,2}&=\frac{4}{\alpha^2}\E\bigg[\bigg(\eta_1^{3/2}\eta_1\eta_2\bigg(\frac{1}{2\eta_1}+\frac{\eta_2^2}{4\eta_1^2}-X_1^2\bigg)+(\eta_1(1+\sqrt{\eta_1})+\eta_2^2)\bigg(X_1-\frac{\eta_2}{2\eta_1}\bigg)\bigg)^4\bigg]\\
&\leq  \frac{32}{\alpha^2}\bigg\{\eta_1^4\eta_2^4\E\bigg[\bigg(X_1^2-\frac{1}{2\eta_1}-\frac{\eta_2^2}{4\eta_1^2}\bigg)^4\bigg]+(\eta_1(1+\sqrt{\eta_1})+\eta_2^2)^4\E\bigg[\bigg(X_1-\frac{\eta_2}{2\eta_1}\bigg)^4\bigg]\bigg\}\\
&\leq\frac{32}{\alpha^2}\bigg\{\frac{\eta_2^4}{\eta_1^2}\big(15\eta_1^2+30\eta_1\eta_2^2+3\eta_2^4\big)+8(\eta_1^4(1+\sqrt{\eta_1})^4+\eta_2^8)\cdot\frac{3}{4\eta_1^2}\bigg\} \\
&=\frac{32}{\alpha^2}\bigg\{\frac{\eta_2^4}{\eta_1^2}\big(15\eta_1^2+30\eta_1\eta_2^2+9\eta_2^4\big)+6\eta_1^2(1+\sqrt{\eta_1})^4\bigg\}.
\end{align*}
Combining our bounds for $R_{1,1}$ and $R_{1,2}$ gives that
\begin{align}K_1(\boldsymbol{\eta}_0)&=14 \cdot 2^{5/4}\max_{1\leq j\leq 2}\bigg(\mathbb{E}\bigg[\bigg(\sum_{k=1}^{2}\tilde{V}_{j,k}\frac{\partial}{\partial\theta_k}\log\left(f(\boldsymbol{X}_1|\boldsymbol{\theta}_0)\right)\bigg)^4\bigg]\bigg)^{1/2}\nonumber\\
&< \frac{14\cdot2^{5/4}\cdot\sqrt{32}}{\alpha}\bigg((1+\sqrt{\eta_1})^4(15\eta_1^2+30\eta_1\eta_2^2+15\eta_2^4)+\frac{30\eta_2^6}{\eta_1}+\frac{9\eta_2^8}{\eta_1^2}\bigg)^{1/2} \nonumber\\
\label{term1bound}&<\frac{189}{\alpha}\bigg(15(1+\sqrt{\eta_1})^4(\eta_1+\eta_2^2)^2+\frac{3\eta_2^6}{\eta_1}\bigg(10+\frac{3\eta_2^2}{\eta_1}\bigg)\bigg)^{1/2}.
\end{align}
We now bound $K_3(\boldsymbol{\eta}_0)$, as given by
\begin{align}K_3(\boldsymbol{\eta}_0)&=\frac{1}{2}\sum_{k=1}^{2}\sum_{j=1}^{2}|\tilde{V}_{kj}|\sum_{l=1}^{2}\sum_{q=1}^2\sum_{\substack{\tilde{\eta}_m \in \left\lbrace\hat{\eta}_n(\boldsymbol{X})_m, \eta_{0,m}\right\rbrace\\m \in \left\lbrace 1,2\right\rbrace}}\E\big|Q_lQ_qM_{qlj}(\boldsymbol{\tilde{\eta}};\boldsymbol{X})\big|\nonumber\\
\label{term234}&=:\frac{1}{2}\sum_{k=1}^2\sum_{j=1}^2|\tilde{V}_{kj}|\sum_{l=1}^2\sum_{q=1}^2 R_{q,l,j}^{M_{qlj}}.
\end{align}
Here the superscript $M_{qlj}$ in $R_{q,l,j}^{M_{qlj}}$ emphasises the fact the quantity depends on the choice of dominating function $M_{qlj}$. In bounding $K_3(\boldsymbol{\eta}_0)$ we first note the following inequalities which will simplify the final bound:
\begin{align*}|\tilde{V}_{11}|+|\tilde{V}_{21}|&=\sqrt{\frac{2}{\alpha}}\big(\eta_1^{3/2}(1+\sqrt{\eta_1})+\eta_1|\eta_2|\big)\leq \frac{3}{2}\eta_1+2\eta_1^2+\frac{3}{2}\eta_2^2,\\
|\tilde{V}_{12}|+|\tilde{V}_{22}|&=\sqrt{\frac{2}{\alpha}}\big(\eta_1|\eta_2|+\eta_1(1+\sqrt{\eta_1})+\eta_2^2\big)\leq \frac{3}{2}\eta_1+2\eta_1^2+\frac{3}{2}\eta_2^2,
\end{align*}
which can be seen to hold from several applications of the simple inequality $ab\leq \frac{1}{2}(a^2+b^2)$.

From the formulas in (\ref{asdf}) we readily obtain that
\begin{align*}&\frac{\partial^3}{\partial \eta_1^3}\ell(\boldsymbol{\eta};\boldsymbol{x})=\frac{n}{\eta_1^3} + \frac{3n\eta_2^2}{2\eta_1^4}, \quad \frac{\partial^3}{\partial \eta_2^3}\ell(\boldsymbol{\eta};\boldsymbol{x})=0, \\
&\frac{\partial^3}{\partial \eta_1^2\partial\eta_2}\ell(\boldsymbol{\eta};\boldsymbol{x})=\frac{\partial^3}{\partial \eta_1\partial\eta_2\partial\eta_1}\ell(\boldsymbol{\eta};\boldsymbol{x})=\frac{\partial^3}{\partial \eta_2\partial\eta_1^2}\ell(\boldsymbol{\eta};\boldsymbol{x})=-\frac{n\eta_2}{\eta_1^3}, \\
&\frac{\partial^3}{\partial \eta_1\partial\eta_2^2}\ell(\boldsymbol{\eta};\boldsymbol{x})=\frac{\partial^3}{\partial \eta_2\partial\eta_1\partial\eta_2}\ell(\boldsymbol{\eta};\boldsymbol{x})=\frac{\partial^3}{\partial \eta_2^2\partial\eta_1}\ell(\boldsymbol{\eta};\boldsymbol{x})=\frac{n}{2\eta_1^2}.
\end{align*}
Therefore we can take
\begin{align*}&M_{111}(\tilde{\boldsymbol{\eta}},\boldsymbol{x})=\frac{n}{\eta_1^3} + \frac{3n\eta_2^2}{2\eta_1^4}, \quad M_{112}(\tilde{\boldsymbol{\eta}},\boldsymbol{x})=M_{121}(\tilde{\boldsymbol{\eta}},\boldsymbol{x})=M_{211}(\tilde{\boldsymbol{\eta}},\boldsymbol{x})=\frac{n|\eta_2|}{\eta_1^3},\\
& M_{122}(\tilde{\boldsymbol{\eta}},\boldsymbol{x})=M_{212}(\tilde{\boldsymbol{\eta}},\boldsymbol{x})=M_{221}(\tilde{\boldsymbol{\eta}},\boldsymbol{x})=\frac{n}{2\eta_1^2}, \quad M_{222}(\tilde{\boldsymbol{\eta}},\boldsymbol{x})=0.
\end{align*}
At this stage we note that $R_{2,2,2}^{M_{222}}=0$ and that $R_{1,2,1}^{M_{121}}=R_{1,1,2}^{M_{112}}$ and $R_{2,1,2}^{M_{212}}=R_{2,2,1}^{M_{221}}$.  Therefore we only need to bound $R_{1,1,1}^{M_{111}}$, $R_{2,1,1}^{M_{211}}$, $R_{1,1,2}^{M_{112}}$, $R_{2,1,2}^{M_{212}}$ and $R_{1,2,2}^{M_{122}}$.  In order to bound each of these terms, we must consider four cases: (A) $\tilde{\boldsymbol{\eta}}=(\eta_1, \eta_2)$, (B) $\tilde{\boldsymbol{\eta}}=(\hat{\eta_1}, \eta_2)$, (C) $\tilde{\boldsymbol{\eta}}=(\eta_1, \hat{\eta_2})$ and (D) $\tilde{\boldsymbol{\eta}}=(\hat{\eta_1}, \hat{\eta_2})$.  It will be convenient to write $R_{1,1,1}^{M_{111},A}=\E\big|Q_lQ_qM_{qlj}((\eta_1,\eta_2);\boldsymbol{X})\big|$, with the notation $R_{1,1,1}^{M_{111},B}$, $R_{1,1,1}^{M_{111},C}$ and $R_{1,1,1}^{M_{111},D}$ defined in the obvious manner. 

We first bound $R_{1,1,1}^{M_{111}}$.  We consider the four case (A), (B), (C) and (D), and bound the terms by using the Cauchy-Schwarz inequality and the bounds of Lemma \ref{momentlem}: 
\begin{align*}R_{1,1,1}^{M_{111},A}=\E\bigg[Q_1^2\bigg(\frac{n}{\eta_1^3} + \frac{3n\eta_2^2}{2\eta_1^4}\bigg)\bigg]\leq\frac{1}{\eta_1^2}(15\eta_1+3\eta_2^2),
\end{align*}
\begin{align*}R_{1,1,1}^{M_{111},B}=\E\bigg[Q_1^2\bigg(\frac{n}{\hat{\eta}_1^3} + \frac{3n\eta_2^2}{2\hat{\eta}_1^4}\bigg)\bigg]\leq n\sqrt{\E[Q_1^4]\E[\hat{\eta}_1^{-6}]}+\frac{3\eta_2^2n}{2}\sqrt{\E[Q_1^4]\E[\hat{\eta}_1^{-8}]}<\frac{1}{\eta_1^2}(221\eta_1+126\eta_2^2),
\end{align*}
\begin{align*}R_{1,1,1}^{M_{111},C}&=\E\bigg[Q_1^2\bigg(\frac{n}{\eta_1^3} + \frac{3n\hat{\eta}_2^2}{2\eta_1^4}\bigg)\bigg]\leq \frac{n}{\eta_1^3}\E[Q_1^2]+\frac{3n}{2\eta_1^4}\sqrt{\E[Q_1^4]\E[\hat{\eta}_2^{4}]}\\
&<\frac{10}{\eta_1}+\frac{3}{2\eta_1^4}\sqrt{6958\eta_1^4(69\eta_1^2+153\eta_2^4)}<\frac{1}{\eta_1^2}(1050\eta_1+1548\eta_2^2),
\end{align*}
\begin{align*}R_{1,1,1}^{M_{111},D}&=\E\bigg[Q_1^2\bigg(\frac{n}{\hat{\eta}_1^3} + \frac{3n\hat{\eta}_2^2}{2\hat{\eta}_1^4}\bigg)\bigg]\leq \sqrt{\E[Q_1^4]\E[\hat{\eta}_1^{-6}]}+\frac{3}{2}\sqrt{\E[Q_1^4]\E\bigg[\frac{\hat{\eta}_2^4}{\hat{\eta}_1^8}\bigg]} \\
&<\frac{1}{\eta_1}\sqrt{6958\times 7}+\frac{3}{2\eta_1^2}\sqrt{6958\times2(\eta_1^2+2\eta_2^4)}<\frac{1}{\eta_1^2}(398\eta_1+251\eta_2^2).
\end{align*}
Thus,
\[R_{1,1,1}^{M_{111}}<\frac{1684}{\eta_1}+\frac{1928\eta_2^2}{\eta_1^2}.\]
Similar calculations (which are given in Appendix \ref{appnor4}) show that
\begin{align*}&R_{2,1,1}^{M_{211}}<\frac{168}{\sqrt{\eta_1}}+\frac{494|\eta_2|}{\eta_1},\quad R_{1,1,2}^{M_{112}}<\frac{386}{\eta_1}+\frac{746\eta_2^2}{\eta_1^2}, \\
&R_{2,1,2}^{M_{212}}<\frac{122}{\sqrt{\eta_1}}+\frac{146|\eta_2|}{\eta_1},\quad
R_{1,2,2}^{M_{122}}<\frac{116}{\eta_1}+\frac{164\eta_2^2}{\eta_1^2}.
\end{align*}
Applying these bounds to (\ref{term234}) yields the following bound: 
\begin{align}
K_3(\boldsymbol{\eta}_0)&\leq\frac{1}{2\sqrt{2\alpha}}(3\eta_1+4\eta_1^2+3\eta_2^2)\bigg[\bigg(\frac{1684}{\eta_1}+\frac{1928\eta_2^2}{\eta_1^2}\bigg)+\bigg(\frac{386}{\eta_1}+\frac{746\eta_2^2}{\eta_1^2}\bigg)+\bigg(\frac{168}{\sqrt{\eta_1}}+\frac{494|\eta_2|}{\eta_1}\bigg)\nonumber\\
&\quad+\bigg(\frac{122}{\sqrt{\eta_1}}+\frac{146|\eta_2|}{\eta_1}\bigg)+\bigg(\frac{386}{\eta_1}+\frac{746\eta_2^2}{\eta_1^2}\bigg)+\bigg(\frac{116}{\eta_1}+\frac{164\eta_2^2}{\eta_1^2}\bigg)+\bigg(\frac{122}{\sqrt{\eta_1}}+\frac{146|\eta_2|}{\eta_1}\bigg)\bigg]\nonumber\\
\label{term2bound}&\quad=\frac{1}{\sqrt{2\alpha}}(3\eta_1+4\eta_1^2+3\eta_2^2)\bigg[\frac{206}{\sqrt{\eta_1}}+\frac{1286}{\eta_1}+\frac{393|\eta_2|}{\eta_1}+\frac{1792\eta_2^2}{\eta_1^2}\bigg].
\end{align}
Finally, summing up the bounds (\ref{term1bound}) and (\ref{term2bound}) completes the proof. \hfill $\Box$

\subsection{The multivariate normal distribution under non-canonical parametrisation}\label{sec4.3nor}
\subsubsection{Diagonal covariance matrix}
\label{subsec::diagonal_cov_matrix}
Let $\boldsymbol{X}_1,\ldots,\boldsymbol{X}_n$ be i.i.d$.$ $\mathrm{MVN}(\boldsymbol{\mu},\Sigma)$ random variables, where $\boldsymbol{\mu}=(\mu_1,\ldots,\mu_p)^{\intercal}$ and $\Sigma=\mathrm{diag}(\sigma_1^2,\ldots,\sigma_p^2)$.  Here $\boldsymbol{\theta}_0=(\mu_1,\ldots,\mu_p,\sigma_1^2,\ldots,\sigma_p^2)^{\intercal}$.  The density function here is
\begin{equation*}f(\boldsymbol{x}|\boldsymbol{\theta})=\frac{1}{(2\pi)^{p/2}\sqrt{\sigma_1^2\cdots\sigma_p^2}}\exp\bigg\{-\sum_{i=1}^p\frac{(x_i-\mu_i)^2}{2\sigma_i^2}\bigg\}, \quad \boldsymbol{x}=(x_1,\ldots,x_p)^{\intercal}\in\mathbb{R}^p.
\end{equation*}
For $1\leq j\leq p$, let $\bar{X}_j$ \gr{denote the sample mean of $X_{1,j},\ldots,X_{n,j}$}. Then it is well-known in this case that the MLE is unique and equal to
\begin{equation*}\hat{\boldsymbol{\theta}}_n(\boldsymbol{X})=\bigg(\bar{X}_1,\ldots\bar{X}_p,\frac{1}{n}\sum_{i=1}^n(X_{i,1}-\bar{X}_1)^2,\ldots,\frac{1}{n}\sum_{i=1}^n(X_{i,p}-\bar{X}_p)^2\bigg)^{\intercal}.
\end{equation*}

Let $\boldsymbol{W}=\sqrt{n}[I(\boldsymbol{\theta}_0)]^{1/2}\big(\hat{\boldsymbol{\theta}}_n(\boldsymbol{X})-\boldsymbol{\theta}_0\big)$.  Then it is readily checked that all the assumptions of Theorem \ref{Theorem_multi} are met and so an application of the bound (\ref{final_bound_regression}) would yield a bound of the form $d_{\mathrm{W}}(\boldsymbol{W},\boldsymbol{Z})\leq Cn^{-1/2}$, where $\boldsymbol{Z}\sim\mathrm{MVN}(\boldsymbol{0},I_{2p})$, for some constant $C$ that does not depend on $n$.  However, the term $K_3(\boldsymbol{\theta}_0)$ \gr{has} a very poor dependence on the dimension $d$ and would be tedious to compute.  Instead, we take advantage of the particular representation of the MLE to derive a neat optimal $\mathcal{O}(n^{-1/2})$ \gr{1-}Wasserstein distance \gr{(and 2-Wasserstein distance)} bound with good dependence on the dimension.  \gr{In deriving this bound we make use of Theorem \ref{bonisthm}.}

\begin{theorem}\label{thmdiag}  Let $\boldsymbol{X}_1,\ldots,\boldsymbol{X}_n$ be i.i.d$.$ $\mathrm{MVN}(\boldsymbol{\mu},\Sigma)$ random vectors, where $\boldsymbol{\mu}=(\mu_1,\ldots,\mu_p)^{\intercal}$ and $\Sigma=\mathrm{diag}(\sigma_1^2,\ldots,\sigma_p^2)$.  
Let $\boldsymbol{W}=\sqrt{n}[I(\boldsymbol{\theta}_0)]^{1/2}\big(\hat{\boldsymbol{\theta}}_n(\boldsymbol{X})-\boldsymbol{\theta}_0\big)$ and $\boldsymbol{Z}\sim\mathrm{MVN}(\boldsymbol{0},I_{2p})$.  Then
\begin{equation}\label{bound78}d_{\mathrm{W}}(\boldsymbol{W},\boldsymbol{Z})\gr{\leq d_{\mathrm{W}_2}(\boldsymbol{W},\boldsymbol{Z})<56\sqrt{\frac{p}{n}}}.
\end{equation}
\end{theorem}

\begin{remark}\label{remcanex}Corollary 3.1 of \cite{ag19} gave a bound in the weaker $d_{1,2}$ metric for the case that $X_1,\ldots,X_n$ are i.i.d$.$ $\mathrm{N}(\mu,\sigma^2)$ random variables.  Theorem \ref{thmdiag} generalises the setting from $p=1$ to $p\geq1$ and gives a bound in the stronger \gr{1-}Wasserstein distance. \gr{Our bound shows that the MLE converges in distribution to the multivariate normal distribution for even large $p$ provided $p\ll n$.  We believe that the dependence on the dimension $p$ in our bound is optimal, and this seems to be supported by empirical results in Section \ref{sec4.56}.}
\end{remark}

\begin{proof}\gr{The inequality $d_{\mathrm{W}}(\boldsymbol{W},\boldsymbol{Z})\leq d_{\mathrm{W}_2}(\boldsymbol{W},\boldsymbol{Z})$ is immediate from (\ref{pqpq}), and the rest of the proof is devoted to bounding $d_{\mathrm{W}_2}(\boldsymbol{W},\boldsymbol{Z})$. We begin by recalling} the standard result that the expected Fisher information matrix is given by 
\[I(\boldsymbol{\theta}_0)=\mathrm{diag}\bigg(\frac{1}{\sigma_1^2},\ldots,\frac{1}{\sigma_p^2},\frac{1}{2\sigma_1^4},\ldots,\frac{1}{2\sigma_p^4}\bigg),\]
and therefore
\[[I(\boldsymbol{\theta}_0)]^{1/2}=\mathrm{diag}\bigg(\frac{1}{\sigma_1},\ldots,\frac{1}{\sigma_p},\frac{1}{\sqrt{2}\sigma_1^2},\ldots,\frac{1}{\sqrt{2}\sigma_p^2}\bigg).\]
Now, for $1\leq i\leq n$, write $\boldsymbol{X}_i=(X_{i,1},\ldots,X_{i,p})^{\intercal}$, and define the standardised random variables $Y_{i,j}=(X_{i,j}-\mu_j)/\sigma_j$, $1\leq i\leq n$, $1\leq j\leq p$.  For $1\leq j\leq p$, let $\bar{X}_j$ and $\bar{Y}_j$ denote the sample means of $X_{1,j},\ldots,X_{n,j}$ and $Y_{1,j},\ldots,Y_{n,j}$. A simple calculation gives the useful equation
\[\sum_{i=1}^n(X_{i,j}-\bar{X}_j)^2=\sum_{i=1}^n(X_{i,j}-\mu_{j})^2-n(\bar{X}_j-\mu_j)^2.\]
Putting all this together gives that $\boldsymbol{W}$ can be written as $\boldsymbol{W}=(W_{1},\ldots,W_{2p})^{\intercal}$, where, for $1\leq j\leq p$,
\begin{align*}W_{j}=\frac{1}{\sqrt{n}}\sum_{i=1}^n\frac{X_{i,j}-\mu_j}{\sigma_j}=\frac{1}{\sqrt{n}}\sum_{i=1}^n Y_{i,j}
\end{align*}
and
\begin{align*}W_{j+p}&=\frac{1}{\sqrt{n}}\sum_{i=1}^n\frac{(X_{i,j}-\mu_j)^2-\sigma_j^2}{\sqrt{2}\sigma_j^2}-\sqrt{n}\frac{(\bar{X}_j-\mu_j)^2}{\sqrt{2}\sigma_j^2} =\frac{1}{\sqrt{n}}\sum_{i=1}^n\frac{Y_{i,j}^2-1}{\sqrt{2}}-\frac{\sqrt{n}}{\sqrt{2}}(\bar{Y}_{j})^2.
\end{align*}
It will be useful to define $\boldsymbol{V}=(V_{1},\ldots,V_{2p})^{\intercal}$, where, for $1\leq j\leq p$, 
\[V_{j}=W_{j} \quad \text{and} \quad V_{j+p}=\frac{1}{\sqrt{n}}\sum_{i=1}^n\frac{Y_{i,j}^2-1}{\sqrt{2}}.\]

\gr{We now note that $\bar{X}_1,\ldots\bar{X}_p,\frac{1}{n}\sum_{i=1}^n(X_{i,1}-\bar{X}_1)^2,\ldots,\frac{1}{n}\sum_{i=1}^n(X_{i,p}-\bar{X}_p)^2$ are independent (see Section 3b.3 of \cite{r73}), from which it follows that $W_1,\ldots,W_{2p}$ are independent. As the infimum in the definition (\ref{pwasdefn}) of the 2-Wasserstein distance is attained, for each $j=1,\ldots,2p$ we may construct a probability space on which the random variables $W_j^*$ and $Z_j^*$ with $\mathcal{L}(W_j^*)=\mathcal{L}(W_j)$ and $\mathcal{L}(Z_j^*)=\mathcal{L}(Z_j)$ are such that $d_{\mathrm{W}_2}(W_j,Z_j)=\sqrt{\mathbb{E}[(W_j^*-Z_j^*)^2]}$. By independence, on taking the product of these probabilities spaces, we can construct random vectors $\boldsymbol{W}^*=(W_1^*,\ldots,W_{2p}^*)^\intercal$ and $\boldsymbol{Z}^*=(Z_1^*,\ldots,Z_{2p}^*)^\intercal$ with $\mathcal{L}(\boldsymbol{W}^*)=\mathcal{L}(\boldsymbol{W})$ and $\mathcal{L}(\boldsymbol{Z}^*)=\mathcal{L}(\boldsymbol{Z})$ such that   $d_{\mathrm{W}_2}(\boldsymbol{W},\boldsymbol{Z})=\sqrt{\mathbb{E}[|\boldsymbol{W}^*-\boldsymbol{Z}^*|^2]}$. Therefore
\begin{align}\label{nearr}d_{\mathrm{W}_2}(\boldsymbol{W},\boldsymbol{Z})=\sqrt{\mathbb{E}[|\boldsymbol{W}^*-\boldsymbol{Z}^*|^2]}=\sqrt{\sum_{j=1}^{2p}\mathbb{E}[(W_j^*-Z_j^*)^2]}=\sqrt{\sum_{j=1}^{2p}d_{\mathrm{W}_2}(W_j,Z_j)^2}.
\end{align}
For $j=1,\ldots,p$, $W_j\sim \mathrm{N}(0,1)$, and so $d_{\mathrm{W}_2}(W_j,Z_j)=0$ for $j=1,\ldots,p$. Now suppose $j\in\{p+1,\ldots,2p\}$. Then, by the triangle inequality,
\begin{equation}\label{tri22}d_{\mathrm{W}_2}(W_j,Z_j)\leq d_{\mathrm{W}_2}(W_j,V_j)+d_{\mathrm{W}_2}(V_j,Z_j).
\end{equation}
By the definition of the 2-Wasserstein distance,
\begin{align*}d_{\mathrm{W}_2}(W_j,Z_j)&\leq\sqrt{\mathbb{E}\bigg[\bigg(\bigg(\sum_{i=1}^n\frac{Y_{ij}^2-1}{\sqrt{2}}-\frac{\sqrt{n}}{\sqrt{2}}(\bar{Y}_j)^2\bigg)-\sum_{i=1}^n\frac{Y_{ij}^2-1}{\sqrt{2}}\bigg)^2\bigg]}=\frac{\sqrt{n}}{\sqrt{2}}\sqrt{\mathbb{E}[(\bar{Y}_j)^2]}=\sqrt{\frac{3}{2n}},
\end{align*}
where we used that $\bar{Y}_j\sim\mathrm{N}(0,\frac{1}{n})$, so that $\mathbb{E}[(\bar{Y}_j)^4]=\frac{3}{n^2}$.

To bound $d_{\mathrm{W}_2}(V_j,Z_j)$, we apply Theorem \ref{bonisthm} in the univariate case $d=1$. We can write $V_j=\frac{1}{\sqrt{n}}\sum_{i=1}^n\xi_{i,j}$, where $\xi_{1,j},\ldots,\xi_{n,j}$ are i.i.d$.$ random variables with $\xi_{i,j}=\frac{1}{\sqrt{2}}(Y_{i,1}^2-1)$, $i=1,\ldots,n$.
We note that that the assumptions $\mathbb{E}[\xi_{1,j}] = 0$ and $\mathbb{E}[\xi_{1,j}^2] =1$ are satisfied. Applying the bound (\ref{bonisbound}) of Theorem \ref{bonisthm} now yields, for $j=p+1,\ldots,2p$,
\begin{align*}d_{\mathrm{W}_2}(V_j,Z_j)\leq\frac{14}{\sqrt{n}}\sqrt{\mathbb{E}[\xi_{1,j}^4]}=\frac{7}{\sqrt{n}}\sqrt{\mathbb{E}[(Y_{1,j}^2-1)^4]}=\frac{7}{\sqrt{n}}\sqrt{\mathbb{E}[(Z^2-1)^4]}=\frac{14\sqrt{15}}{\sqrt{n}},
\end{align*}
where we used that $Y_{1,j}=_d Z\sim \mathrm{N}(0,1)$, and the final equality follows from an application of standard formulas for moments of the normal distribution. Substituting our bounds for $d_{\mathrm{W}_2}(W_j,V_j)$ and $d_{\mathrm{W}_2}(V_j,Z_j)$ into (\ref{tri22}) gives that, for $j=p+1,\ldots,2p$, 
\begin{equation*}d_{\mathrm{W}_2}(W_j,Z_j)\leq\bigg(\sqrt{\frac{3}{2}}+14\sqrt{15}\bigg)\frac{1}{\sqrt{n}},
\end{equation*}
and plugging this bound into (\ref{nearr}) yields
\[d_{\mathrm{W}_2}(\boldsymbol{W},\boldsymbol{Z})\leq\sqrt{p\bigg(\sqrt{\frac{3}{2}}+14\sqrt{15}\bigg)^2\frac{1}{n}}<56\sqrt{\frac{p}{n}},\]
as required.}
\end{proof}

\subsubsection{The general case}

Let $\boldsymbol{X}_1,\ldots,\boldsymbol{X}_n$ be i.i.d$.$ $\mathrm{MVN}(\boldsymbol{\mu},\Sigma)$ random vectors, where $\boldsymbol{\mu}=(\mu_1,\ldots,\mu_p)^{\intercal}$ and $\Sigma=(\sigma_{i,j})$.  Here $\boldsymbol{\theta}_0=(\mu_1,\ldots,\mu_p,\sigma_{1,1},\ldots,\sigma_{1,p},\ldots\sigma_{p,1},\ldots,\sigma_{p,p})^{\intercal}$.
The density function here is
\begin{equation*}f(\boldsymbol{x}|\boldsymbol{\theta})=\frac{1}{(2\pi)^{p/2}\sqrt{\mathrm{det}(\Sigma)}}\exp\bigg\{-\frac{1}{2}(\boldsymbol{x}-\boldsymbol{\mu})^{\intercal}\Sigma^{-1} (\boldsymbol{x}-\boldsymbol{\mu})\bigg\}, \quad \boldsymbol{x}=(x_1,\ldots,x_p)^{\intercal}\in\mathbb{R}^p.
\end{equation*}
It is well-known in this case that the MLE is unique and equal to
$\hat{\boldsymbol{\theta}}_n(\boldsymbol{X})=\big(\bar{\boldsymbol{X}},\frac{1}{n}\sum_{i=1}^n(\boldsymbol{X}_i-\bar{\boldsymbol{X}})(\boldsymbol{X}_i-\bar{\boldsymbol{X}})^{\intercal}\big)^{\intercal}.
$
Since the covariance matrix $\Sigma$ and its MLE estimator $\hat{\Sigma}$ are symmetric, for the purpose of presenting a multivariate normal approximation for the MLE we restrict $\boldsymbol{\theta}_0$ to only include $\sigma_{i,j}$, $i\geq j$, and $\hat{\boldsymbol{\theta}}_n(\boldsymbol{X})$ to only include the estimators $\hat{\sigma}_{i,j}$, $i\geq j$.  This restricted MLE has $p+\binom{p}{2}=p(p+3)/2$ parameters.  As in diagonal case, we could apply Theorem \ref{Theorem_multi} to obtain a optimal order $\mathcal{O}(n^{-1/2})$ \gr{1-}Wasserstein distance bound, but we prefer to proceed as we did there and exploit the particular representation of the MLE in deriving our bound.


The proof of the following theorem follows a similar \gr{basic} approach to that of Theorem \ref{thmdiag}\gr{, again making use of Theorem \ref{bonisthm}, although as the components of the random vector $\boldsymbol{W}$ are now no longer independent our calculations are a little more involved, as we cannot reduce the problem to the univariate case as we did in proving Theorem \ref{thmdiag}. We defer the proof} to Appendix \ref{appnorgen}. For a matrix $A$, let $\|A\|_{\mathrm{max}}=\max_{i,j}|a_{i,j}|$.


\begin{theorem}\label{thmnondiag}  Let $\boldsymbol{X}_1,\ldots,\boldsymbol{X}_n$ be i.i.d$.$ $\mathrm{MVN}(\boldsymbol{\mu},\Sigma)$ random vectors, where $\boldsymbol{\mu}=(\mu_1,\ldots,\mu_p)^{\intercal}$ and  $\Sigma=(\sigma_{ij})\in\mathbb{R}^{p\times p}$ is positive semi-definite.  
Let $\hat{\boldsymbol{\theta}}_n(\boldsymbol{X})$ be the MLE restricted in the manner as described above.  Let $\boldsymbol{W}=\sqrt{n}[I(\boldsymbol{\theta}_0)]^{1/2}\big(\hat{\boldsymbol{\theta}}_n(\boldsymbol{X})-\boldsymbol{\theta}_0\big)$ and $\boldsymbol{Z}\sim\mathrm{MVN}(\boldsymbol{0},I_{p(p+3)/2})$.  Write $\sigma_*^2=\max_{1\leq j\leq p}\sigma_{jj}$ (the largest variance in the covariance matrix $\Sigma$).  Then
\begin{equation*}d_{\mathrm{W}}(\boldsymbol{W},\boldsymbol{Z})<\frac{1}{\sqrt{n}}\Big(p^4\sigma_*^2\|[I(\boldsymbol{\theta}_0)]^{1/2}\|_{\mathrm{max}}+15.1\,p^{13/4}(p+3)^{13/4}\sigma_*^4\|[I(\boldsymbol{\theta}_0)]^{1/2}\|_{\mathrm{max}}^2\Big).
\end{equation*}
\end{theorem}


\gdr{\subsection{Implicitly defined MLEs}
\label{sec:implicit_MLE}
In order to be calculated, the general upper bound on the 1-Wasserstein distance of interest, as expressed in Theorem \ref{Theorem_multi}, requires a closed-form expression for the MLE. In this section, we explain how an upper bound on the weaker bounded Wasserstein distance can be obtained when the MLE is implicitly defined. Our strategy is split into two steps; first, put the dependence of the bound on the MLE only through the mean squared error (MSE), $\EE[\sum_{j=1}^{d}Q_j^2]$  with $Q_j$ as in \eqref{cm}, and secondly discuss how upper bounds can be obtained for the MSE. In addition to the regularity conditions needed in Theorem \ref{Theorem_multi}, in order to attain an upper bound on the bounded Wasserstein distance 
when the MLE is not expressed analytically, we
replace assumption (R.C.4") by 
(Con.1) as below:

\vspace{3mm}

\begin{itemize}[leftmargin=0.53in]
\item [(Con.1)] For $\epsilon > 0$ and for all $\boldsymbol{\theta}_0 \in \Theta$,
\begin{equation}
\label{M_ijk_implicit}
\sup_{\substack{\boldsymbol{\theta}:\left|\theta_q - \theta_{0,q}\right| < \epsilon\\ \forall q \in \left\lbrace 1,2,\ldots, d\right\rbrace}}\left| \frac{\partial{^3}}{\partial \theta_k\partial \theta_j\partial \theta_i}\log f(\boldsymbol{x}_1|\boldsymbol{\theta})\right| \leq M_{kji},
\end{equation}
where $M_{kji} = M_{kji}(\boldsymbol{\theta}_0)$  only depends on $\boldsymbol{\theta}_0$.
\end{itemize}
Theorem \ref{Theorem_multi} provides an upper bound on the 1-Wasserstein distance between the distribution of the MLE and the multivariate normal distribution. In Proposition \ref{prop_implicit_MLE} below, we put the dependence of the 
upper bound in \eqref{final_bound_regression} 
on the MLE only through the MSE, $\EE[\sum_{j=1}^dQ_j^2]$.
\begin{proposition}
\label{prop_implicit_MLE}
Let $\boldsymbol{X}=(\boldsymbol{X}_1, \boldsymbol{X}_2, \ldots, \boldsymbol{X}_n)$ be i.i.d$.$ $\mathbb{R}^t$-valued, $t \in \mathbb{Z}^+$, random vectors with probability density (or mass) function $f(\boldsymbol{x}_i|\boldsymbol{\theta})$, for which the true parameter value is $\boldsymbol{\theta}_0$ and the parameter space $\Theta$ is an open subset of $\mathbb{R}^d$. Assume that the MLE exists and is unique, but cannot be expressed in a closed-form, and that (R.C.1)--(R.C.3) and (Con.1) are satisfied. In addition, for $\tilde{V}$ as in \eqref{cm}, assume that $\E[|\tilde{V}\nabla\left(\log\left(f(\boldsymbol{X}_1|\boldsymbol{\theta}_0)\right)\right)|^4] < \infty$, where  $\nabla=\big(\frac{\partial}{\partial \theta_1},\ldots,\frac{\partial}{\partial \theta_d}\big)^\intercal$. Then, for $\epsilon > 0$ being a positive constant, as in (Con.1), that need not depend
on the sample size $n$, and with $\boldsymbol{W}$ as in \eqref{cm},
\begin{align}
\label{final_bound_implicit}
\nonumber d_{\mathrm{bW}}\left(\boldsymbol{W}, \boldsymbol{Z}\right) \leq & \frac{1}{\sqrt{n}}K_{1}(\boldsymbol{\theta}_0) + \sqrt{d}\sum_{k=1}^{d}\sum_{l=1}^{d}|\tilde{V}_{lk}|\sqrt{\sum_{i=1}^{d}{\rm Var}\bigg(\frac{\partial^2}{\partial \theta_k\partial \theta_i}\log f(\boldsymbol{X}_1|\boldsymbol{\theta}_0)\bigg)}\sqrt{\EE\bigg[\sum_{j=1}^{d}Q_j^2\bigg]}\\
& + \frac{2}{\epsilon^2}\EE\bigg[\sum_{j=1}^{d}Q_j^2\bigg] + \frac{\sqrt{n}}{2}\sum_{k=1}^{d}\sum_{l=1}^{d}|\tilde{V}_{lk}|\sum_{m=1}^{d}\sum_{i=1}^{d}M_{kmi}\EE\bigg[\sum_{j=1}^{d}Q_j^2\bigg].
\end{align}
where $K_1(\boldsymbol{\theta}_0)$ is as in \eqref{notation_Ki}.
\end{proposition}

\begin{remark}There is a well-developed theory to verify the bound
\[\sup_n\mathbb{E}[|\sqrt{n}(\hat{\boldsymbol{\theta}}_n(\boldsymbol{X})-\boldsymbol{\theta}_0)|^p]<\infty\]
for any $p>0$ in general settings (see Chapter III, Sections 1 and 3 of \cite{ih81}, and Sections 3--4 of \cite{y11}). Using such results, we can deduce that the bound \eqref{final_bound_implicit} is of the optimal order $\mathcal{O}\left(n^{-1/2}\right)$; notice that the positive constant $\epsilon$ need
not depend on $n$ and its choice could be optimised in examples. In addition, we note that the bound \eqref{final_bound_implicit} has a better dependence on the dimension $d$ than the 1-Wasserstein distance bound of Theorem \ref{Theorem_multi}.
To be more precise, assuming that $\tilde{V}_{lk}=\mathcal{O}(1)$ and $M_{kmi}=\mathcal{O}(1)$ 
it can be seen that \eqref{final_bound_implicit} is of order $\mathcal{O}(d^5)$, while the 1-Wasserstein distance bound \eqref{final_bound_regression} 
is of the much larger order $\mathcal{O}(d^{4}2^d)$.
\end{remark}
\begin{remark}
Condition (Con.1) in \eqref{M_ijk_implicit} is non-restrictive and is satisfied by various distributions for which the MLE of their parameters cannot be expressed analytically. Here, we give two examples:
\begin{enumerate}
\item {\textbf{Gamma distribution:}} With $\alpha, \beta > 0$ and $\boldsymbol{\theta} = (\alpha, \beta)^{\intercal}$ being the vector parameter, the probability density function is $f(x|\boldsymbol{\theta}) = \frac{\beta^\alpha}{\Gamma(\alpha)}x^{\alpha - 1}{\rm e}^{-\beta x}$, $x>0$.
We have that
\begin{align}
\label{derivatives_Gamma}
\nonumber & \frac{\partial^{j+1}}{\partial\alpha^{j+1}}\log f(x|\boldsymbol{\theta}) = -\psi_{j}(\alpha), \forall j \in \mathbb{Z}^{+}, \quad \frac{\partial^3}{\partial\beta^3}\log f(x|\boldsymbol{\theta}) = \frac{2\alpha}{\beta^3}, \\
& \frac{\partial^3}{\partial\alpha^2\partial\beta}\log f(x|\boldsymbol{\theta}) = 0, \quad \frac{\partial^3}{\partial\alpha\partial\beta^2}\log f(x|\boldsymbol{\theta}) = -\frac{1}{\beta^2},
\end{align}
where, for any $z\in\mathbb{C}\setminus\{0,-1,-2,\ldots\}$, the polygamma function $\psi_{m}(z)$ is defined by $\psi_m(z):=\frac{\mathrm{d}^m}{\mathrm{d}z^m}(\psi(z))$, with $\psi(z)=\frac{\mathrm{d}}{\mathrm{d}z}(\log\Gamma(z))$ denoting the digamma function. The polygamma function has the series representation (differentiate both sides of formula 5.15.1 of \cite{olver})
\begin{equation}
\label{psi_j}
\psi_m(z) = (-1)^{m+1}m!\sum_{k=0}^{\infty}\frac{1}{(z+k)^{m+1}},
\end{equation}
which holds for any $m\geq1$ and any $z\in\mathbb{C}\setminus\{0,-1,-2,\ldots\}$. It is easy to see that for $x > 0$, $|\psi_{2}(x)|$ is a decreasing function of $x$ and, using \eqref{derivatives_Gamma}, (Con.1) is satisfied with $M_{112} = 0$ and
\begin{align}
\nonumber & \sup_{\substack{\boldsymbol{\theta}:\left|\theta_q - \theta_{0,q}\right| < \epsilon\\ \forall q \in \left\lbrace 1,2\right\rbrace}}\left| \frac{\partial{^3}}{\partial \theta_1^3}\log f(x_1|\boldsymbol{\theta})\right| \leq |\psi_2(\alpha - \epsilon)| = M_{111},\\
\nonumber & \sup_{\substack{\boldsymbol{\theta}:\left|\theta_q - \theta_{0,q}\right| < \epsilon\\ \forall q \in \left\lbrace 1,2\right\rbrace}}\left| \frac{\partial{^3}}{\partial \theta_2^3}\log f(x_1|\boldsymbol{\theta})\right| \leq \frac{2(\alpha + \epsilon)}{(\beta - \epsilon)^3} = M_{222},\\
\nonumber & \sup_{\substack{\boldsymbol{\theta}:\left|\theta_q - \theta_{0,q}\right| < \epsilon\\ \forall q \in \left\lbrace 1,2\right\rbrace}}\left| \frac{\partial{^3}}{\partial \theta_1\partial\theta_2^2}\log f(x_1|\boldsymbol{\theta})\right| \leq \frac{1}{(\beta-\epsilon)^2} = M_{122}.
\end{align}
\item {\textbf{Beta distribution:}} The probability density function is
\begin{equation}
\nonumber f(x|\boldsymbol{\theta}) = \frac{\Gamma(\alpha+\beta)}{\Gamma(\alpha)\Gamma(\beta)}x^{\alpha-1}(1-x)^{\beta-1},
\end{equation}
with $\alpha, \beta >0$ and $x \in (0,1)$. Hence, for $j,k\in \mathbb{Z^+}$
\begin{align}
\label{multi_Beta_likelihood}
\nonumber & \frac{\partial^{j+1}}{\partial\alpha^{j+1}} \log f(x|\boldsymbol{\theta}) = \psi_{j}(\alpha + \beta) - \psi_{j}(\alpha),\\
\nonumber & \frac{\partial^{j+1}}{\partial\beta^{j+1}} \log f(x|\boldsymbol{\theta}) = \psi_{j}(\alpha + \beta) - \psi_{j}(\beta),\\
& \frac{\partial^{k+j}}{\partial \alpha^k \partial \beta^j}\log f(x|\boldsymbol{\theta}) = \psi_{k+j-1}(\alpha + \beta),
\end{align}
where as in the case of the gamma distribution, 
$\psi_{j}(\cdot)$ is the polygamma function defined in \eqref{psi_j}. (Con.1) is again satisfied with
\begin{align}
\nonumber & \sup_{\substack{\boldsymbol{\theta}:\left|\theta_q - \theta_{0,q}\right| < \epsilon\\ \forall q \in \left\lbrace 1,2\right\rbrace}}\left| \frac{\partial{^3}}{\partial \theta_1^3}\log f(x_1|\boldsymbol{\theta})\right| \leq |\psi_2(\alpha + \beta - 2\epsilon)| + |\psi_2(\alpha - \epsilon)| = M_{111},\\
\nonumber & \sup_{\substack{\boldsymbol{\theta}:\left|\theta_q - \theta_{0,q}\right| < \epsilon\\ \forall q \in \left\lbrace 1,2\right\rbrace}}\left| \frac{\partial{^3}}{\partial \theta_2^3}\log f(x_1|\boldsymbol{\theta})\right| \leq |\psi_2(\alpha + \beta - 2\epsilon)| + |\psi_2(\beta - \epsilon)| = M_{222},\\
\nonumber & \sup_{\substack{\boldsymbol{\theta}:\left|\theta_q - \theta_{0,q}\right| < \epsilon\\ \forall q \in \left\lbrace 1,2\right\rbrace}}\left| \frac{\partial{^3}}{\partial \theta_1\partial\theta_2^2}\log f(x_1|\boldsymbol{\theta})\right| = \sup_{\substack{\boldsymbol{\theta}:\left|\theta_q - \theta_{0,q}\right| < \epsilon\\ \forall q \in \left\lbrace 1,2\right\rbrace}}\left| \frac{\partial{^3}}{\partial \theta_1^2\partial\theta_2}\log f(x_1|\boldsymbol{\theta})\right|\\
\nonumber & \qquad\qquad\qquad\qquad\qquad\qquad\quad\;\;\leq \left|\psi_2(\alpha + \beta - 2\epsilon)\right| = M_{122} = M_{112}.
\end{align}
\end{enumerate}
\end{remark}

{\raggedright{\textit{Proof of Proposition \ref{prop_implicit_MLE}}.}} 
With $\tilde{V}$ and $\boldsymbol{W}$ as in \eqref{cm}, we obtain through the  method of proof of Theorem \ref{Theorem_multi}, that
\begin{equation}
\label{middle_step_bw_implicit}
d_{{\mathrm{bW}}}\left(\boldsymbol{W}, \boldsymbol{Z}\right) \leq d_{\mathrm{bW}}\bigg(\frac{1}{\sqrt{n}}\tilde{V}\nabla\left(\ell(\boldsymbol{\theta}_0;\boldsymbol{X})\right), \boldsymbol{Z}\bigg) + d_{\mathrm{bW}}\bigg(\boldsymbol{W}, \frac{1}{\sqrt{n}}\tilde{V}\nabla\left(\ell\left(\boldsymbol{\theta}_0;\boldsymbol{X}\right)\right)\bigg)
\end{equation}
For the first quantity on the right-hand side of the result in \eqref{middle_step_bw_implicit}, we obtain using Theorem \ref{bonisthm} that
\begin{equation}
\label{bound1}
d_{\mathrm{bW}}\bigg(\frac{1}{\sqrt{n}}\tilde{V}\nabla\left(\ell(\boldsymbol{\theta}_0;\boldsymbol{X})\right), \boldsymbol{Z}\bigg) \leq d_{\mathrm{W}_2}\bigg(\frac{1}{\sqrt{n}}\tilde{V}\nabla\left(\ell(\boldsymbol{\theta}_0;\boldsymbol{X})\right), \boldsymbol{Z}\bigg)\leq \frac{1}{\sqrt{n}}K_{1}(\boldsymbol{\theta}_0).
\end{equation}
With respect to the second term in \eqref{middle_step_bw_implicit}, note that
\begin{equation}
d_{\mathrm{bW}}\bigg(\boldsymbol{W}, \frac{1}{\sqrt{n}}\tilde{V}\nabla\left(\ell\left(\boldsymbol{\theta}_0;\boldsymbol{X}\right)\right)\bigg) = \sup_{h\in\mathcal{H}_{\mathrm{bW}}}\left|\mathbb{E}[h(\boldsymbol{W})]-\mathbb{E}\left[h\left(\frac{1}{\sqrt{n}}\tilde{V}\nabla\left(\ell\left(\boldsymbol{\theta}_0;\boldsymbol{X}\right)\right)\right)\right]\right|.
\end{equation}
For $h \in \mathcal{H}_{\mathrm{bW}}$ and with $\tilde{V}$ and $Q_j$ as in \eqref{cm}, for ease of presentation let us denote by
\begin{align}
\label{notationmultiT1T2}
\nonumber & \boldsymbol{R}_1(\boldsymbol{\theta}_0;\boldsymbol{x}) = \frac{1}{2\sqrt{n}}\tilde{V}\sum_{j=1}^{d}\sum_{q=1}^{d}Q_jQ_q\left(\nabla\left(\frac{\partial^2}{\partial\theta_j\partial\theta_q}\ell(\boldsymbol{\theta};\boldsymbol{x})\Big|_{\substack{\boldsymbol{\theta} = \boldsymbol{\theta}_0^{*}}}\right)\right),\\
& D_1 = D_1(\boldsymbol{\theta}_0;\boldsymbol{X},h) := h\left(\boldsymbol{W}\right) - h\bigg(\vphantom{(\left(\sup_{\theta:|\theta-\theta_0|\leq\epsilon}\left|l^{(3)}(\theta;\boldsymbol{X})\right|\right)^2}\frac{1}{\sqrt{n}}\tilde{V}\left(\nabla (\ell(\boldsymbol{\theta}_0;\boldsymbol{x}))\right) + \boldsymbol{R}_1(\boldsymbol{\theta}_0;\boldsymbol{X})\bigg),\\
\nonumber & D_2 = D_2(\boldsymbol{\theta}_0;\boldsymbol{X},h) := h\bigg(\vphantom{(\left(\sup_{\theta:|\theta-\theta_0|\leq\epsilon}\left|l^{(3)}(\theta;\boldsymbol{X})\right|\right)^2}\frac{1}{\sqrt{n}}\tilde{V}\left(\nabla (\ell(\boldsymbol{\theta}_0;\boldsymbol{x}))\right) + \boldsymbol{R}_1(\boldsymbol{\theta}_0;\boldsymbol{x})\bigg) -  h\left(\frac{1}{\sqrt{n}}\tilde{V}\left(\nabla\left(\ell(\boldsymbol{\theta}_0;\boldsymbol{X})\right)\right)\right),
\end{align}
where $\boldsymbol{\theta}_0^*$ is as in \eqref{Taylor1}. Using the above notation and the triangle inequality,
\begin{equation}
\label{T1T2mean}
\left|\mathbb{E}[h(\boldsymbol{W})]-\mathbb{E}\left[h\left(\frac{1}{\sqrt{n}}\tilde{V}\nabla\left(\ell\left(\boldsymbol{\theta}_0;\boldsymbol{X}\right)\right)\right)\right]\right| = \left|\EE\left[D_1 + D_2\right]\right| \leq \EE|D_1| + \EE|D_2|.
\end{equation}
Since $\boldsymbol{W}$ is as in \eqref{cm}, then for $A_{[j]}$ denoting the $j$-th row of a matrix $A$, a first order multivariate Taylor expansion gives that
\begin{align}
\nonumber & \left|D_1\right| \leq \|h\|_{\rm Lip}\left|\vphantom{(\left(\sup_{\theta:|\theta-\theta_0|\leq\epsilon}\left|l^{(3)}(\theta;\boldsymbol{X})\right|\right)^2}\sum_{j=1}^{d}\left(\vphantom{(\left(\sup_{\theta:|\theta-\theta_0|\leq\epsilon}\left|l^{(3)}(\theta;\boldsymbol{X})\right|\right)^2}\sqrt{n}\left[\left[I(\boldsymbol{\theta}_0)\right]^{\frac{1}{2}}\right]_{[j]}(\boldsymbol{\hat{\theta}}_n(\boldsymbol{X}) - \boldsymbol{\theta}_0)-\frac{1}{\sqrt{n}}\tilde{V}_{[j]}\nabla\left(\ell(\boldsymbol{\theta}_0;\boldsymbol{X})\right)\right.\right.\\
\nonumber & \qquad\quad\qquad\quad\left.\left.- \frac{1}{2\sqrt{n}}\tilde{V}_{[j]}\left\lbrace\vphantom{(\left(\sup_{\theta:|\theta-\theta_0|\leq\epsilon}\left|l^{(3)}(\theta;\boldsymbol{X})\right|\right)^2}\sum_{k=1}^{d}\sum_{q=1}^{d}Q_kQ_q\left(\nabla\left(\frac{\partial^2}{\partial\theta_k\partial\theta_q}\ell(\boldsymbol{\theta};\boldsymbol{x})\Big|_{\substack{\boldsymbol{\theta} = \boldsymbol{\theta}_0^{*}}}\right)\right)\vphantom{(\left(\sup_{\theta:|\theta-\theta_0|\leq\epsilon}\left|l^{(3)}(\theta;\boldsymbol{X})\right|\right)^2}\right\rbrace\vphantom{(\left(\sup_{\theta:|\theta-\theta_0|\leq\epsilon}\left|l^{(3)}(\theta;\boldsymbol{X})\right|\right)^2}\right)\vphantom{(\left(\sup_{\theta:|\theta-\theta_0|\leq\epsilon}\left|l^{(3)}(\theta;\boldsymbol{X})\right|\right)^2}\right|.
\end{align}
Using \eqref{Taylor2} component-wise and the Cauchy-Schwarz inequality, we have that, for $T_{kj}$ as in \eqref{cm},
\begin{align}
\label{boundforT1first}
& \EE|D_1| \leq\frac{\|h\|_{\mathrm{Lip}}}{\sqrt{n}} \sum_{k=1}^{d}\sum_{l=1}^{d}|\tilde{V}_{lk}|\sum_{j=1}^{d}\sqrt{\EE[Q_j^2]\EE[T_{kj}^2]}.
\end{align}
Since $\EE[T_{kj}] = 0$, $\forall j, k \in \left\lbrace 1,2,\ldots, d \right\rbrace$, we have that
\begin{align}
\label{first_term_implicit_multiparameter}
\nonumber \EE|D_1| &\leq \|h\|_{\mathrm{Lip}}\sum_{k=1}^{d}\sum_{l=1}^{d}|\tilde{V}_{lk}|\sum_{j=1}^{d}\sqrt{\EE[Q_j^2]}\sqrt{{\rm Var}\bigg(\frac{\partial^2}{\partial \theta_k\partial \theta_j}\log f(\boldsymbol{X}_1|\boldsymbol{\theta}_0)\bigg)}\\
& \leq \|h\|_{\mathrm{Lip}}\sum_{k=1}^{d}\sum_{l=1}^{d}|\tilde{V}_{lk}|\sum_{j=1}^{d}\sqrt{\EE[Q_j^2]}\sqrt{\sum_{i=1}^{d}{\rm Var}\bigg(\frac{\partial^2}{\partial \theta_k\partial \theta_i}\log f(\boldsymbol{X}_1|\boldsymbol{\theta}_0)\bigg)},
\end{align}
where the inequality trivially holds
since the variance of a random variable is always non-negative. Now, using that  $(\sum_{j=1}^{d}\alpha_j)^2 \leq d(\sum_{j=1}^{d}\alpha_j^2)$ for $\alpha_j \in \mathbb{R}$, yields
\begin{equation}
\nonumber \bigg(\sum_{j=1}^{d}\sqrt{\EE[Q_j^2]}\bigg)^2\leq d\sum_{j=1}^{d}\EE[Q_j^2].
\end{equation} 
Taking square roots in both sides of the above inequality and 
applying this inequality to \eqref{first_term_implicit_multiparameter} yields
\begin{equation}
\label{boundDimplicit1bound}
\EE|D_1| \leq \|h\|_{\mathrm{Lip}}\sqrt{d}\sum_{k=1}^{d}\sum_{l=1}^{d}|\tilde{V}_{lk}|\sqrt{\sum_{i=1}^{d}{\rm Var}\bigg(\frac{\partial^2}{\partial \theta_k\partial \theta_i}\log f(\boldsymbol{X}_1|\boldsymbol{\theta}_0)\bigg)}\sqrt{\EE\bigg[\sum_{j=1}^{d}Q_j^2\bigg]}.
\end{equation}
To bound now $\EE\left|D_2\right|$, with $D_2$ as in \eqref{notationmultiT1T2}, we need to take into account that $\frac{\partial^3}{\partial\theta_k\partial\theta_q\partial\theta_j}\ell(\boldsymbol{\theta};\boldsymbol{x})\Big|_{\substack{\boldsymbol{\theta} = \boldsymbol{\theta}_0^{*}}}$ is in general not uniformly bounded and there is a positive probability that the MLE will be outside an $\epsilon$-neighbourhood of the true value of the parameter. For $\epsilon>0$, the law of total expectation and Markov's inequality yield
\begin{align}
\label{chebyshevmultiT2}
\nonumber \EE\left|D_2\right| &\leq 2\|h\|\PP\left(|Q_{(m)}|\geq\epsilon\right) + \EE\left[|D_2|\,\middle|\,|Q_{(m)}| < \epsilon\right]\\
& \leq \frac{2\|h\|}{\epsilon^2}\EE\bigg[\sum_{j=1}^{d}Q_j^2\bigg] + \EE\left[|D_2|\,\middle|\,|Q_{(m)}| < \epsilon\right],
\end{align}
where for the subscript $(m)$ it holds that
\begin{align}
\nonumber & (m) \in \left\lbrace 1,\ldots,d\right\rbrace\;{\rm is\;such\;that\;} |\hat{\theta}_n(\boldsymbol{x})_{(m)} - \theta_{0,(m)}| \geq |\hat{\theta}_n(\boldsymbol{x})_j - \theta_{0,j}|,\quad \forall j \in \left\lbrace 1,\ldots, d \right\rbrace,
\end{align}
and $Q_{(m)}=Q_{(m)}(\boldsymbol{X},\boldsymbol{\theta}_0) := \hat{\theta}_n(\boldsymbol{X})_{(m)} - \theta_{0,(m)}$. It remains to bound $\EE\left[|D_2|\,\middle|\,|Q_{(m)}| < \epsilon\right]$ by a quantity whose dependence on the MLE is merely through the MSE. A first-order Taylor expansion  and \eqref{Taylor_multi} yield
\begin{align}
\label{T2multijustabound}
\left|D_2\right| &\leq \frac{\|h\|_{\rm Lip}}{2\sqrt{n}}\sum_{k=1}^{d}\sum_{l=1}^{d}|\tilde{V}_{lk}|\sum_{j=1}^{d}\sum_{v=1}^{d}\left|Q_jQ_v\frac{\partial^3}{\partial\theta_{k}\partial\theta_{j}\partial\theta_{v}}\ell(\boldsymbol{\theta};\boldsymbol{X})\Big|_{\substack{\boldsymbol{\theta} = \boldsymbol{\theta}_0^{*}}}\right|.
\end{align}
Therefore, from \eqref{chebyshevmultiT2} and \eqref{T2multijustabound} we have that
\begin{align*}
\nonumber \EE|D_2| &\leq \frac{2\|h\|}{\epsilon^2}\EE\bigg[\sum_{j=1}^{d}Q_j^2\bigg] \\
&\quad+\frac{\|h\|_{\rm Lip}}{2\sqrt{n}}\sum_{k=1}^{d}\sum_{l=1}^{d}|\tilde{V}_{lk}|\EE\bigg[\vphantom{(\left(\sup_{\theta:|\theta-\theta_0|\leq\epsilon}\left|l^{(3)}(\theta;\boldsymbol{X})\right|\right)^2}\sum_{j=1}^{d}\sum_{v=1}^{d}\bigg|\vphantom{(\left(\sup_{\theta:|\theta-\theta_0|\leq\epsilon}\left|l^{(3)}(\theta;\boldsymbol{X})\right|\right)^2}Q_jQ_v\frac{\partial^3}{\partial\theta_{k}\partial\theta_{j}\partial\theta_{v}}\ell(\boldsymbol{\theta};\boldsymbol{X})\Big|_{\substack{\boldsymbol{\theta} = \boldsymbol{\theta}_0^{*}}}\vphantom{(\left(\sup_{\theta:|\theta-\theta_0|\leq\epsilon}\left|l^{(3)}(\theta;\boldsymbol{X})\right|\right)^2}\bigg|\,\bigg|\,\left|Q_{(m)}\right| < \epsilon\vphantom{(\left(\sup_{\theta:|\theta-\theta_0|\leq\epsilon}\left|l^{(3)}(\theta;\boldsymbol{X})\right|\right)^2}\bigg],
\end{align*}
and using (Con.1), we have that
\begin{equation}
\nonumber \EE|D_2| \leq \frac{2\|h\|}{\epsilon^2}\EE\bigg[\sum_{j=1}^{d}Q_j^2\bigg] + \frac{\sqrt{n}\|h\|_{\mathrm{Lip}}}{2}\sum_{k=1}^{d}\sum_{l=1}^{d}|\tilde{V}_{lk}|\EE\bigg[\sum_{j=1}^{d}\sum_{l=1}^{d}\left|Q_jQ_i\right|M_{kji}\,\bigg|\,|Q_{(m)}| < \epsilon\bigg]
\end{equation}
Simple calculations lead to
\begin{align}
\nonumber &\sum_{j=1}^{d}\sum_{i=1}^{d}\left|Q_jQ_i\right|M_{kji} = \sum_{j=1}^{d}Q_j^2M_{kjj} + 2\sum_{i=1}^{d-1}\sum_{j=i+1}^{d}\left|Q_j\right|\left|Q_i\right|M_{kij}
\end{align}
and using that $2\alpha\beta \leq \alpha^2 + \beta^2$,  $\forall \alpha, \beta \in \mathbb{R}$,
\begin{align}
\label{mid_step_boundDimplicit2}
\nonumber \sum_{j=1}^{d}\sum_{i=1}^{d}\left|Q_jQ_i\right|M_{kji} & \leq \sum_{j=1}^{d}Q_j^2M_{kjj} + \sum_{i=1}^{d-1}\sum_{j=i+1}^{d}\left[Q_j^2 + Q_i^2\right]M_{kji} = \sum_{j=1}^{d}Q_j^2\sum_{i=1}^{d}M_{kji}\\
& \leq \sum_{j=1}^{d}Q_j^2\sum_{m=1}^{d}\sum_{i=1}^{d}M_{kmi}.
\end{align}
Using \eqref{mid_step_boundDimplicit2} and Lemma 4.1 from \cite{a18}, yields
\begin{equation}
\label{boundDimplicit2bound}
E|D_2| \leq \frac{2\|h\|}{\epsilon^2}\EE\bigg[\sum_{j=1}^{d}Q_j^2\bigg] + \frac{\sqrt{n}\|h\|_{\mathrm{Lip}}}{2}\sum_{k=1}^{d}\sum_{l=1}^{d}|\tilde{V}_{lk}|\sum_{m=1}^{d}\sum_{i=1}^{d}M_{kmi}\EE\bigg[\sum_{j=1}^{d}Q_j^2\bigg].
\end{equation}
Hence, from \eqref{middle_step_bw_implicit}, \eqref{bound1}, \eqref{T1T2mean}, \eqref{boundDimplicit1bound} and \eqref{boundDimplicit2bound} and using that $\|h\|\leq1$ and $\|h\|_{\mathrm{Lip}}\leq1$ for $h\in\mathcal{H}_{\mathrm{bW}}$, we obtain the upper bound \eqref{final_bound_implicit}, which depends on $\boldsymbol{\hat{\theta}}_n(\boldsymbol{X})$ only through the MSE, $\EE[\sum_{j=1}^{d}Q_j^2]. \hfill \square$
}
\subsection{Empirical results}\label{sec4.56}
In this section, we investigate, through a simulation study, the accuracy of our bounds given in Sections \ref{sec:one_parameter} -- \ref{sec4.3nor}.   We carried out the study using \textsf{R}.  For the exponential distribution with $\theta=1$ under canonical and non-canonical parametrisation (this bound is given in Appendix \ref{expapp}) and the normal distribution under canonical parametrisation with $\boldsymbol{\eta}=(1,1)^\intercal$, we calculated our bound and estimated the true value of $d_{\mathrm{W}}(\boldsymbol{W},\boldsymbol{Z})$ for sample sizes $n=10^j$, $j=1,2,3,4$ (Tables \ref{tableresults_exp_can} -- \ref{tableresults_linear_d3}).  For the multivariate normal distribution under non-canonical \gr{parametrisation} with diagonal covariance matrix we studied the dependence of $d_{\mathrm{W}}(\boldsymbol{W},\boldsymbol{Z})$ on the dimension $p$ with $n=1000$ fixed and $\mu_k=\sigma_k^2=1$ for all $1\leq k\leq p$ (Figure \ref{figure_thm414}). 

Calculating our bounds is straightforward, but estimating the \gr{1-}Wasserstein distance $d_{\mathrm{W}}(\boldsymbol{W},\boldsymbol{Z})$ is more involved.  For a given example and given sample size $n$, we simulated $N$ realisations of the distributions of $\boldsymbol{W}$ and $\boldsymbol{Z}$ to obtain the empirical distribution functions of both distributions.  We then used the \textsf{R} package \textbf{\textit{transport}} to compute the \gr{1-}Wasserstein distance between these two empirical distributions.  As we simulated the distributions, we only obtained an estimate for the \gr{1-}Wasserstein distance $d_{\mathrm{W}}(\boldsymbol{W},\boldsymbol{Z})$, although this estimate improves as $N$ increases.  To mitigate the random effects from the simulations, we repeated this $K=100$ times and then took the sample mean to obtain our estimate $\hat{d}_{\mathrm{W}}(\boldsymbol{W},\boldsymbol{Z})$.  We used $N=10^4$ for all simulations, except for the multivariate normal distribution under non-canonical parametrisation for which we used $N=10^3$ on account of the many simulations for the 99 values of the dimension $p$.    

\begin{table}[H]
\centering
\caption{Simulation results for the $\mathrm{Exp}(1)$ distribution under canonical parametrisation}
\begin{tabular}{r|r|r|r}
	  $n$ & $\hat{d}_{\mathrm{W}}(W,Z)$ & Bound & Error\\
	  \hline
	  \hline
  10 & 0.351 & 2.303  & 1.952\\
  \hline
  100 & 0.100 & 0.649  & 0.548\\
  \hline
  1000 & 0.034 & 0.203 & 0.169\\
  \hline
  10,000 & 0.020 & 0.064  & 0.044\\
  \end{tabular}
\label{tableresults_exp_can}
\end{table}

\vspace{-3mm}

\begin{table}[H]
\centering
\caption{Simulation results for the $\mathrm{Exp}(1)$ distribution under non-canonical parametrisation}
\begin{tabular}{r|r|r|r|r}
	  $n$ & $\hat{d}_{\mathrm{W}}(W,Z)$ & Bound & Error & Bound using Theorem \ref{thmapw}\\
	  \hline
	  \hline
  10 & 0.103 & 7.499  & 7.396 & 0.321\\
  \hline
  100 & 0.036 & 1.498  & 1.463 & 0.101\\
  \hline
  1000 & 0.021 & 0.458 & 0.437 & 0.032\\
  \hline
  10,000 & 0.017 & 0.144  & 0.127 & 0.010\\
\end{tabular}
\label{tableresults_exp_noncan}
\end{table}
\vspace{-3mm}

\begin{table}[H]
\caption{Simulation results for the $\mathrm{N}(1,1)$ distribution under canonical parametrisation}
\centering
\begin{tabular}{r|r|r|r}
	  $n$ & $\hat{d}_{\mathrm{W}}(\boldsymbol{W},\boldsymbol{Z})$ & Bound & Error\\
	  \hline
	  \hline
  10 & 1.032 & 8962.830  & 8961.798\\
  \hline
  100 & 0.224 & 2834.296  &  2834.072\\
  \hline
  1000 & 0.083 & 896.283 & 896.200\\
  \hline
  10,000 & 0.057 & 283.430 & 283.373\\
\end{tabular}
\label{tableresults_linear_d3}
\end{table}

From the tables we see that at each step we increase the sample size by a factor of ten, the value of the upper bound drops by approximately a factor of $\sqrt{10}$, which is expected as our bounds are of order $\mathcal{O}\left(n^{-1/2}\right)$.  The simulated \gr{1-}Wasserstein distances $\hat{d}_{\mathrm{W}}(\boldsymbol{W},\boldsymbol{Z})$ do not decrease by a factor of roughly $\sqrt{10}$ for larger sample sizes, because the approximation errors resulting from taking a finite value of $N$ become more noticeable when the value of $\hat{d}_{\mathrm{W}}(\boldsymbol{W},\boldsymbol{Z})$ decreases.

Our bounds for the exponential distribution perform reasonably well, particularly in the canonical parametrisation case.  In Table \ref{tableresults_exp_noncan} for the exponential distribution under non-canonical parametrisation we also provide the bound obtained from a direct application of Theorem \ref{thmapw} (this is inequality (\ref{boundf})), which as expected is an order of magnitude better than our bound resulting from the general approach.  The bounds for the normal distribution under canonical parametrisation are much bigger than for the exponential distribution.  This is a result of the increased complexity of this example and the fact that we sacrificed best possible constants in favour of a simpler proof and compact final bound.



Figure \ref{figure_thm414} shows the behaviour of the simulated \gr{1-}Wasserstein distance $\hat{d}_{\mathrm{W}}(\boldsymbol{W}, \boldsymbol{Z})$ for the multivariate normal distribution with diagonal covariance matrix with $\mu_k=\sigma_k^2=1$, $1\leq k\leq p$, when the dimension $p$ varies from 2 up to 100.  Here our focus was on the dependence on the dimension for fixed $n$, so we chose a small sample size $n=1000$ to reduce the computational complexity of the simulations.  Figure \ref{figure_thm414} also contains a log-log plot.  Across all 99 data points there is clearly not a straight line fit, but after the value 3.8 for $\log(p)$ (the 45th data point), we start to see some stabilisation towards a straight line.  We obtained a slope of 0.576 between the 70th and 99th data points, which reduced to 0.569 between the 90th and 99th data points. \gr{The results from these simulations suggest that the slope is converging down to 0.5, which would be consistent with the theoretical $\mathcal{O}(p^{1/2})$ scaling of our bound (\ref{bound78}).}


\begin{figure}[H]
\centering
\includegraphics[width=1\textwidth, height=0.3\textheight]{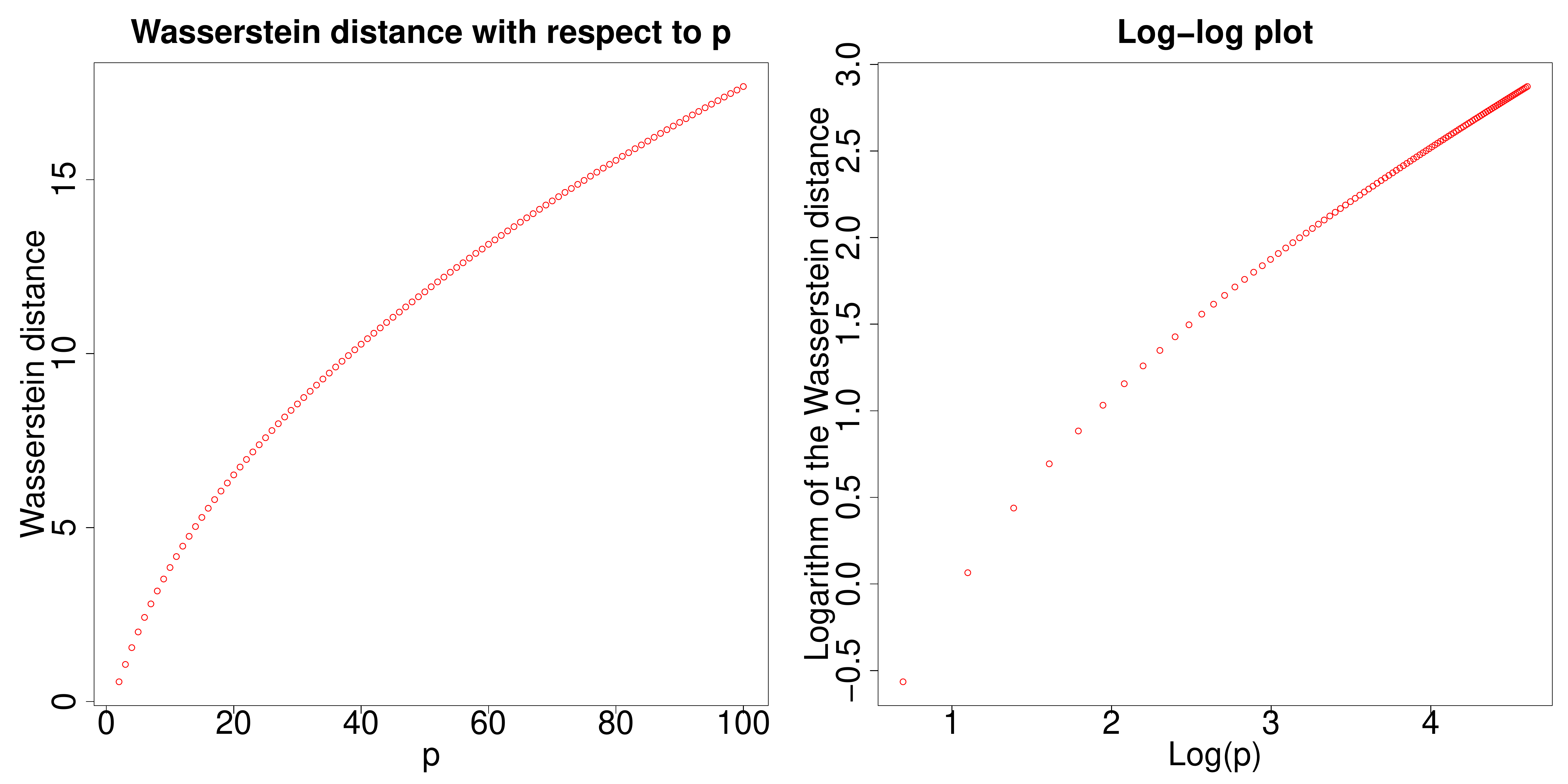}
\caption{Simulated values of $d_{\mathrm{W}}(\boldsymbol{W}, \boldsymbol{Z})$ in the setting of Theorem \ref{thmdiag} when the dimension $p$ varies in the set $\left\lbrace 2, 3, 4, \ldots, 100\right\rbrace$.} 
\label{figure_thm414}
\end{figure}

\appendix

\section{Further examples, proofs and calculations}\label{appa}

\subsection{Verifying (R.C.4'') for the inverse gamma distribution}\label{appig}

Let $X_1, X_2, \ldots, X_n$ be i.i.d$.$ inverse gamma random variables with parameters $\alpha>0$ and $\beta>0$ and probability density function
\begin{equation*}f(x|\alpha,\beta)=\frac{\beta^\alpha}{\Gamma(\alpha)}x^{-\alpha-1}\exp\Big\{-\frac{\beta}{x}\Big\}, \quad x>0.
\end{equation*}
In this appendix, we verify condition (R.C.4'') for the single-parameter MLE for the inverse gamma distribution (fixed $\alpha$ or fixed $\beta$).  The purpose is to give an illustration of how (R.C.4'') can be verified for more complicated MLEs than those considered in Section \ref{sec4}.  To keep the calculations manageable, we focus on the single-parameter case. 

For the moment, let $\theta$ denote the unknown parameter, either $\alpha$ or $\beta$. Recall that in the single-parameter case condition (R.C.4'') is
\begin{equation*}
\mathrm{max}_{\tilde{\theta}\in \left\lbrace\hat{\theta}_n(\boldsymbol{X}), \theta_{0}\right\rbrace}\E\big|(\hat{\theta}_n(\boldsymbol{X}) - \theta_{0})^2M(\tilde{\boldsymbol{\theta}}; \boldsymbol{X})\big| < \infty.
\end{equation*}
We shall verify the stronger (and, in this case, simpler to verify) condition that
\begin{equation*}
\mathbb{E}\big[(\hat{\theta}_n(\boldsymbol{X}) - \theta_{0})^4\big]\mathrm{max}_{\tilde{\theta}\in \left\lbrace\hat{\theta}_n(\boldsymbol{X}), \theta_{0}\right\rbrace}\E\big[\big(M(\tilde{\boldsymbol{\theta}}; \boldsymbol{X})\big)^2\big] < \infty,
\end{equation*}
which implies (R.C.4'') by the Cauchy-Schwarz inequality.  It should be noted that provided $\mathbb{E}\big[(\hat{\theta}_n(\boldsymbol{X}) - \theta_{0})^4\big]<\infty$, the argument of part (1) of Remark \ref{rem12} shows that this quantity is order $O(n^{-2})$.  In verifying that the expectations involving the monotonic dominating function $M$ are finite, we shall see, as expected, that these expectations are order $O(n^2)$. Therefore the final term in bound (\ref{boundTHEOREM}) of Theorem \ref{Theoremnoncan} is of the desired order $O(1)$.  An application of Theorem \ref{Theoremnoncan}, and further calculations to bound the other (simpler) terms would confirm that we obtain a Wasserstein distance bound with $O(n^{-1/2})$ convergence rate.

\vspace{2mm}

\noindent{\emph{1. Unknown $\beta$, fixed $\alpha=\alpha_0$.}} The log-likelihood function is
\begin{equation*}\ell(\beta;\boldsymbol{x})=n\alpha_0\log\beta+n\log\Gamma(\alpha_0)-(\alpha_0+1)\sum_{i=1}^n \log x_i-\beta\sum_{i=1}^nx_i^{-1},
\end{equation*}
from which we readily obtain the unique MLE $\hat{\beta}=\frac{n\alpha_0}{\sum_{i=1}^nX_i^{-1}}$.  Note that $\hat{\beta}\stackrel{d}{=} G^{-1}$, where $G\sim\mathrm{G}(n\alpha_0,n\alpha_0\beta_0)$, which can be seen from standard properties of the gamma distribution and the relation that if $X\sim\mathrm{Inv.G}(\alpha,\beta)$, then $X^{-1}\sim\mathrm{G}(\alpha, \beta)$.  Therefore
\begin{align*}\mathbb{E}[(\hat{\beta}-\beta_0)^4]\leq 8(\mathbb{E}[\hat{\beta}^4]+\beta_0^4)=8(\mathbb{E}[G^{-4}]+\beta_0^4)<\infty, \quad \text{for $\alpha_0>4n^{-1}$}.
\end{align*} 
We have that $\ell^{(3)}(\beta;\boldsymbol{x})=\frac{2n\alpha_0}{\beta^3}$, and so we may take $M(\beta;\boldsymbol{x})=\frac{2n\alpha_0}{\beta^3}$. We have 
\begin{align*}\mathbb{E}[(M(\beta_0;\boldsymbol{X}))^2]=\frac{4n^2\alpha_0^2}{\beta_0^6}<\infty, \quad
\mathbb{E}\big[(M(\hat{\beta};\boldsymbol{X}))^2\big]=4n^2\alpha_0^2\mathbb{E}[G^6]<\infty,
\end{align*}
and, moreover, $\mathbb{E}[(M(\hat{\beta};\boldsymbol{X}))^2]=O(n^2)$, since $\E[G^6]=O(1)$.

\vspace{2mm}

\noindent{\emph{2. Unknown $\alpha$, fixed $\beta=\beta_0$.}} The log-likelihood function is
\begin{equation*}\ell(\alpha;\boldsymbol{x})=n\alpha\log\beta_0+n\log\Gamma(\alpha)-(\alpha+1)\sum_{i=1}^n \log x_i-\beta_0\sum_{i=1}^nx_i^{-1},
\end{equation*}
and differentiating gives
\begin{equation*}\ell'(\alpha;\boldsymbol{x})=n\log \beta_0+n\psi(\alpha)-\sum_{i=1}^n\log x_i,
\end{equation*}
where $\psi(x)=\frac{\mathrm{d}}{\mathrm{d}x}(\log \Gamma(x))$ is the digamma function. The unique MLE is thus given by
\begin{equation*}\hat{\alpha}=\psi^{-1}\bigg(\frac{1}{n}\sum_{i=1}^n\log\Big(\frac{X_i}{\beta_0}\Big)\bigg),
\end{equation*}
where $\psi^{-1}(x)$ is the inverse digamma function.  In verifying (R.C.4''), we shall make use of the following inequality of \cite{b18}:
\begin{equation}\label{b18bound}\frac{1}{\log(1+\mathrm{e}^{-x})}<\psi^{-1}(x)<\mathrm{e}^x+\frac{1}{2}, \quad x\in\mathbb{R}.
\end{equation}

Let us first show that $\mathbb{E}[(\hat{\alpha}-\alpha_0)^4]<\infty$.  We have that $\mathbb{E}[(\hat{\alpha}-\alpha_0)^4]\leq 8(\mathbb{E}[\hat{\alpha}^4]+\alpha_0^4)$, so it suffices to prove that $\mathbb{E}[\hat{\alpha}^4]<\infty$. By the upper bound in (\ref{b18bound}),
\begin{align*}\mathbb{E}[\hat{\alpha}^4]&\leq\mathbb{E}\bigg[\bigg(\exp\bigg\{\frac{1}{n}\sum_{i=1}^n\log\Big(\frac{X_i}{\beta_0}\Big)\bigg\}+\frac{1}{2}\bigg)^4\bigg]\leq 8\bigg(\mathbb{E}\bigg[\exp\bigg\{\frac{4}{n}\sum_{i=1}^n\log\Big(\frac{X_i}{\beta_0}\Big)\bigg\}\bigg]+\frac{1}{16}\bigg)\\
&=\frac{1}{2}+\mathbb{E}\bigg[\prod_{i=1}^n\Big(\frac{X_i}{\beta_0}\Big)^{4/n}\bigg]=\frac{1}{2}+\frac{1}{\beta_0^n}\big(\E[X_1^{4/n}]\big)^n <\infty, \quad \text{for $\alpha_0>4n^{-1}$,}
\end{align*}
where we used that $X_1,\ldots,X_n$ are i.i.d$.$ in the final equality, and in the final step we used that, for $X\sim \mathrm{Inv.G}(\alpha,\beta)$, $\mathbb{E}[X^\gamma]<\infty$ for $\alpha>\gamma$.

  We have that $\ell^{(3)}(\alpha;\boldsymbol{x})=-n\psi_2(\alpha)$, where $\psi_2(x)=\frac{\mathrm{d}^2}{\mathrm{d}x^2}(\psi(x))$ is a polygamma function. From the infinite series representation $\psi_2(x)=-2\sum_{k=0}^\infty(k+x)^{-3}$, $x>0$ (differentiate both sides of formula 5.15.1 of \cite{olver}), it follows that $-\psi_2(x)$ a positive, monotone strictly decreasing function of $x$ on $(0,\infty)$.  We may therefore take $M(\alpha;\boldsymbol{x})=-n\psi_2(\alpha)$. As in the case of unknown $\beta$ and fixed $\alpha=\alpha_0$, it is immediate that $\mathbb{E}[(M(\alpha_0;\boldsymbol{X}))^2]<\infty$, and that this quantity is $O(n^2)$. We now focus on the more involved task of showing that $\mathbb{E}[(M(\hat{\alpha};\boldsymbol{X}))^2]<\infty$. We begin by noting the elementary inequality $-\psi_2(x)\leq2x^{-3}+x^{-2}$, $x>0$ \cite{gq10}. On using this inequality we obtain 
\begin{align*}\mathbb{E}[(M(\hat{\alpha};\boldsymbol{X}))^2]=n^2\mathbb{E}[(\psi_2(\hat{\alpha}))^2]\leq n^2\mathbb{E}\bigg[\bigg(\frac{2}{\hat{\alpha}^3}+\frac{1}{\hat{\alpha}^2}\bigg)^2\bigg]\leq 2n^2\big(4\mathbb{E}[\hat{\alpha}^{-6}]+\mathbb{E}[\hat{\alpha}^{-4}]\big).
\end{align*}
Using the lower bound of (\ref{b18bound}) followed by the elementary inequality $\log(1+\mathrm{e}^{-x})\leq \log 2+|x|$, $x\in\mathbb{R}$, we obtain that, for $m=4,6$,
\begin{align}\label{asdcv}\mathbb{E}[\hat{\alpha}^{-m}]\leq \mathbb{E}\bigg[\bigg(\log2+\frac{1}{n}\bigg|\sum_{i=1}^n\log\Big(\frac{X_i}{\beta_0}\Big)\bigg|\bigg)^m\bigg],
\end{align}
which is finite because $\E[|\log(X_i)|^k]<\infty$, for $k=1,\ldots,6$. Moreover, it is readily seen from (\ref{asdcv}) that $\mathbb{E}[\hat{\alpha}^{-4}]=O(1)$ and $\mathbb{E}[\hat{\alpha}^{-6}]=O(1)$. We therefore conclude that $\mathbb{E}[(M(\hat{\alpha};\boldsymbol{X}))^2]<\infty$, and that this quantity is of the expected order $O(n^{2})$.

\subsection{Exponential distribution: the non-canonical case}\label{expapp}

Let $X_1, X_2, \ldots, X_n$ be i.i.d$.$ random variables from the $\mathrm{Exp}\left(\frac{1}{\theta}\right)$ distribution with probability density function
\begin{align*}
 f(x|\theta) &= \frac{1}{\theta}{\rm exp}\left\lbrace-\frac{1}{\theta}x\right\rbrace\mathbf{1}_{\{x>0\}} = {\rm exp}\left\lbrace-{\rm log}\theta - \frac{1}{\theta}x\right\rbrace\mathbf{1}_{\{x >0\}}\\
&= {\rm exp}\left\lbrace k(\theta)T(x) - A(\theta) + S(x)\right\rbrace\mathbf{1}_{\{x \in B\}},
\end{align*}
where $B = (0,\infty)$, $\theta \in \Theta = (0,\infty)$, $T(x) = -x$, $k(\theta) = \frac{1}{\theta}$, $A(\theta) = {\rm log}\theta$ and $S(x) = 0$. Thus, Exp$\left(\frac{1}{\theta}\right)$ is a non-canonical exponential family distribution. The MLE is unique and equal to $\hat{\theta}_n(\boldsymbol{X}) = \bar{X}$.
\begin{corollary}
\label{Corollarynoncanexp}
Let $X_1, X_2, \cdots, X_n$ be i.i.d$.$ random variables that follow the $\mathrm{Exp}(\frac{1}{\theta_0})$ distribution.  Let $W=\sqrt{n\:i(\theta_0)}(\hat{\theta}_n(\boldsymbol{x})- \theta_0)$ and $Z\sim\mathrm{N}(0,1)$. Then, for $n>3$,
\begin{align}
\label{boundnoncanexponential3}
 d_{\mathrm{W}}(W,Z) < \frac{10.41456}{\sqrt{n}} + \frac{4n^{3/2}(n+6)}{(n-1)(n-2)(n-3)}+\frac{6}{n^{3/2}}.
\end{align}
\end{corollary}

\begin{remark}
\label{remarknoncan}
\textbf{(1)} This example is given for purely illustrative purposes, as an improved bound can be obtained directly by Stein's method.   Define $S = \frac{\sqrt{n}(\bar{X} - \theta_0)}{\theta_0} = \frac{1}{\sqrt{n}}\sum_{i=1}^{n} Y_i$, where $Y_i = \frac{X_i - \theta_0}{\theta_0}$ are i.i.d$.$ zero mean and unit variance random variables.  Therefore, by Theorem \ref{thmapw},
\begin{equation}
\label{boundf}
d_{\mathrm{W}}(W,Z) \leq \frac{1}{\sqrt{n}}\left(2 + \frac{1}{\theta_0^3}{\EE}[|X_1 - \theta_0|^3]\right) < \frac{4.41456}{\sqrt{n}}.
\end{equation}
However, in order to apply Stein's method directly, we require the quantity $W=\sqrt{n\:i(\theta_0)}(\hat{\theta}_n(\boldsymbol{x})- \theta_0)$ to be a sum of independent random variables. The general theorems obtained in this paper are, however, applicable whatever the form of the MLE is, as long as the regularity conditions are met. \\
\textbf{(2)} Like the bound of Corollary \ref{corollaryexponential} for the exponential distribution under canonical parametrisation, the bound (\ref{boundnoncanexponential3}) of Corollary \ref{Corollarynoncanexp} is of order $\mathcal{O}(n^{-1/2})$ and does not depend on $\theta_0$.  These features are shared by the bound (\ref{boundf}) obtained by a direct application of Stein's method.  A bound with these features was also obtained by \cite{ar15} in the weaker bounded Wasserstein metric.  Despite being given in a stronger metric, our bound has numerical constants that are an order of magnitude smaller.
\end{remark}

\begin{proof} It is straightforward to show that $\hat{\theta}_n(\boldsymbol{X}) = \bar{X}$ and that the conditions (R1)--(R3), (R.C.4'') are satisfied for this specific example. The log-likelihood function is 
\begin{align}
\nonumber  \ell(\theta_0;\boldsymbol{x}) = -nA(\theta_0) + k(\theta_0)\sum_{i=1}^{n}T(x_i) = -n\bigg(\log\theta_0 +\frac{\bar{x}}{\theta_0}\bigg).
\end{align}
We have that
\begin{equation*}
|\ell^{(3)}(\theta;\boldsymbol{x})| = n\left|\frac{2}{\theta^3} - \frac{6\bar{x}}{\theta^4}\right| \leq \frac{2n}{\theta^3}\left|1 + \frac{3\bar{x}}{\theta}\right|,
\end{equation*}
which is a decreasing function with respect to $\theta$, and therefore condition (R.C.4'') is satisfied with $M(\theta;\boldsymbol{x}) = \frac{2n}{\theta^3}\left|1 + \frac{3\bar{x}}{\theta}\right|$. Basic calculations of integrals show that $\EE[|T(X_1)-D(\theta_0)|^3] = \EE[\left|\theta_0 - X_1\right|^3] < 2.41456\theta_0^3$.  In addition, since $T(x)=x$, we have that ${\rm Var}(T(X_1)) = {\rm Var}(X_1) = \theta_0^2$ and therefore for the first term of the upper bound in \eqref{noncanexponential}, we have that
\begin{align}
\label{upperbound_first_term_noncan_exponential}
\frac{1}{\sqrt{n}}\left(2 + \frac{{\rm E}[|T(X_1)-D(\theta_0)|^3]}{\left[{\rm Var}(T(X_1))\right]^{3/2}}\right) < \frac{4.41456}{\sqrt{n}}.
\end{align}
Now, consider the second term.  The quantity $\EE[(\bar{X} - \theta_0)^2]$ is calculated using the results in p$.$ 73 and the equations (3.38), p$.$ 70 of \cite{DistributionTheory} along with the fact that $\hat{\theta}_n(\boldsymbol{X}) = \bar{X} \sim  \mathrm{G}\big(n, \frac{n}{\theta_0}\big)$. We obtain that $\EE[(\bar{X} - \theta_0)^2] =\frac{\theta_0^2}{n}$.  We also have that $i(\theta_0) =\frac{1}{\theta_0^2}$, and therefore
\begin{equation}
\label{upperbound_second_term_noncan_exponential}
\frac{|k''(\theta_0)|}{\sqrt{i(\theta_0)}}\sqrt{{\rm Var}\left(T(X_1)\right)}\sqrt{\EE\big[\big(\hat{\theta}_n(\boldsymbol{X})- \theta_0\big)^2\big]} = \frac{2}{\sqrt{n}}.
\end{equation}
Finally, we work on the third term. Since $\bar{X} \sim {\rm G}\big(n,\frac{n}{\theta_0}\big)$ and $\frac{1}{\bar{X}} \sim {\rm Inv. G}\big(n,\frac{n}{\theta_0}\big)$ (where ${\rm Inv. G}$ denotes the inverse gamma distribution), we have that
\begin{align}
\label{M_for_theta0_noncan}
\nonumber  \EE\big|(\hat{\theta}_n(\boldsymbol{X}) - \theta_0)^2M(\theta_0;\boldsymbol{X})\big| &= \frac{2n}{\theta_0^4}\EE\big[\left(\bar{X} - \theta_0\right)^2\left(3\bar{X} + \theta_0\right)\big]\\
\nonumber &  = \frac{2n}{\theta_0^4}\left\lbrace 3\EE[\bar{X}^3] - 5\theta_0\EE[\bar{X}^2] + \theta_0^2\EE[\bar{X}] + \theta_0^3\right\rbrace\\
\nonumber &  = \frac{2n}{\theta_0^4}\left\lbrace \frac{3n(n+1)(n+2)\theta_0^3}{n^3} - \frac{5n(n+1)\theta_0^3}{n^2} + 2\theta_0^3\right\rbrace\\
&  = \frac{4(2n+3)}{n\theta_0}
\end{align}
and, for $n>3$,
\begin{align}
\label{M_for_thetahat_noncan}
\nonumber  \EE\big|(\hat{\theta}_n(\boldsymbol{X}) - \theta_0)^2M(\hat{\theta}_n(\boldsymbol{X});\boldsymbol{X})\big| &=  8n\EE\left[\frac{\left(\bar{X} - \theta_0\right)^2}{\bar{X}^3}\right] = 8n\EE\left[\frac{1}{\bar{X}} + \frac{\theta_0^2}{\bar{X}^3} - \frac{2\theta_0}{\bar{X}^2}\right]\\
\nonumber &  = \frac{8n}{n-1}\left(\frac{n}{\theta_0} + \frac{n^3}{(n-2)(n-3)\theta_0} - \frac{2n^2}{(n-2)\theta_0}\right)\\
&  = \frac{8n^2(n+6)}{(n-1)(n-2)(n-3)\theta_0}.
\end{align}
Applying the results of \eqref{upperbound_first_term_noncan_exponential}, \eqref{upperbound_second_term_noncan_exponential}, \eqref{M_for_theta0_noncan} and \eqref{M_for_thetahat_noncan} to \eqref{noncanexponential}, and using that $i(\theta_0) = \frac{1}{\theta_0^2}$, yields the desired bound. 
\end{proof}


\subsection{Further calculations from the proof of Corollary \ref{thmnorcan}}\label{appnor4}

\noindent{\emph{Proof of Lemma \ref{momentlem}.}} Let us first note the standard result that $\bar{X}$ and $\sum_{i=1}^{n}\left(X_i - \bar{X}\right)^2$ are independent, which follows from Basu's theorem.  We also have that $\bar{X}\sim \mathrm{N}(\mu,\frac{\sigma^2}{n})$ and $\frac{1}{\sigma^2}\sum_{i=1}^{n}\left(X_i - \bar{X}\right)^2\sim \chi_{(n-1)}^2$, the chi-square distribution with $n-1$ degrees of freedom.  We therefore have that $\hat{\eta}_1=_d \frac{n}{2\sigma^2}V$ and $\hat{\eta}_2=_d\frac{n}{\sigma^2}UV$, where $U\sim \mathrm{N}(\mu,\frac{\sigma^2}{n})$ and $V\sim\mathrm{Inv-}\chi_{(n-1)}^2$ are independent.  All expectations as given in the lemma can therefore be computed exactly using the formulas
\begin{align*}&\E[U]=\mu,\quad \E[U^2]=\mu^2+\frac{\sigma^2}{n}, \quad \E[U^3]=\mu^3+\frac{3\mu\sigma^2}{n},\quad \E[U^4]=\mu^4+\frac{6\mu^2\sigma^2}{n}+\frac{3\sigma^4}{n^2},\\
&\E[V^k]=\frac{1}{(n-3)(n-5)\cdots(n-2k-1)}, \quad k=1,2,\ldots,\quad n>2k+1, \\
&\E[V^{-k}]=(n-1)(n+1)\cdots(n+2k-3), \quad k=1,2,\ldots,\quad n>1,
\end{align*}
and then expressing the resulting expression in terms of the canonical parametrisation $(\eta_1,\eta_2)=(\frac{1}{2\sigma^2},\frac{\mu}{\sigma^2})$.  (Here the expectations $\E[V^k]$ and $\E[V^{-k}]$, follow from the standard formula that, for $Y\sim\chi_{(r)}^2$, $\E[Y^m]=2^m\frac{\Gamma(m+r/2)}{\Gamma(r/2)}$, $r>0$, $m>-\frac{r}{2}$ and the identity $\Gamma(x+1)=x\Gamma(x)$.) As an example,
\begin{align*}\E[Q_1^2]=\E[\hat{\eta}_1^2]-2\eta_1\E[\hat{\eta}_1]+\eta_1^2=\frac{\eta_1^2n^2}{(n-3)(n-5)}-\frac{2\eta_1^2n}{n-3}+\eta_1^2=\frac{\eta_1^2(2n+15)}{(n-3)(n-5)}.
\end{align*}
To obtain the compact bound for $\E[Q_1^2]$ as stated in the lemma, we note that $f(n):=\frac{n(2n+15)}{(n-3)(n-5)}$ is a decreasing function of $n$ for $n>9$ with $f(10)=10$.  Similar calculations show that, for $n>9$,
\begin{align*}\E[Q_2^2]&=\frac{2\eta_1n+(2n+15)\eta_2}{(n-3)(n-5)}, \quad
\E[Q_1^4]=\frac{\eta_1^4(12n^2+516n+945)}{(n-3)(n-5)(n-7)(n-9)}, \\
\E[Q_2^4]&=\frac{12n^2\eta_1^2+12n(2n+63)\eta_1\eta_2^2+3(4n^2+172n+315)\eta_2^4}{(n-3)(n-5)(n-7)(n-9)}, 
\end{align*}
and, by the Cauchy-Schwarz inequality, $\E[Q_1^2Q_2^2]\leq\sqrt{\E[Q_1^4]\E[Q_2^4]}$.  We also have that, for $n>9$,
\begin{align*}&\E[\hat{\eta}_1^{-8}]=\frac{(n-1)(n+1)(n+3)\cdots(n+13)}{\eta_1^8n^8},\quad
\E[\hat{\eta}_1^{-6}]=\frac{(n-1)(n+1)(n+3)\cdots(n+9)}{\eta_1^6n^6},\\
&\E[\hat{\eta}_1^{-4}]=\frac{(n-1)(n+1)(n+3)(n+5)}{\eta_1^4n^4},\E[\hat{\eta}_2^{2}]=\frac{n^2}{(n-3)(n-5)}\bigg(\eta_2^2+\frac{2\eta_1}{n}\bigg), \\
&\E[\hat{\eta}_2^{4}]=\frac{n^4}{(n-3)(n-5)(n-7)(n-9)}\bigg(\eta_2^4+\frac{12\eta_1\eta_2^2}{n}+\frac{12\eta_1^2}{n^2}\bigg),\\
&\E\bigg[\frac{\hat{\eta}_2}{\hat{\eta}_1^3}\bigg]=\frac{(n-1)(n+1)\eta_2}{n^2\eta_1^3}, \quad
\E\bigg[\frac{\hat{\eta}_2^2}{\hat{\eta}_1^6}\bigg]=\frac{(n-1)(n+1)(n+3)(n+5)}{\eta_1^6n^4}\bigg(\eta_2^2+\frac{2\eta_1}{n}\bigg), \\
&\E\bigg[\frac{\hat{\eta}_2^4}{\hat{\eta}_1^8}\bigg]=\frac{(n-1)(n+1)(n+3)(n+5)}{\eta_1^8n^4}\bigg(\eta_2^4+12\frac{\eta_1\eta_2^2}{n}+\frac{12\eta_1^2}{n^2}\bigg).
\end{align*}
From these formulas we are able to obtain compacts bounds for all expectations given in the lemma, that are valid for $n\geq10$, using a similar argument to the one we used to bound $\E[Q_1^2]$.  We round up all numerical constants to the nearest integer  We further simplify the bounds for $\E[Q_2^4]$, $\E[\hat{\eta}_2^4]$ and $\E[\hat{\eta}_2^4/\hat{\eta}_1^8]$ using the inequality $ab\leq \frac{1}{2}(a^2+b^2)$.    \hfill $\Box$

\vspace{3mm}

\noindent{\emph{Bounding the terms $R_{2,1,1}^{M_{211}}$, $R_{1,1,2}^{M_{112}}$, $R_{2,1,2}^{M_{212}}$ and $R_{1,2,2}^{M_{122}}$.}} 

\vspace{2mm}

\noindent{$R_{2,1,1}^{M_{211}}$}:
\begin{align*}R_{2,1,1}^{M_{211},A}=\E\bigg[Q_1^2\frac{n|\eta_2|}{\eta_1^3}\bigg]\leq\frac{10|\eta_2|}{\eta_1},
\end{align*}
\begin{align*}R_{2,1,1}^{M_{211},B}&=\E\bigg[Q_1^2\frac{n|\eta_2|}{\hat{\eta}_1^3}\bigg]\leq n|\eta_2|\sqrt{\E[Q_1^4]\E[\hat{\eta}_1^{-6}]}<|\eta_2|\sqrt{6958\eta_1^4\cdot\frac{7}{\eta_1^6}}<\frac{221|\eta_2|}{\eta_1},
\end{align*}
 \begin{align*}R_{2,1,1}^{M_{211},C}&=\E\bigg[Q_1^2\frac{n|\hat{\eta}_2|}{\eta_1^3}\bigg]\leq \frac{n}{\eta_1^3}\sqrt{\E[Q_1^4]\E[\hat{\eta}_2^{2}]}<\frac{1}{\eta_1^3}\sqrt{6958\eta_1^4(\eta_1+3\eta_2^2)}<\frac{84}{\sqrt{\eta_1}}+\frac{145|\eta_2|}{\eta_1},
\end{align*}
\begin{align*}R_{2,1,1}^{M_{211},D}&=\E\bigg[Q_1^2\frac{n|\hat{\eta}_2|}{\hat{\eta}_1^3}\bigg]\leq n\sqrt{\E[Q_1^4]\E\bigg[\frac{\hat{\eta}_2^{2}}{\hat{\eta}_1^6}\bigg]}<\sqrt{6958\eta_1^4\cdot\frac{1}{\eta_1^6}(\eta_1+2\eta_2^2)}<\frac{84}{\sqrt{\eta_1}}+\frac{118|\eta_2|}{\eta_1}.
\end{align*}

\vspace{2mm}

\noindent{$R_{1,1,2}^{M_{112}}$}:
\begin{align*}R_{1,1,2}^{M_{112},A}&=\E\bigg|Q_1Q_2\frac{n|\eta_2|}{\eta_1^3}\bigg|\leq\frac{n|\eta_2|}{\eta_1^3}\sqrt{\E[Q_1^2]\E[Q_2^2]}<\frac{n|\eta_2|}{\eta_1^3}\sqrt{10\eta_1^2(6\eta_1+10\eta_2^2)}\\
&<\frac{8|\eta_2|}{\eta_1^{3/2}}+\frac{10\eta_2^2}{\eta_1^2}\leq\frac{4}{\eta_1}+\frac{14\eta_2^2}{\eta_1^2},
\end{align*}
\begin{align*}R_{1,1,2}^{M_{112},B}&=\E\bigg|Q_1Q_2\frac{n|\eta_2|}{\hat{\eta}_1^3}\bigg|\leq n|\eta_2|\sqrt{\E[Q_1^2Q_2^2]\E[\hat{\eta}_1^{-6}]}<n|\eta_2|\sqrt{\eta_1^2(6400\eta_1+9023\eta_2^2)\cdot\frac{7}{\eta_1^6}}\\
&<\frac{212|\eta_2|}{\eta_1^{3/2}}+\frac{252\eta_2^2}{\eta_1^2}\leq\frac{106}{\eta_1}+\frac{358\eta_2^2}{\eta_1^2},
\end{align*}
\begin{align*}R_{1,1,2}^{M_{112},C}&=\E\bigg|Q_1Q_2\frac{n|\hat{\eta}_2|}{\eta_1^3}\bigg|\leq \frac{n}{\eta_1^3}\sqrt{\E[Q_1^2Q_2^2]\E[\hat{\eta}_2^{2}]}<\frac{1}{\eta_1^3}\sqrt{\eta_1^2(6400\eta_1+9023\eta_2^2)(\eta_1+3\eta_2^2)}\\
&\leq\frac{1}{\eta_1^3}\sqrt{\frac{41023}{2}\eta_1^2+\frac{82361}{2}\eta_2^4}<\frac{144}{\eta_1}+\frac{203\eta_2^2}{\eta_1^2},
\end{align*}
\begin{align*}R_{1,1,2}^{M_{112},D}&=\E\bigg|Q_1Q_2\frac{n|\hat{\eta}_2|}{\hat{\eta}_1^3}\bigg|\leq n\sqrt{\E[Q_1^2Q_2^2]\E\bigg[\frac{\hat{\eta}_2^{2}}{\hat{\eta}_1^6}\bigg]}<\frac{1}{\eta_1^3}\sqrt{\eta_1^2(6400\eta_1+9023\eta_2^2)\cdot\frac{1}{\eta_1^6}(\eta_1+2\eta_2^2)}\\
&\leq\frac{1}{\eta_1^3}\sqrt{\frac{34623}{2}\eta_1^2+\frac{57915}{2}\eta_2^4}<\frac{132}{\eta_1}+\frac{171\eta_2^2}{\eta_1^2}.
\end{align*}

\vspace{2mm}

\noindent{$R_{2,1,2}^{M_{212}}$}: 
\begin{align*}R_{2,1,2}^{M_{212},A}=R_{2,1,2}^{M_{212},C}&=\E\bigg[Q_1Q_2\frac{n}{2\eta_1^2}\bigg]\leq\frac{n}{2\eta_1^2}\sqrt{\E[Q_1^2]\E[Q_2^2]}\\
&<\frac{n}{2\eta_1^2}\sqrt{10\eta_1^2(6\eta_1+10\eta_2^2)}<\frac{4}{\sqrt{\eta_1}}+\frac{5|\eta_2|}{\eta_1},
\end{align*}
\begin{align*}R_{2,1,2}^{M_{212},B}=R_{2,1,2}^{M_{212},D}&=\E\bigg[Q_1Q_2\frac{n}{2\hat{\eta}_1^2}\bigg]\leq\frac{n}{2}\sqrt{\E[Q_1^2Q_2^2]\E[\hat{\eta}_1^{-4}]}\\
&<\frac{n}{2}\sqrt{\eta_1^2(6400\eta_1+9023\eta_2^2)\times\frac{2}{\eta_1^4}}<\frac{57}{\sqrt{\eta_1}}+\frac{68|\eta_2|}{\eta_1}.
\end{align*}

\vspace{2mm}

\noindent{$R_{1,2,2}^{M_{122}}$}: 
\begin{align*}R_{1,2,2}^{M_{122},A}=R_{1,2,2}^{M_{122},C}=\E\bigg[Q_2^2\frac{n}{2\eta_1^2}\bigg]<\frac{3}{\eta_1}+\frac{5\eta_2^2}{\eta_1^2},
\end{align*}
\begin{align*}R_{1,2,2}^{M_{122},B}=R_{1,2,2}^{M_{122},D}&=\E\bigg[Q_2^2\frac{n}{2\hat{\eta}_1^2}\bigg]\leq\frac{n}{2}\sqrt{\E[Q_2^4]\E[\hat{\eta}_1^{-4}]}\\
&<\frac{1}{2}\sqrt{(5886\eta_1^2+11700\eta_2^4)\cdot\frac{2}{\eta_1^4}}<\frac{55}{\eta_1}+\frac{77\eta_2^2}{\eta_1^2}.
\end{align*}

\subsection{Proof of Theorem \ref{thmnondiag}}\label{appnorgen}

\begin{proof}Let $\widetilde{\boldsymbol{W}}=\sqrt{n}(\hat{\boldsymbol{\theta}}_n(\boldsymbol{X})-\boldsymbol{\theta}_0)$, so that $\boldsymbol{W}=[I(\boldsymbol{\theta}_0)]^{1/2}\widetilde{\boldsymbol{W}}$.  Now, for $1\leq i\leq n$, write $\boldsymbol{X}_i=(X_{i,1},\ldots,X_{i,p})^{\intercal}$, and define the centered random variables $Y_{i,j}=X_{i,j}-\mu_j$, $1\leq i\leq n$, $1\leq j\leq p$.  For $1\leq j\leq p$, let $\bar{X}_j$ and $\bar{Y}_j$ denote the sample means of $X_{1,j},\ldots,X_{n,j}$ and $Y_{1,j},\ldots,Y_{n,j}$. A simple calculation gives the useful equation
\[\sum_{i=1}^n(X_{i,j}-\bar{X}_j)^2=\sum_{i=1}^n(X_{i,j}-\mu_{j})^2-n(\bar{X}_j-\mu_j)^2.\]
Putting all this together gives that $\widetilde{\boldsymbol{W}}$ can be written as $\widetilde{\boldsymbol{W}}=(\widetilde{W}_{1},\ldots,\widetilde{W}_{p(p+3)/2})^{\intercal}$, where, for $1\leq j\leq p$,
\begin{align*}\widetilde{W}_{j}=\frac{1}{\sqrt{n}}\sum_{i=1}^n X_{i,j}-\mu_j=\frac{1}{\sqrt{n}}\sum_{i=1}^n Y_{i,j},
\end{align*}
and, for $p+1\leq j\leq p(p+3)/2$, we associate $\widetilde{W}_j$ with an ordering of the random variables $\widetilde{W}_{k,\ell}$ which are given, for $1\leq \ell\leq k\leq p$, by
\begin{align*}\widetilde{W}_{k,\ell}&=\frac{1}{\sqrt{n}}\sum_{i=1}^n((X_{i,k}-\mu_k)(X_{i,\ell}-\mu_\ell)-\sigma_{k,\ell})-\sqrt{n}(\bar{X}_k-\mu_k)(\bar{X}_\ell-\mu_\ell) \\
&=\frac{1}{\sqrt{n}}\sum_{i=1}^n(Y_{i,k}Y_{i,\ell}-\sigma_{k,\ell})-\sqrt{n}\bar{Y}_{k}\bar{Y}_\ell.
\end{align*}
Now define $\widetilde{\boldsymbol{V}}=(\widetilde{V}_{1},\ldots,\widetilde{V}_{p(p+3)/2})^{\intercal}$, where, for $1\leq j\leq p$ and $1\leq \ell\leq k\leq p$, 
\[\widetilde{V}_{j}=\widetilde{W}_{j} \quad \text{and} \quad \widetilde{V}_{k,\ell}=\frac{1}{\sqrt{n}}\sum_{i=1}^n(Y_{i,k}Y_{i,\ell}-\sigma_{k,\ell}),\]
(here we associate $\widetilde{V}_j$, $p+1\leq j\leq p(p+3)/2$, with an ordering of $\widetilde{V}_{k,\ell}$, $1\leq\ell\leq k\leq p$) and let $\boldsymbol{V}=[I(\boldsymbol{\theta}_0)]^{1/2}\widetilde{\boldsymbol{V}}$.

Let $h\in\mathcal{H}_{\mathrm{W}}$.  Then
\begin{align}\mathbb{E}[h(\boldsymbol{W})]-\mathbb{E}[h(\boldsymbol{Z})]&= \big(\mathbb{E}[h(\boldsymbol{W})]-\mathbb{E}[h(\boldsymbol{V})]\big)+\big(\mathbb{E}[h(\boldsymbol{V})]-\mathbb{E}[h(\boldsymbol{Z})]\big)\nonumber\\
\label{r1r2}&=:R_1+R_2.
\end{align}
Now write $[I(\boldsymbol{\theta}_0)]^{1/2}=(a_{i,j})$. The remainder $R_1$ is readily bounded by applying the mean value theorem:
\begin{align*}|R_1|&\leq \|h\|_{\mathrm{Lip}}\mathbb{E}\bigg|\sum_{1\leq k\leq j\leq p}\sum_{1\leq r\leq q\leq p}a_{(j,k),(q,r)}\sqrt{n}\bar{Y}_q\bar{Y}_r\bigg|\\
&\leq\sum_{1\leq k\leq j\leq p}\sum_{1\leq r\leq q\leq p}\sqrt{n}\|[I(\boldsymbol{\theta}_0)]^{1/2}\|_{\mathrm{max}}\max_{1\leq t\leq p}\mathbb{E}[(\bar{Y}_t)^2]\\
&<\frac{p^4\sigma_*^2\|[I(\boldsymbol{\theta}_0)]^{1/2}\|_{\mathrm{max}}}{\sqrt{n}},
\end{align*}
where in the second step we used the triangle inequality and the Cauchy-Schwarz inequality, and that $\|h\|_{\mathrm{Lip}}\leq1$, since $h\in\mathcal{H}_{\mathrm{W}}$.

Now we bound $R_2$.  We can write $\boldsymbol{V}=\frac{1}{\sqrt{n}}\sum_{i=1}^n \boldsymbol{\xi}_i$, where $\boldsymbol{\xi}_1,\ldots,\boldsymbol{\xi}_n$ are i.i.d$.$ random vectors, and $\widetilde{\boldsymbol{V}}=\frac{1}{\sqrt{n}}\sum_{i=1}^n \tilde{\boldsymbol{\xi}}_i$, where $\tilde{\boldsymbol{\xi}}_1,\ldots,\tilde{\boldsymbol{\xi}}_n$ are i.i.d$.$ random vectors with $\boldsymbol{\xi}_i=[I(\boldsymbol{\theta}_0)]^{1/2}\tilde{\boldsymbol{\xi}}_i$.  Here the components of $\tilde{\boldsymbol{\xi}}_1$ are given by $\tilde{\boldsymbol{\xi}}_{1,j}=Y_{1,j}$, $1\leq j\leq p$, and $\tilde{\boldsymbol{\xi}}_{1,(k,\ell)}=Y_{1,k}Y_{1,\ell}-\sigma_{k,\ell}$, $1\leq \ell\leq k\leq p$, where, for $d+1\leq j\leq p(p+3)/2$ we associate $\widetilde{\boldsymbol{\xi}}_{1,j}$ with an ordering of $\widetilde{\boldsymbol{\xi}}_{1,(k,\ell)}$, $1\leq \ell\leq k\leq p$.
We begin by showing that the assumptions of Theorem \ref{bonisthm} are met, that is $\mathbb{E}[\boldsymbol{\xi}_1] = 0$ and $\mathbb{E}[\boldsymbol{\xi}_1\boldsymbol{\xi}_1^{\intercal}] =I_{p(p+3)/2}$.  The components of $\tilde{\boldsymbol{\xi}}_1$ are given by $\tilde{\boldsymbol{\xi}}_{1,j}=Y_{1,j}$ and $\tilde{\boldsymbol{\xi}}_{(k,\ell),1}=Y_{1,k}Y_{1,\ell}-\sigma_{k,\ell}$.  We can immediately see that $\mathbb{E}[\boldsymbol{\xi}_1]=[I(\boldsymbol{\theta}_0)]^{1/2}\mathbb{E}[\tilde{\boldsymbol{\xi}}_1]= \boldsymbol{0}$.  Let us now show that $\mathbb{E}[\boldsymbol{\xi}_1\boldsymbol{\xi}_1^{\intercal}] =I_{p(p+3)/2}$.  As the MLE is asymptotically multivariate normally distributed we have that $\boldsymbol{W} \stackrel{d}{\rightarrow}\boldsymbol{Z}$, as $n\rightarrow\infty$ (with an abuse of notation, as we have not indexed $\boldsymbol{W}$ with $n$).  We have just shown that $R_1\rightarrow0$, as $n\rightarrow\infty$, (again with the same abuse of notation) for all $h\in\mathcal{H}_1$. Therefore by (\ref{r1r2}) we have that $\boldsymbol{V} \stackrel{d}{\rightarrow}\boldsymbol{Z}$, as $n\rightarrow\infty$.  Therefore $\mathbb{E}[\boldsymbol{V}\boldsymbol{V}^{\intercal}] =I_{p(p+3)/2}+o(1)$, as $n\rightarrow\infty$.  But since $\boldsymbol{\xi}_1,\ldots,\boldsymbol{\xi}_n$ are i.i.d$.$ we have that $\mathbb{E}[\boldsymbol{\xi}_1\boldsymbol{\xi}_1^{\intercal}]=\mathbb{E}[\boldsymbol{V}\boldsymbol{V}^{\intercal}]$.  Since $\mathbb{E}[\boldsymbol{\xi}_1\boldsymbol{\xi}_1^{\intercal}]$ does not involve $n$, we deduce that $\mathbb{E}[\boldsymbol{\xi}_1\boldsymbol{\xi}_1^{\intercal}]=I_{p(p+3)/2}$.  

Now we obtain the bound
\begin{align*}\mathbb{E}[([I(\boldsymbol{\theta}_0)]^{1/2}\tilde{\xi}_{1,j})^4]&=\mathbb{E}\bigg[\bigg(\sum_{q=1}^{p(p+3)/2}a_{j,q}\tilde{\xi}_{1,q}\bigg)^4\bigg] \\
&\leq \frac{p^4(p+3)^4}{16}\|[I(\boldsymbol{\theta}_0)]^{1/2}\|_{\mathrm{max}}^4\cdot\mathrm{max}_{1\leq t\leq p(p+3)/2}\mathbb{E}[\tilde{\xi}_{1,t}^4] \\
&=\frac{p^4(p+3)^4}{16}\|[I(\boldsymbol{\theta}_0)]^{1/2}\|_{\mathrm{max}}^4\cdot105\sigma_*^8.
\end{align*}
As the assumptions of Theorem \ref{bonisthm} are satisfied, we may apply inequality (\ref{bonisbound2}) to obtain the bound
\begin{align*}R_2&\leq\frac{14 (p(p+3)/2)^{5/4}}{\sqrt{n}}\bigg(\frac{p^4(p+3)^4}{16}\|[I(\boldsymbol{\theta}_0)]^{1/2}\|_{\mathrm{max}}^4\cdot105\sigma_*^8\bigg)^{1/2}\\
&< \frac{15.1}{\sqrt{n}}p^{13/4}(p+3)^{13/4}\sigma_*^4\|[I(\boldsymbol{\theta}_0)]^{1/2}\|_{\mathrm{max}}^2.
\end{align*}
Finally, combining the bounds for $R_1$ and $R_2$ gives the bound for $d_{\mathrm{W}}(\boldsymbol{W},\boldsymbol{Z})$ as stated in the theorem.
\end{proof}
\subsection*{Acknowledgements}
 AA would like to thank the Department of Mathematics, The University of Manchester for the kind hospitality, where work on this project began.  RG is supported by a Dame Kathleen Ollerenshaw Research Fellowship.  We are very grateful to Thomas Bonis for valuable discussions concerning the results from his paper \cite{bonis} and for working out an explicit bound on the constant in one of the main quantitative limit theorems from his paper that we used in our paper. \gr{We would like to thank the referees for their helpful comments and suggestions that have enabled us to greatly improve our paper. In particular, we are very grateful to one of the referees  for their insightful comments and explanations, which enabled us to obtain $p$-Wasserstein analogues of the 1-Wasserstein distance bounds given in the original submission.}

\end{document}